\documentclass{amsart}

\usepackage{graphicx}
\usepackage{tikz,tikz-cd}
\usepackage{comment}
\usepackage{a4wide}
\usepackage[colorlinks=true]{hyperref}
\usepackage[capitalize,nameinlink]{cleveref}

\newtheorem{thm}{Theorem}[section]
\newtheorem{lemma}[thm]{Lemma}
\newtheorem{prop}[thm]{Proposition}
\newtheorem{cor}[thm]{Corollary}

\theoremstyle{definition}

\newtheorem{defn}[thm]{Definition}
\newtheorem{remark}[thm]{Remark}

\newtheorem{example}[thm]{Example}
\newtheorem{algorithm}[thm]{Algorithm}

\crefname{lemma}{Lemma}{Lemmas}
\crefname{defn}{Definition}{Definitions}
\crefname{lem}{Lemma}{Lemmas}
\crefname{thm}{Theorem}{Theorems}
\crefname{prop}{Proposition}{Propositions}
\crefname{question}{Question}{Questions}
\crefname{defn}{Definition}{Definitions}
\crefname{conj}{Conjecture}{Conjectures}
\crefname{figure}{Figure}{Figures}
\crefname{cor}{Corollary}{Corollaries} 

\AddToHook{env/lemma/begin}{\crefalias{thm}{lemma}}
\AddToHook{env/prop/begin}{\crefalias{thm}{prop}}
\AddToHook{env/cor/begin}{\crefalias{thm}{cor}}
\AddToHook{env/example/begin}{\crefalias{thm}{example}}
\AddToHook{env/defn/begin}{\crefalias{thm}{defn}}
\AddToHook{env/conj/begin}{\crefalias{thm}{conj}}
\AddToHook{env/question/begin}{\crefalias{thm}{question}}
\AddToHook{env/problem/begin}{\crefalias{thm}{problem}}
\AddToHook{env/remark/begin}{\crefalias{thm}{remark}}

\numberwithin{equation}{section}
\renewcommand\flat{\operatorname{flat}}
\newcommand\SP{\operatorname{SP}}
\newcommand\MLH{\operatorname{MLH}}
\newcommand\sinv{\operatorname{sinv}}
\newcommand\sh{\operatorname{sh}}

\newcommand\tile{\operatorname{tile}}

\DeclareMathOperator{\iDes}{iDes}

\DeclareMathOperator{\wt}{wt}
\DeclareMathOperator{\w}{w}

\DeclareMathOperator{\heads}{{heads}}

\DeclareMathOperator{\wex}{{wex}}

\DeclareMathOperator{\fcol}{{free-col}}
\DeclareMathOperator{\frow}{{free-row}}
\DeclareMathOperator{\topup}{top-up}
\DeclareMathOperator{\row}{{row}}
\DeclareMathOperator{\stat}{{stat}}

\DeclareMathOperator{\fcell}{{free-cell}}

\DeclareMathOperator{\LRmax}{{LRmax}}

\DeclareMathOperator{\RLmax}{{RLmax}}
\DeclareMathOperator{\RLmin}{{RLmin}}

\DeclareMathOperator{\diag}{diag}
\DeclareMathOperator{\inv}{{inv}}

\DeclareMathOperator{\dd}{{2}}
\DeclareMathOperator{\ee}{{0}}

\DeclareMathOperator{\A}{\mathcal{A}}
\DeclareMathOperator{\Z}{\mathcal{Z}}
\DeclareMathOperator{\Y}{\mathcal{Y}}
\DeclareMathOperator{\AS}{AS}

\DeclareMathOperator{\RAT}{RAT}
\DeclareMathOperator{\AT}{AT}

\newcommand\cro{\operatorname{crossing}}
\renewcommand\neg{\operatorname{neg}}

\newcommand\zrpt{\zeta}

\newcommand\ferat{\Phi_{\operatorname{FE}}}

\newcommand\biletter[2]{\genfrac{}{}{0pt}{}{#1}{#2}}

\DeclareMathOperator{\str}{Straight}
\DeclareMathOperator{\spl}{Split}

\usetikzlibrary {arrows.meta}
\usetikzlibrary{snakes}
\tikzset{every picture/.style={x=1em, y=1em}}

\newcommand\vertex[2]{\node[fill=black, inner sep=1pt, circle] (#1) at (#1,0) {};
\node at (#1,-.5) {\small \( #2 \)};}

\newcommand\Rvertex[2]{\node[fill=red, inner sep=2pt, circle] (#1) at (#1,0) {};
\node at (#1,-.5) {\small \( #2 \)};}
\newcommand\Uvertex[2]{\node[fill=black, inner sep=1pt, circle] (#1) at (#1,0) {};
\node at (#1,.5) {\small \( #2 \)};}

\newcommand\upperBrace[3]{\draw[snake=brace] (#1,#2)--(#1+1,#2); \node[above] at (#1+0.5,#2+0.2) {#3};}
\newcommand\lowerBrace[3]{\draw[snake=brace] (#1+1,#2)--(#1,#2); \node[below] at (#1+0.5,#2-0.2) {#3};}

\newcommand\shortUpperHalfEdge[1]{\draw[thick,blue] (#1,0) to (#1+.5,.5);}
\newcommand\shortLowerHalfEdge[1]{\draw[thick,blue] (#1,0) to (#1+.5,-.5);}
\newcommand\longUpperHalfEdge[1]{\draw[thick,blue] (#1,0) to (#1+1.5,1.5);}
\newcommand\longRedUpperHalfEdge[1]{\draw[thick,red] (#1,0) to (#1+1.5,1.5);}
\newcommand\longLowerHalfEdge[1]{\draw[thick,blue] (#1,0) to (#1+1.5,-1.5);}
\newcommand\upperHalfEdge[1]{\draw[thick,red] (#1,0) to (#1+.5,0.5);}
\newcommand\lowerHalfEdge[1]{\draw[thick,red] (#1,0) to (#1+.5,-0.5);}
\newcommand\upperEdge[2]{\draw[out=60,in=120] (#1,0) to (#2,0);}
\newcommand\lowerEdge[2]{\draw[out=-120,in=-60] (#1,0) to (#2,0);}
\newcommand\oneSpiralEdge[3]{\draw[blue] (0.5,0) arc [start angle=180, end angle=360, radius=0.25]; \draw[out=90,in=120,blue] (#2,0) to (#3,0);}
\newcommand\spiralEdge[3]{\draw[out=-120,in=-90,blue] (#1,0) to (#2,0); \draw[out=90,in=120,blue] (#2,0) to (#3,0);}
\newcommand\upperLoop[1]{\draw (#1,0.15) circle [radius=0.15];}

\newcommand\cell[2]{\draw (#1,#2) --++(1,0) --++(0,1) --++(-1,0) --++(0,-1) -- cycle;}
\newcommand\cells[2]{\draw [fill=cyan!20!white,draw=black] (#1,#2) --++(1,0) --++(0,1) --++(-1,0) --++(0,-1) -- cycle;}
\newcommand\cellt[2]{\draw [fill=gray!30!white,draw=black] (#1,#2) --++(1,0) --++(0,1) --++(-1,0) --++(0,-1) -- cycle;}
\newcommand\tall[2]{\draw [fill=gray!30!white,draw=black] (#1,#2) --++(1,1) --++(0,1) --++(-1,-1) --++(0,-1) -- cycle;}
\newcommand\short[2]{\draw [fill=cyan!20!white,draw=black](#1,#2) --++(1,0) --++(1,1) --++(-1,0) --++(-1,-1) -- cycle;}
\newcommand\upArr[2]{\draw [blue,arrows = {-Stealth[scale=.8]}, thick] (#1+.5,#2+.2) --++(0,0.6);}
\newcommand\slantUpArr[2]{\draw [blue,arrows = {-Stealth[scale=1]}, thick] (#1+0.7,#2+.2) --++(0.6,0.6);}
\newcommand\slantLeftArr[2]{\draw [red,arrows = {-Stealth[scale=1]}, thick] (#1+0.8,#2+1.3) --++(-0.6,-0.6);}
\newcommand\leftArr[2]{\draw [red,arrows = {-Stealth[scale=.8]}, thick] (#1+0.8,#2+.5) --++(-0.6,0);}
\newcommand\cellDot[2]{\fill (#1+.5,#2+.5) circle [radius=1pt];}
\newcommand\shortDot[2]{\fill (#1+1,#2+.5) circle [radius=1pt];}
\newcommand\tallDot[2]{\fill (#1+.5,#2+1) circle [radius=1pt];}

\newcommand\US[2]{\draw (#1,#2) -- (#1+1,#2+1);}
\newcommand\DS[2]{\draw (#1,#2) -- (#1+1,#2-1);}
\newcommand\HS[2]{\draw (#1,#2) -- (#1+1,#2);}
\newcommand\dHS[2]{\draw[dashed,thick] (#1,#2) -- (#1+1,#2);}
\newcommand\UMark[3]{\node at (#1+.7,#2+.3) {\tiny \textbf{\textcolor{blue}{#3}}};}
\newcommand\DMark[3]{\node at (#1+.4,#2-.7) {\tiny \textbf{\textcolor{blue}{#3}}};}
\newcommand\HMark[3]{\node at (#1+.5,#2-.3) {\tiny \textbf{\textcolor{blue}{#3}}};}
\newcommand\HLabel[3]{\node at (#1+.5,#2+.4) {\small #3};}
\newcommand\DLabel[3]{\node at (#1+.8,#2-.2) {\small #3};}
\newcommand\ULabel[3]{\node at (#1+.4,#2+.7) {\small #3};}

\newcommand\Grid[2]{
  \draw[help lines,step=1] (0,0) grid (#1,#2);
  \foreach \x in {0,...,#1}
  \draw (\x,0) node[below] {\tiny \x};
  \foreach \y in {0,...,#2}
  \draw (0,\y) node[left] {\tiny \y};
}

\title{Bijections for rhombic alternative tableaux}

\author{Sylvie Corteel, Jang Soo Kim, Olya Mandelshtam, and Philippe Nadeau}

\address[S.~Corteel]{Department of Mathematics, University of California, Berkeley, CA, USA and CNRS, IMJ-PRG, Sorbonne Universit\'e, France}
\email{corteel@berkeley.edu}

\address[J.~S.~Kim]{Department of Mathematics, Sungkyunkwan University, Suwon 16419, Republic of Korea}
\email{jangsookim@skku.edu}

\address[O.~Mandelshtam]{University of Waterloo, Waterloo, Ontario, Canada}
\email{omandels@uwaterloo.ca}

\address[P.~Nadeau]{Universite Claude Bernard Lyon 1, CNRS, Ecole Centrale de Lyon, INSA Lyon, Universit\'e Jean Monnet, ICJ UMR5208, 69622 Villeurbanne, France}
\email{nadeau@math.univ-lyon1.fr}

\thanks{Kim is supported by the National Research Foundation of Korea
(NRF) grant RS-2025-00557835 funded by the Korea government. Mandelshtam is supported by the NSERC grant RGPIN-2021-02568 and was partially supported by the NSF grant DMS-1704874. Corteel and Nadeau were partially supported by French ANR grant ANR-19-CE48-0011 (COMBINÉ). Parts of this work were completed during visits of Corteel and Nadeau to Korea funded by French PHC STAR 2024 project number 50134 VJ.}

\begin{document}

\maketitle

\begin{abstract}
  We generalize well-known bijections between alternative tableaux and permutations to bijections between rhombic alternative tableaux (RAT) and assembl\'ees of permutations. We show how these various bijections are connected. As a consequence, we find a refined enumeration formula for RAT. One of our bijections carries many statistics from RAT to assembl\'{e}es; notably, it sends the number of free cells to the number of crossings, which answers a question of Mandelshtam and Viennot. We also find an $r!$-to-$1$ map from marked Laguerre histories to assembl\'{e}es, answering a question of Corteel and Nunge.
\end{abstract}

\section{Introduction}

\emph{Rhombic alternative tableaux} 
provide a combinatorial interpretation for the steady-state probabilities of the \emph{two-species partially asymmetric simple exclusion process} (PASEP). The PASEP, or more generally \emph{asymmetric simple exclusion process} (ASEP), is a classical interacting particle system on a finite one-dimensional lattice with open boundaries, where particles hop right and left with rates determined by fixed parameters. Originally introduced in the 1960s by biologists and mathematicians (see for example the survey papers \cite{blythe, chou}),  the ASEP has since received significant attention, in part due to its connections with Askey--Wilson polynomials \cite{Corteel2011} and, more generally, Koornwinder polynomials \cite{cantini16,cantini17,Corteel18}.

The single-species PASEP consists of indistinguishable particles on a lattice with \(n\) sites, with at most one particle per site. A particle may hop into an adjacent vacant site to its right (resp.~left) with rate 1 (resp.~\(q\)), and border parameters \(\alpha\) and \(\beta\) are the rates at which a particle may enter from the left and exit from the right of the lattice, respectively. The steady-state probabilities (given by the left-eigenvector of the transition matrix) may be computed through the \emph{matrix ansatz} 
\cite{Derrida1993} and combinatorial formulas were obtained by the first author and Williams in \cite{Corteel2007} with \emph{permutation tableaux}, later reformulated as \emph{alternative tableaux}
\cite{Viennot2013} and the closely related \emph{tree-like tableaux} \cite{Aval2013}. The connection to Askey--Wilson polynomials is realized by the ASEP (the 5-parameter generalization of the PASEP), in which particles may enter and exit at both boundaries of the lattice, which are modeled combinatorially with \emph{staircase tableaux}
\cite{Corteel2011}.  

The \emph{two-species PASEP} is a generalization with two types of particles, \emph{heavy} and \emph{light}, as depicted in \cref{fig:2-parameters}. Adjacent particles can swap with rate 1 (resp.~\(q\)) if the heavier particle is on the left (resp.~right). Only the heavy particles may enter or exit the lattice. The two-species PASEP admits an analogous matrix ansatz solution \cite{Uchiyama2008}. The combinatorial construction of alternative tableaux was extended by the third author and Viennot \cite{MV18} to the two-species setting by introducing a diagonal edge, yielding \emph{rhombic alternative tableaux}. This construction was extended to the 5-parameter setting with \emph{rhombic staircase tableaux} by the first and third authors together with Williams in \cite{CMW17}, to make a connection with Koornwinder polynomials.

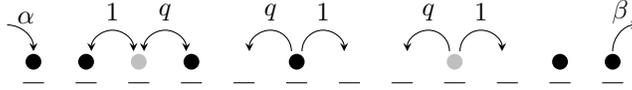
\begin{figure}
\centering
\begin{tikzpicture}[x=.7cm,y=.7cm,>=stealth]
\foreach \i in {1,...,12}{\draw (\i-0.2,-0.1) -- (\i+0.2,-0.1);}
\fill (1,0.3) circle (0.15);
\fill (2,0.3) circle (0.15);
\fill[lightgray] (3,0.3) circle (0.15);
\fill (4,0.3) circle (0.15);
\fill (6,0.3) circle (0.15);
\fill[lightgray] (9,0.3) circle (0.15);
\fill (11,0.3) circle (0.15);
\fill (12,0.3) circle (0.15);
\draw[<-] (1,.5) arc [radius=.5, start angle=0, end angle=90] node [midway, above] {\(\alpha\)};
\draw[<->] (2.9,.5) arc [radius=.4, start angle=0, end angle=180] node [midway, above] {\(1\)};
\draw[<->] (3.9,.5) arc [radius=.4, start angle=0, end angle=180] node [midway, above] {\(q\)};
\draw[->] (5.9,.5) arc [radius=.4, start angle=0, end angle=180] node [midway, above] {\(q\)};
\draw[<-] (6.9,.5) arc [radius=.4, start angle=0, end angle=180] node [midway, above] {\(1\)};
\draw[->] (8.9,.5) arc [radius=.4, start angle=0, end angle=180] node [midway, above] {\(q\)};
\draw[<-] (9.9,.5) arc [radius=.4, start angle=0, end angle=180] node [midway, above] {\(1\)};
\draw[->] (12,.5) arc [radius=.5, start angle=180, end angle=90] node [midway, above] {\(\beta\)};
\end{tikzpicture}
\caption{The parameters of the two-species PASEP.}
\label{fig:2-parameters}
\end{figure}

A rhombic alternative tableau (RAT) is a filling of a \emph{rhombic
  diagram} that generalizes an alternative tableau. Rhombic diagrams
of size \((n,r)\) are rhombic tilings of a certain closed shape on a
triangular lattice that correspond to states of the two-species PASEP
with \(n\) sites and \(r\) light particles. For \(r=0\), RAT coincide
with alternative tableaux. It was shown in \cite{MV18} that the partition function \( \Z_{n,r}(\alpha,\beta,q) \) for
the two-species PASEP is the generating function for RAT of size
\( (n,r) \) (with the relevant terms defined in \cref{sec:preliminaries}):
\[
  \Z_{n,r}(\alpha,\beta,q) = (\alpha\beta)^{n-r} \Y_{n,r}(\alpha^{-1}, \beta^{-1}, q),
\]
where
\begin{equation}\label{eq:Y}
  \Y_{n,r}(\alpha,\beta,q) = \sum_{T\in \RAT(n,r)} \alpha^{\frow(T)} \beta^{\fcol(T)} q^{\fcell(T)}.
\end{equation}
The enumeration of RAT of size \( (n,r) \) specialized at \(q=1\)
yields the following formula in \cite{MV18}:
\begin{equation}\label{Z_nr}
\Y_{n,r}(\alpha,\beta,1) = \binom{n}{r} (\alpha+\beta+r)_{n-r}\,,
\end{equation}
where \((x)_n=x (x+1) \cdots (x+n-1)\).

Alternative tableaux, permutation tableaux, and tree-like tableaux
have rich combinatorial structures: their generating functions are
given by product formulas, and they encode various permutation
statistics \cite{Steingrimsson2007} as well as moments of
orthogonal polynomials \cite{Medicis1994,Corteel2011a}. These
connections are realized through several bijections with permutations,
including a zigzag map on permutation tableaux
\cite{Steingrimsson2007} (equivalent to that of \cite{Postnikov}), the
fusion-exchange map \cite{Viennota}, a zigzag map on alternative
tableaux \cite{CorteelKim}, insertion algorithms
\cite{Corteel2009,Aval2013}, and a recursive decomposition
\cite{Nadeau}.

\emph{Assembl\'{e}es of permutations} are generalizations of
permutations: they consist of unordered blocks of disjoint permutations
whose union is \( \{ 1,\dots,n\}\). A bijection between RAT and
assembl\'{e}es, generalizing the \emph{fusion-exchange} bijection for
alternative tableaux of \cite{Viennota}, was constructed in
\cite{Mandelshtam2018}.
\smallskip

In this paper, we further develop the combinatorics of RAT, extend
the various bijections between alternative tableaux and permutations
to bijections between RAT and assembl\'{e}es, and describe the
relationship among the bijections. In particular, one of our bijections sends the
number of \emph{free cells} of a RAT to the number of \emph{crossings} of an
assembl\'{e}e, which we represent using arc diagrams by considering assembl\'ees as signed permutations. This extends the crossing statistic on permutations; see \cite{Corteel2011a} and \cite{Corteel2013}.
This provides an answer to the open problem in \cite[Section~5]{Mandelshtam2018}.

By specializing \cite[Corollary~6.2]{Corteel18}, one can obtain that
the partition function of the two-species PASEP, or equivalently the
generating function of RAT, is equal to a mixed moment of Al-Salam--Chihara polynomials. It can be written as the generating function for marked Laguerre
histories. Corteel and Nunge \cite{Corteel2020a} showed that a variant of this generating function has a factor of
\( (1+q)(1+q+q^2) \cdots (1+ q+ \cdots + q^{r-1}) \), and posed as an open question the problem of finding a bijective proof of this factorization. We answer this question by constructing a bijection between signed permutations and marked Laguerre histories, inspired by crossings and arc diagrams.

Let us now outline the content of the paper. In
\Cref{sec:preliminaries}, we provide basic definitions and some known
results on alternative tableaux. In \cref{sec:structure}, we
generalize the decomposition of alternative tableaux in \cite{Nadeau}
to understand the structure of RAT. From this perspective, RAT can
be decomposed into indecomposable pieces called \emph{packed} RAT.
Using this, we give a simple proof of \eqref{Z_nr}. In that section,
we also define a \emph{flattening map}, which embeds RAT into a
subclass of alternative tableaux. In \Cref{sec:bijections}, we
generalize the insertion map in \cite{Corteel2009} and the zigzag map
in \cite{CorteelKim} to RAT. Using the insertion map, we give a refined
formula of \eqref{Z_nr}. We also show that the insertion map, the
zigzag map, and the fusion-exchange map in \cite{Mandelshtam2018} are
all essentially the same up to simple transformations. In
\cref{sec:zigzag-RPT}, we introduce a fundamentally different
bijection between RAT and assembl\'{e}es that generalizes the zigzag
map of \cite{Steingrimsson2007} and maps free cells of a RAT to crossings of the corresponding assembl\'{e}e.
Finally, in \cref{sec:laguerre-histories}, we construct an
\(r!\)-to-\(1\) map from marked Laguerre histories to assembl\'{e}es, answering a question in
\cite[p.~15]{Corteel2020a}.

\section{Preliminaries}
\label{sec:preliminaries}

In this section we provide necessary definitions and recall some known results.

For integers \( i \) and \( j \), we define
\( [i,j] = \{i,i+1,\dots,j\} \) if \( i\le j \) and
\( [i,j] = \emptyset \) otherwise. For an integer \(n\geq 0\), we
denote \([n]=[1,n]=\{1,2,\ldots,n\}\).

Let \(v=(v_1,\ldots,v_n)\) be a sequence of integers.
An \emph{inversion} of \( v \) is a pair \( (v_i,v_j) \)
such that \( i<j \) and \( v_i>v_j \). We denote by \( \inv(v) \)
the number of inversions of \( v \).
We say that
\( v_i \) is a \emph{right-to-left-minimum} (or \emph{RL-minimum}) of
\( v \) if \( v_i<v_j \) for all \( j\in [i+1,n] \).  Similarly,
\( v_i \) is a \emph{right-to-left-maximum} (or \emph{RL-maximum}) of
\( v \) if \( v_i>v_j \) for all \( j\in [i+1,n] \). We denote by
\( \RLmin(v) \) (resp.~\( \RLmax(v) \)) the number of
RL-minima (resp.~RL-maxima) of \( v \).

We will sometimes identify the sequence \(v=(v_1,\ldots,v_n)\) as the
word \(v=v_1\cdots v_n\). Note that, by definition, \( v_n \) is always
an RL-minimum and also an RL-maximum.

\subsection{Permutations and assembl\'{e}es} 

For an integer \(n\geq 1\), denote the set of permutations of \([n]\)
by \(S_n\). We will consider a permutation \(\sigma\in S_n\) as a
bijection \( \sigma:[n]\to [n] \) and also as a word
\(\sigma=\sigma_1\sigma_2\cdots\sigma_n\) so that
\(\sigma_i=\sigma(i)\) denotes the entry at position \(i\). More
generally, for an ordered set of labels \(X\), denote by \(S_X\) the
set of permutations of \(X\). Hence, \(S_n=S_{[n]}\), and if
\(|X|=n\), then \(S_n\) and \(S_X\) are in bijection.

\begin{defn}
  For positive integers \(n\geq r\), an \emph{assembl\'ee} of size
  \((n,r)\) is a collection \( \{\sigma^{(1)},\dots,\sigma^{(r)}\} \)
  of \( r \) permutations \( \sigma^{(i)}\in S_{B_i} \) such that
  \( \{B_1,\dots,B_r\} \) is a set partition of \( [n] \). Each
  permutation \( \sigma^{(i)} \) is called a \emph{block} of the
  assembl\'{e}e. Denote the set of assembl\'ees of size \((n,r)\) by
  \(\A(n,r)\).
\end{defn}

For example, \(\pi= \{ 3529, 647,81\} \in\A(9,3)\) is an assembl\'{e}e
with \( 3 \) blocks. Note that \(\A(n,1)\) is naturally identified
with \(S_n\). Using a counting argument, we have
\[
  |\A(n,r)| = \binom{n-1}{r-1} \frac{n!}{r!},
\]
which is known as the \emph{Lah number}. 

It is convenient to impose a canonical ordering on the blocks of an assembl\'ee, so that it can be
considered as a permutation together with a set of markers delineating
the blocks.

\begin{defn}\label{def:f}
  For \(\pi\in\A(n,r)\), the \emph{head} of a block \( B \) in
  \( \pi \) is its first element, and the set of heads of $\pi$ is denoted $\heads(\pi)$. Define the \emph{canonical order}
  of \(\pi\) to be the arrangement of blocks such that their heads are
  in increasing order. Define the \emph{canonical permutation} \(\sigma=f(\pi)\) to be the
  permutation \(\sigma\in S_n\) obtained by forgetting the
  delineations of the blocks in \(\pi\) when it is in canonical order. 
\end{defn}

When writing an assembl\'ee in canonical order, we use brackets to
delineate the blocks. By definition, \(\pi\) can be recovered uniquely
from \(f(\pi)\) and \(\heads(\pi)\). For example, the canonical order
of the assembl\'ee \(\pi= \{ 3529, 647,81\} \) is
\([3\, 5\, 2\, 9]\, [6\, 4\, 7]\, [8\, 1]\), and we have
\(\heads(\pi)=\{3, 6, 8\}\) and
$f(\pi)=3\,5\,2\,9\,6\,4\,7\,8\,1\in S_9$.

\subsection{Alternative tableaux} 

For a word \(w\in\{\dd,\ee\}^n\), the \emph{diagram} \( \Gamma_w \) of
\( w \) is obtained as follows. Draw a lattice path, called the
\emph{southeast border}, by reading \(w\) from
left to right and drawing a step south for a \(\dd\) and a step west
for a \(\ee\). From the same starting point, draw a lattice path,
called the \emph{northwest border}, consisting of \( a \) steps
east and \( b \) steps south, where \( a \) and \( b \) are the
numbers of \( 0 \)'s and \( 2 \)'s in \( w \), respectively. Then tile
the region between the two paths with non-overlapping squares. Note that the
number of squares in the \(i\)th row (from top to bottom) of
\( \Gamma_w \) is equal to the number of \(\ee\)'s in the subword of
\( w \) following the \(i\)th \(\dd\).

\begin{defn}\label{def:AT}
  An \emph{alternative tableau} of \emph{shape}
  \( w\in\{\dd,\ee\}^n \) is a filling of the squares in \(\Gamma_w\)
  with \emph{up-arrows} and \emph{left-arrows} such that there is no
  arrow pointing to any other. We say an up-arrow (resp.~a left-arrow)
  \emph{points to} a square if the square lies south (resp.~east) of
  the arrow and in the same column (resp.~row); see \Cref{fig:AT}. The
  \emph{size} of this alternative tableau is defined to be \( n \). We
  denote by \( \AT(n) \) the set of alternative tableaux of size
  \( n \).
\end{defn}

\begin{figure}
  \centering
  \begin{tikzpicture}[scale=1.5]
  \cell{3}{4} \cell{2}{4} \cell{1}{4} \cell{0}{4} \cell{3}{3} \cell{2}{3} \cell{1}{3} \cell{0}{3} \cell{3}{2} \cell{2}{2} \cell{1}{2} \cell{0}{2} \cell{2}{1} \cell{1}{1} \cell{0}{1} \cell{1}{0} \cell{0}{0}
 \upArr{1}{1} \leftArr{2}{3} \upArr{3}{4} \leftArr{0}{1} \leftArr{0}{0}
 \draw (4,5) -- (6,5); \draw (0,0) -- (0,-1);
\end{tikzpicture} \qquad \qquad \qquad \qquad 
\begin{tikzpicture}[scale=1.5]
  \cell{3}{4} \cell{2}{4} \cell{1}{4} \cell{0}{4} \cell{3}{3} \cell{2}{3} \cell{1}{3} \cell{0}{3} \cell{3}{2} \cell{2}{2} \cell{1}{2} \cell{0}{2} \cell{2}{1} \cell{1}{1} \cell{0}{1} \cell{1}{0} \cell{0}{0}
  \cell{0}{5} \cell{1}{5} \cell{2}{5} \cell{3}{5} \cell{4}{5} \cell{5}{5}
 \upArr{1}{1} \leftArr{2}{3} \upArr{3}{4} \leftArr{0}{1} \leftArr{0}{0}
 \draw (4,5) -- (6,5); \draw (0,0) -- (0,-1);
 \upArr{0}{5} \upArr{2}{5} \upArr{4}{5} \upArr{5}{5}
\end{tikzpicture}
\caption{An alternative tableau \( T\in \AT(12) \) of shape \( w=002220202002 \) on the left and the
  corresponding extended alternative tableau on the right.}
  \label{fig:AT}
\end{figure}
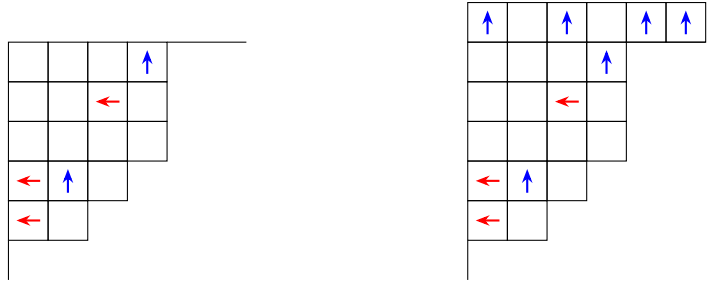

Let \( T\in \AT(n) \). A \emph{free cell} of \( T \) is a cell without
an arrow such that there is no arrow pointing to this cell. A
\emph{free row} is a row that does not contain any left-arrows. A
\emph{free column} is a column that does not contain any up-arrows.

It will be convenient for us to work with \emph{extended alternative tableaux}, which are alternative tableaux \(T\)
such that the shape of \( T \) begins with a 2 and every column of
\(T\) contains an up-arrow. We denote by \( \AT^+(n) \) the set of
extended alternative tableaux in \( \AT(n) \).

There is a simple bijection between \( \AT(n) \) and \( \AT^+(n+1) \).
For \(T\in\AT(n)\), the corresponding extended alternative tableau is
constructed as follows. First, add a row to the top of \( T \).
For every column without an up-arrow in \(T\), place an up-arrow in
its topmost cell. See \Cref{fig:AT} for an example. 

\begin{remark}
  Extended alternative tableaux can be identified with
  \emph{permutation tableaux} \cite{Steingrimsson2007}. Given an extended alternative tableau,
  if we replace every up-arrow and every free cell by a \( 1 \) and
  every left-arrow and every non-free cell by a \( 0 \), we obtain a
  permutation tableau.
\end{remark}

There are several known bijections between \( \AT^+(n+1) \) (or
\( \AT(n) \)) and \( S_{n+1} \). We review a zigzag map and an
insertion map, which will be useful for studying RAT later.

From now on, we label the southeast border edges of \(T\in \AT^+(n+1)\) with
\( 1,2,\dots,n+1 \). The northwest border edges are also labeled to match
the labels of the corresponding rows or columns.

\begin{defn}[Zigzag map] \label{def:zigzag AT} Let
  \(T\in \AT^+(n+1)\). A \emph{zigzag path} on \(T\) is a path that
  travels southeast along the rows and columns of \(T\), and makes a
  turn at every cell containing an arrow. Define \( \Phi_Z^{\AT}(T) \) to be
  the permutation \( \pi\in S_{n+1} \) such that for each
  \(i\in [n+1]\), the zigzag path from \(i\) on the northwest border
  reaches \( \pi(i) \) on the southeast border.
\end{defn}

For example, if \( T \) is the alternative tableau in
\Cref{fig:AT-zigzag}, then, in two-line notation, we have
\begin{equation}\label{eq:5}
  \Phi_Z^{\AT}(T) =
  \begin{pmatrix}
1 & 2 & 3 & 4 & 5 & 6 & 7 & 8 & 9 & 10\\
8 & 4 & 3 & 1 & 9 & 7 & 6 & 5 & 2 & 10\\
\end{pmatrix}.
\end{equation}

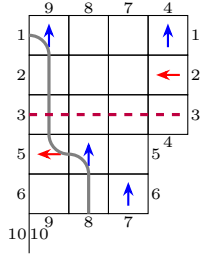
\begin{figure}
  \centering
  \begin{tikzpicture}[scale=1.5]
    \cell{3}{4} \cell{2}{4} \cell{1}{4} \cell{0}{4} \cell{3}{3} \cell{2}{3} \cell{1}{3} \cell{0}{3} \cell{3}{2} \cell{2}{2} \cell{1}{2} \cell{0}{2} \cell{2}{1} \cell{1}{1} \cell{0}{1} \cell{2}{0} \cell{1}{0} \cell{0}{0} \node at (0.500000000000000, -0.200000000000000) {\tiny 9};
\node at (1.50000000000000, -0.200000000000000) {\tiny 8};
\node at (2.50000000000000, -0.200000000000000) {\tiny 7};
\node at (3.20000000000000, 0.500000000000000) {\tiny 6};
\node at (3.20000000000000, 1.50000000000000) {\tiny 5};
\node at (3.50000000000000, 1.80000000000000) {\tiny 4};
\node at (4.20000000000000, 2.50000000000000) {\tiny 3};
\node at (4.20000000000000, 3.50000000000000) {\tiny 2};
\node at (4.20000000000000, 4.50000000000000) {\tiny 1};
 \upArr{0}{4} \upArr{1}{1} \upArr{3}{4} \leftArr{0}{1} \upArr{2}{0} \leftArr{3}{3}
 \draw (0,0) --(0,-1);
\node at (.25,-.5) {\tiny 10};
\node at (-.3,-.5) {\tiny 10};
\node at (.5,5.2) {\tiny 9};
\node at (1.5,5.2) {\tiny 8};
\node at (2.5,5.2) {\tiny 7};
\node at (3.5,5.2) {\tiny 4};
\node at (-.2,4.5) {\tiny 1};
\node at (-.2,3.5) {\tiny 2};
\node at (-.2,2.5) {\tiny 3};
\node at (-.2,1.5) {\tiny 5};
\node at (-.2,.5) {\tiny 6};
\draw[very thick,gray] (0,4.5) to[out=0, in=90] (.5,4) to (.5,2) to[out=270,in=180] (1,1.5) to[out=0, in=90] (1.5,1) to (1.5,0);
  \draw[very thick,purple,dashed] (0,2.5) to (4,2.5);
\end{tikzpicture}
\caption{The zigzag path from \( 1 \) ends at \( 8 \), and the zigzag
  path from \( 3 \) ends at \( 3 \).}
  \label{fig:AT-zigzag}
\end{figure}

\begin{defn}[Insertion map] \label{def:insertion AT} Let
  \(T\in \AT^+(n+1)\). Suppose \( c_1 > \cdots > c_m \) are the labels
  of the columns in \( T \). Then \( \Phi^{\AT}_I(T)=\pi \) is defined as
  follows.
  \begin{enumerate}
  \item Set \( \pi \) to be the word of the indices of the free rows
    of \( T \) in increasing order.
  \item For \(i=1,\ldots,m\), we insert integers to \( \pi \) as
    follows. Let \( j \) be the label of the row containing the
    up-arrow in column \( c_i \). Let \(r_1<\cdots<r_k\) be the labels
    of the rows containing a left-arrow in column \( c_i \). Insert
    \(r_1,\ldots, r_k,c_i\) to the left of \(j\) in \( \pi \).
  \end{enumerate}
\end{defn}

For example, if \( T \) is the alternative tableau in
\Cref{fig:AT-zigzag}, then
\begin{equation}\label{eq:6}
  \Phi^{\AT}_I(T) = 8 \,\, 5 \,\, 9 \,\, 2 \,\, 4 \,\, 1 \,\, 3 \,\, 7 \,\, 6 \,\, 10.
\end{equation}
Corteel and Nadeau showed the following theorem.

\begin{thm} \cite[Section~3]{Corteel2009}
  \label{thm:CN} 
  The insertion map \( \Phi^{\AT}_I:\AT^+(n+1)\to S_{n+1} \) is a
  bijection. Moreover, if \( T = \Phi^{\AT}_I(\pi) \), then for each
  \( i\in[n+1] \), the integer \( i \) is the label of a free row of \( T \) if and only
  if \( i \) is an RL-minimum of \( \pi \).
\end{thm}

The following map, known as the \emph{Foata map}, will be useful in what follows.

\begin{defn}\label{def:phi}
  For \( \pi\in S_n \), we define \( \phi(\pi) \) as follows. Let
  \( a_1 < \cdots < a_t=n \) be the indices such that
  \(\pi_{a_1},\ldots,\pi_{a_t}\) are the RL-minima of \(\pi\). Then,
  \( \phi(\pi) \) is the permutation whose cycle decomposition is given by
\[
  \phi(\pi) = (\pi_1,\dots,\pi_{a_1}) (\pi_{a_1+1},\dots,\pi_{a_2})
  \cdots (\pi_{a_{t-1}+1},\dots,\pi_{a_t}).
\]
\end{defn}

Corteel and Kim showed that the insertion map \( \Phi^{\AT}_I \) and the zigzag map \( \Phi^{\AT}_Z \) are equivalent up to the Foata transformation. 

\begin{thm}{{\cite[Section~5]{CorteelKim}}}\label{eq:CK}
For an alternative tableau $T$,
\[
  \Phi_Z^{\AT}(T) = \phi(\Phi^{\AT}_I(T))\,.
  \]
\end{thm}
For example, applying \( \phi \) to \eqref{eq:6} yields
\[
  \phi(\Phi^{\AT}_I(T)) = (8,5,9,2,4,1) (3) (7,6) (10)\,,
\]
which agrees with the permutation in \eqref{eq:5}.

\subsection{Rhombic alternative tableaux}

In this subsection, we define rhombic alternative tableaux, which were introduced in \cite{MV18} to
generalize alternative tableaux using the following three kinds of
tiles:
\[ \text{\emph{squares} }
\begin{tikzpicture}[scale=0.75]
  \cell{1}{1}
\end{tikzpicture}\,,\qquad\text{\emph{tall rhombi }}
\begin{tikzpicture}[scale=0.75]
  \tall{1}{1}
\end{tikzpicture}\,,\qquad\text{and \emph{short rhombi} }
\begin{tikzpicture}[scale=0.75]
  \short{1}{1}
\end{tikzpicture}\,.
\]
The \emph{north, south, east}, and \emph{west} sides (or simply
\emph{N-side, S-side, E-side}, and \emph{W-side}), denoted
\( N(\tau) \), \( S(\tau) \), \( E(\tau) \), and \( W(\tau) \),
respectively, of a tile \( \tau \) are defined as in \Cref{fig:NESW}. A tile is
also called a \emph{cell}.

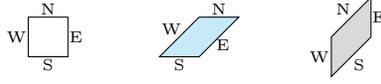
\begin{figure}
  \centering
\begin{tikzpicture}[scale=1.5]
  \cell{0}{0}
  \node at (0.5, -0.2) {\tiny S};
  \node at (0.5, 1.2) {\tiny N};
  \node at (1.2, 0.5) {\tiny E};
  \node at (-.3, 0.5) {\tiny W};
\end{tikzpicture}
\qquad 
\begin{tikzpicture}[scale=1.5]
  \short{0}{0}
  \node at (1.5, 1.2) {\tiny N};
  \node at (0.5, -0.2) {\tiny S};
  \node at (1.6, 0.3) {\tiny E};
  \node at (0.3, 0.7) {\tiny W};
\end{tikzpicture}
\qquad 
\begin{tikzpicture}[scale=1.5]
  \tall{0}{0}
  \node at (.3, 1.7) {\tiny N};
  \node at (0.7, 0.3) {\tiny S};
  \node at (1.2, 1.5) {\tiny E};
  \node at (-.3, 0.5) {\tiny W};
\end{tikzpicture}
\caption{The N-side \( N(\tau) \), S-side \( S(\tau) \), E-side \( E(\tau) \), and W-side \( W(\tau) \) of a tile \( \tau \) are
  indicated with letters N, S, E, and W, respectively.}
  \label{fig:NESW}
\end{figure}

\begin{defn}
  Let \(n\) be a nonnegative integer, and let
  \( w=w_1 \cdots w_n\in \{0,1,2\}^n \). The \emph{rhombic diagram} \( \Gamma_w \) of
  \( w \) is the diagram obtained as follows; see \Cref{fig:rhombic_diag} for an example.
  \begin{enumerate}
  \item Draw a lattice path, called the \emph{southeast border}, by reading \(w\) from left to right and drawing
    a step south for a \(2\), a step southwest for a \(1\), and a step
    west for a \(0\). From the same starting point, draw a lattice
    path, called the \emph{northwest border}, consisting of \( a \)
    steps west, \( b \) steps southwest, and \( c \) steps south to
    obtain a closed region, where \( a \), \( b \), and \( c \) are
    the number of \( 2 \)'s, \( 1 \)'s, and \( 0 \)'s in \( w \),
    respectively.
  \item We tile the closed region recursively with respect to
    \( \inv(w) \). If \( \inv(w) = 0 \), then the region has no
    interior and there is nothing to tile. Otherwise, find the
    smallest integer \( i \) such that \( w_i>w_{i+1} \). Let \( w' \)
    be the word obtained from \( w \) by exchanging \( w_i \) and
    \( w_{i+1} \). Then we first tile the region for \( w' \) using
    the recursive step, and the remaining region for \( w \) is tiled
    with a tall tile, a square, or a short tile if
    \( (w_i,w_{i+1}) = (2,1) \), \( (w_i,w_{i+1}) = (2,0) \), or
    \( (w_i,w_{i+1}) = (1,0) \), respectively.
  \end{enumerate}
\end{defn}

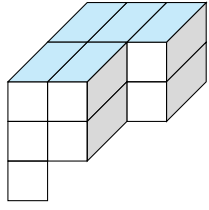
\begin{figure}
  \centering
  \begin{tikzpicture}[scale=1.5]
  \tall{4}{3} \cell{3}{3} \tall{2}{2} \cell{1}{2} \cell{0}{2} \tall{4}{2} \cell{3}{2} \tall{2}{1} \cell{1}{1} \cell{0}{1} \short{3}{4} \short{2}{4} \short{1}{4} \short{1}{3} \short{0}{3} \cell{0}{0} 
\end{tikzpicture}
\caption{The rhombic diagram \( \Gamma_w \) of the word \( w=2\,2\,1\,0\,1\,0\,2\,0 \).}
\label{fig:rhombic_diag}
\end{figure}

We also call south, west, and southwest edges vertical, horizontal, and diagonal, respectively.

\begin{defn}
  A \emph{strip} in a rhombic diagram is a maximal sequence of connected 
  tiles such that all the edges of adjacency are parallel. A
  \emph{horizontal strip} has vertical edges of adjacency (and thus
  contains no short tiles), a \emph{vertical strip} has horizontal
  edges of adjacency (and thus contains no tall tiles), and a
  \emph{diagonal strip} has diagonal edges of adjacency (and thus
  contains no squares). 
  
  Let \(\mathcal{L}=(\ell_1,\ldots,\ell_k)\) be a sequence of real numbers.
  We say that the southeast border of a RAT \(R\) is \emph{labeled by
\(\mathcal{L}\)} if its southeast border edges are labeled by
\(\ell_1,\ldots,\ell_k\) in this order from northeast to southwest. In a RAT with a labeled southeast border, we refer to \emph{the strip with label
    \(i\)} as the strip originating at the edge with label
  \(i\in \mathcal{L}\). Every tile belongs to exactly two strips: for
  \(i,j\in\mathcal{L}\), the cell \((i,j)\) refers to the unique cell at
  the intersection of the strips labeled \(i\) and \(j\), when such a cell
  exists. In analogy with classical diagrams, we refer to horizontal
  strips as \emph{rows} and vertical strips as \emph{columns}.
\end{defn}

If the southeast border and the northwest border have a common
diagonal (resp.~vertical and horizontal) segment \( s \), then we also
consider \( \{ s \} \) as a diagonal (resp.~vertical and horizontal)
strip containing no cells.

\begin{remark}
The recursive construction of the tiling may be equivalently described in terms of strips. Each vertical strip has all its square tiles strictly south of its short tiles; equivalently, each diagonal strip has all of its tall tiles strictly east of its short tiles. We call such a tiling the \emph{maximal tiling}.  
\end{remark}

\begin{defn}\label{RAT_def}
  A \emph{rhombic alternative tableau} (or \emph{RAT} for short) of
  shape \( w \) is the rhombic diagram \( \Gamma_w \) with a filling
  of its tiles with up-arrows and left-arrows satisfying the following
  conditions.
  \begin{itemize}
  \item Every up-arrow (resp.~left-arrow) must lie in a column
    (resp.~row) of the diagram.
  \item There is no arrow pointing to any other. Here, a left-arrow
    (resp.~an up-arrow) points to all cells in the same row
    (resp.~column) that are weakly southwest (resp.~northeast) of it.
  \end{itemize}
  We denote by \(\RAT(n,r)\) the set of RAT whose southeast border has a total of
  \( n \) steps with exactly \( r \) of them being diagonal steps.
\end{defn}

We note that the original definition of RAT in \cite{MV18} uses
\( \alpha \)'s and \( \beta \)'s in place of up-arrows and
left-arrows, respectively.

\begin{remark}\label{rem:1}
  Although we only consider maximal tilings in this paper, it is possible
  to define RAT using arbitrary tilings. It is shown in
  \cite[Proposition~2.8]{MV18} that for any fixed border, there is a
  bijection between the set of RAT with the maximal tiling and the set
  of RAT with any other tiling.
\end{remark}

We extend the definitions of free cells, free rows, and free columns
to RAT: a \emph{free cell} of \( R \) is a cell without an arrow such
that there is no arrow pointing to this cell, and a \emph{free row}
(resp.~\emph{free column}) is a row (resp.~column) that does not
contain any left-arrows (resp.~up-arrows). We denote by \(\fcell(R)\),
\(\frow(R)\), and \(\fcol(R)\) the numbers of free cells, free rows,
and free columns in a RAT \( R \), respectively. We also define
\( \row(R) \) to be the number of rows in \( R \).

As with alternative tableaux, it will be convenient for us to introduce extended RAT.

\begin{defn}\label{def:12}
  An \emph{extended RAT} is a RAT \(R\) satisfying the following conditions:
  \begin{itemize}
  \item the shape of \( R \) begins with a 1, that is, the southwest border begins with a diagonal step;
  \item every column of \(R\) contains an up-arrow.
  \end{itemize}
  We denote by \( \RAT^+(n+1,r+1) \) the set of extended RAT in \( \RAT(n+1,r+1) \).
\end{defn}

As with extended alternative tableaux, there is a bijection between
\( \RAT(n,r) \) and \( \RAT^+(n+1,r+1) \). Let \(R\in\RAT(n,r)\), and
define \( R^+\in\RAT(n+1,r+1) \) to be the RAT obtained as follows.
First, add a diagonal strip consisting of \(\ell\) short rhombi at the
top of \(R\), where \(\ell\) is the number of columns in \(R\). Then,
place an up-arrow in the topmost cell of each free column. Conversely,
removing the topmost diagonal strip from $R^+\in\RAT^+(n+1,r+1)$
recovers $R\in\RAT(n,r)$. See \cref{fig:RAT+} for an example.

\begin{figure}
  \centering
\begin{tikzpicture}[scale=1.5]
  \tall{4}{3} \cell{3}{3} \tall{2}{2} \cell{1}{2} \cell{0}{2} \tall{4}{2} \cell{3}{2} \tall{2}{1} \cell{1}{1} \cell{0}{1} \short{3}{4} \short{2}{4} \short{1}{4} \short{1}{3} \short{0}{3} \cell{0}{0} \node at (0.500000000000000, -0.200000000000000) {\tiny 8};
\node at (1.20000000000000, 0.500000000000000) {\tiny 7};
\node at (1.50000000000000, 0.800000000000000) {\tiny 6};
\node at (2.60000000000000, 1.30000000000000) {\tiny 5};
\node at (3.50000000000000, 1.80000000000000) {\tiny 4};
\node at (4.60000000000000, 2.30000000000000) {\tiny 3};
\node at (5.20000000000000, 3.50000000000000) {\tiny 2};
\node at (5.20000000000000, 4.50000000000000) {\tiny 1};
 \slantUpArr{2}{4} \slantLeftArr{2}{2} \leftArr{0}{0}
\end{tikzpicture}
\qquad \qquad \qquad 
\begin{tikzpicture}[scale=1.5]
  \short{4}{5} \short{3}{5} \short{2}{5} \tall{4}{3} \cell{3}{3} \tall{2}{2} \cell{1}{2} \cell{0}{2} \tall{4}{2} \cell{3}{2} \tall{2}{1} \cell{1}{1} \cell{0}{1} \short{3}{4} \short{2}{4} \short{1}{4} \short{1}{3} \short{0}{3} \cell{0}{0} \node at (0.500000000000000, -0.200000000000000) {\tiny 9};
\node at (1.20000000000000, 0.500000000000000) {\tiny 8};
\node at (1.50000000000000, 0.800000000000000) {\tiny 7};
\node at (2.60000000000000, 1.30000000000000) {\tiny 6};
\node at (3.50000000000000, 1.80000000000000) {\tiny 5};
\node at (4.60000000000000, 2.30000000000000) {\tiny 4};
\node at (5.20000000000000, 3.50000000000000) {\tiny 3};
\node at (5.20000000000000, 4.50000000000000) {\tiny 2};
\node at (5.60000000000000, 5.30000000000000) {\tiny 1};
 \slantUpArr{4}{5} \slantUpArr{2}{4} \slantUpArr{2}{5} \slantLeftArr{2}{2} \leftArr{0}{0}
\end{tikzpicture}
\caption{The correspondence between the RAT \( T\in\RAT(n,r) \) and
  the extended RAT \( T^+\in \RAT^+(n+1,r+1) \). In the RAT on the
  left, there are 11 free cells, row 2 is free, and columns 4 and 8
  are free.}
  \label{fig:RAT+}
\end{figure}
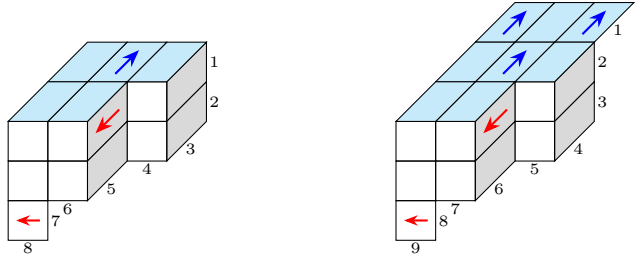

\section{The structure of RAT}
\label{sec:structure}

There are two natural ways of studying RAT through alternative
tableaux. The first way is to decompose a RAT into a labeled set of
special RAT, called packed RAT, which can then be identified with
alternative tableaux: this is covered in \cref{sec:split}. The second
is to use a \emph{flattening map} to embed a RAT inside a straight
shape and thus treat it as an alternative tableau with some additional
restrictions: this is done in \cref{sec:flatten}.

\subsection{Decomposition of a RAT}\label{sec:split}

In earlier work, Nadeau \cite{Nadeau} showed that every alternative tableau can be
decomposed into packed tableaux, which are alternative tableaux with
the maximum number of arrows. In this subsection, we generalize this
decomposition to RAT.

\begin{defn}
We say that a RAT of size \((n,r)\) is \emph{packed} if it has the
maximal number \(n-1\) of arrows: this occurs only for sizes \((n,0)\) and
\((n,1)\). At size \((n,0)\), either the RAT has no left-arrow on the
top row, in which case we say it has \emph{horizontal type}, or no up-arrow on
the leftmost column, in which case we say it has \emph{vertical type}. At size
\((n,1)\), we say that \(T\) has \emph{diagonal type}.
\end{defn}

Consider a RAT \(T\) with its southeast border labeled by \([n]\). We
define \(\simeq\) to be the smallest equivalence relation on these
labels such that \(i\simeq j\) whenever the intersection of the strips
labeled $i$ and $j$ contains an arrow. Given an equivalence class
\(L\subseteq [n]\), we define \(T[L]\) as the RAT labeled by \(L\)
obtained by restricting \(T\) to its edges labeled by \(L\). Observe
that $T[L]$ is packed. Let \( \spl(T) \) denote the collection of the
packed RAT \( T[L] \) for \( L\in [n]/\simeq \); see
\cref{fig:SplitRAT}.

\begin{figure}
\centering
\begin{tikzpicture}[scale=1.5]
    \cell{9}{8} \tall{8}{7} \cell{7}{7} \cell{6}{7} \cell{5}{7} \tall{4}{6} \tall{3}{5} \cell{2}{5} \cell{1}{5} \tall{0}{4} \short{7}{8} \short{6}{8} \short{5}{8} \short{4}{8} \short{3}{8} \cell{7}{6} \cell{6}{6} \cell{5}{6} \tall{4}{5} \tall{3}{4} \cell{2}{4} \cell{1}{4} \tall{0}{3} \cell{7}{5} \cell{6}{5} \cell{5}{5} \tall{4}{4} \tall{3}{3} \cell{2}{3} \cell{1}{3} \tall{0}{2} \cell{6}{4} \cell{5}{4} \tall{4}{3} \tall{3}{2} \cell{2}{2} \cell{1}{2} \tall{0}{1} \short{3}{7} \short{2}{7} \tall{3}{1} \cell{2}{1} \cell{1}{1} \tall{0}{0} \short{2}{6} \short{1}{6} \node at (0.600000000000000, 0.300000000000000) {\tiny 15};
\node at (1.50000000000000, 0.800000000000000) {\tiny 14};
\node at (2.50000000000000, 0.800000000000000) {\tiny 13};
\node at (3.60000000000000, 1.30000000000000) {\tiny 12};
\node at (4.20000000000000, 2.50000000000000) {\tiny 11};
\node at (4.60000000000000, 3.30000000000000) {\tiny 10};
\node at (5.50000000000000, 3.80000000000000) {\tiny 9};
\node at (6.50000000000000, 3.80000000000000) {\tiny 8};
\node at (7.20000000000000, 4.50000000000000) {\tiny 7};
\node at (7.50000000000000, 4.80000000000000) {\tiny 6};
\node at (8.20000000000000, 5.50000000000000) {\tiny 5};
\node at (8.20000000000000, 6.50000000000000) {\tiny 4};
\node at (8.60000000000000, 7.30000000000000) {\tiny 3};
\node at (9.50000000000000, 7.80000000000000) {\tiny 2};
\node at (10.2000000000000, 8.50000000000000) {\tiny 1};
 \upArr{9}{8} \slantUpArr{7}{8} \upArr{6}{4} \slantUpArr{3}{7} \slantUpArr{1}{6} \leftArr{5}{4} \slantLeftArr{3}{3} \leftArr{1}{1}
 \draw[->] (10.5,4) -- (12.5,4);
 \node at (11.5,4.5) {Split};
\end{tikzpicture} \quad 
\raisebox{5.8em}{
\( \left\{
  \begin{aligned}
&\begin{tikzpicture}[scale=1.5]
  \cell{0}{0} \node at (0.500000000000000, -0.200000000000000) {\tiny 2};
\node at (1.20000000000000, 0.500000000000000) {\tiny 1};
 \upArr{0}{0}
\end{tikzpicture}
  \quad
\raisebox{0.8em}{\begin{tikzpicture}[scale=1.5]
\draw (0,0) -- (0,1);
\node at (0.200000000000000, 0.500000000000000) {\tiny 4};
\end{tikzpicture}}\\
&\begin{tikzpicture}[scale=1.5]
 \cell{1}{0} \cell{0}{0} \node at (0.500000000000000, -0.200000000000000) {\tiny 9};
\node at (1.50000000000000, -0.200000000000000) {\tiny 8};
\node at (2.20000000000000, 0.500000000000000) {\tiny 7};
 \upArr{1}{0} \leftArr{0}{0}
\end{tikzpicture}\\
& \begin{tikzpicture}[scale=1.5]
\short{0}{0} \node at (0.500000000000000, -0.200000000000000) {\tiny 6};
\node at (1.60000000000000, 0.300000000000000) {\tiny 3};
 \slantUpArr{0}{0}
\end{tikzpicture}
\quad
\begin{tikzpicture}[scale=1.5]
\tall{1}{1} \cell{0}{1} \tall{1}{0} \cell{0}{0} \short{0}{2} \node at (0.500000000000000, -0.200000000000000) {\tiny 14};
\node at (1.60000000000000, 0.300000000000000) {\tiny 12};
\node at (2.20000000000000, 1.50000000000000) {\tiny 11};
\node at (2.20000000000000, 2.50000000000000) {\tiny 5};
 \slantUpArr{0}{2} \slantLeftArr{1}{1} \leftArr{0}{0}
\end{tikzpicture} \quad 
\begin{tikzpicture}[scale=1.5]
\short{0}{0} \node at (0.500000000000000, -0.200000000000000) {\tiny 13};
\node at (1.60000000000000, 0.300000000000000) {\tiny 10};
 \slantUpArr{0}{0}
\end{tikzpicture}
\quad
\raisebox{0.8em}{\begin{tikzpicture}[scale=1.5]
\draw (0,0) -- (1,1);
\node at (0.600000000000000, 0.300000000000000) {\tiny 15};
\end{tikzpicture}
}
\end{aligned}
\right\} \)
}
\caption{Splitting a RAT of type \((15,4)\). The packed RAT on the right are of horizontal, vertical, and diagonal type, from top to bottom, respectively.
  \label{fig:SplitRAT}}
\end{figure}
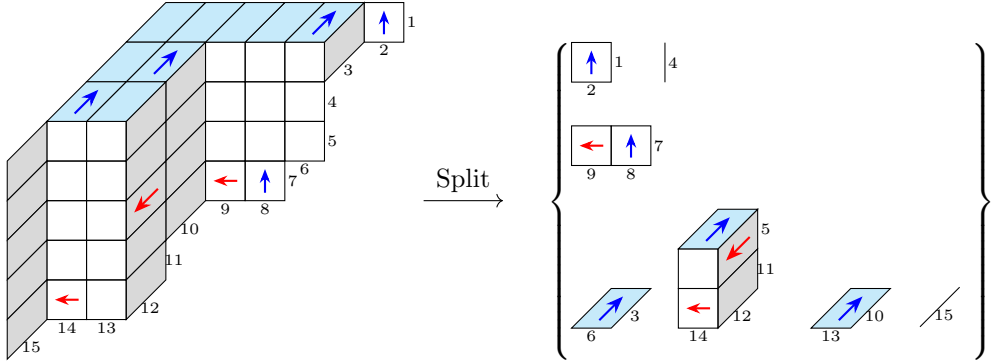

\begin{thm}
\label{theo:SplitRAT}
The map \( \spl \), which sends \( T\mapsto \{T[L]:L\in [n]/\simeq\}\), is a
bijection between
\begin{enumerate}
\item \( \RAT(n,r) \) and
\item the set of collections of packed RAT with labels whose disjoint
  union is \([n]\), and with \(r\) of these RAT being of diagonal
  type.
\end{enumerate}
Moreover, the number of packed RAT of horizontal (\emph{resp.}
vertical) type is the number of free rows (\emph{resp.} columns) of
\(T\).
\end{thm}

\begin{proof}
The case \(r=0\) is the case of alternative tableaux, which is shown in \cite[Theorem 2.14]{Nadeau}. To extend the proof to the case of general \(r\), one notices first that diagonal steps on the border are always free, in the sense that no arrow points towards them. It follows that they behave as free rows or columns and the proof of \cite[Theorem 2.14]{Nadeau} applies verbatim. 
\end{proof}

Using this correspondence, we give a simple proof of the enumeration
formula \eqref{Z_nr} for the RAT of size \((n,r)\). Let
\(H(x) = \sum_{n\ge1} h(n)\frac{x^n}{n!} \),
\(V(x) = \sum_{n\ge1} v(n)\frac{x^n}{n!}\), and
\(D(x) = \sum_{n\ge1} d(n)\frac{x^n}{n!}\), where \( h(n) \),
\( v(n) \), and \( d(n) \) are the numbers of packed RAT of
horizontal, vertical and diagonal types, respectively, whose southeast
borders have \( n \) steps. For a fixed integer \( r\ge0 \), by
\Cref{theo:SplitRAT} and the exponential formula
\cite[Corollary~5.1.6]{EC2}, we have
\begin{equation}
\label{eq:split}
\sum_{n\geq r}\Y_{n,r}(\alpha,\beta,1)\frac{x^n}{n!}=\frac{D(x)^r}{r!}\exp(\alpha H(x))\exp(\beta V(x)).
\end{equation}

It is shown in \cite{Nadeau} (in terms of alternative tableaux) that
\( h(n)=v(n)=(n-1)! \), and thus \(H(x)=V(x)=-\log(1-x)\). Now given a
packed RAT \(T\) of diagonal type, consider the following construction
illustrated in \cref{fig:StraightRAT}: delete its short rhombi on top,
and change tall rhombi to square cells. Let the resulting alternative
tableau be denoted by \(\str(T)\).

\begin{figure}
\centering
\begin{tikzpicture}[scale=1.5]
  \cell{6}{6} \cell{5}{6} \tall{4}{5} \cell{3}{5} \cell{2}{5} \cell{1}{5} \cell{0}{5} \cell{6}{5} \cell{5}{5} \tall{4}{4} \cell{3}{4} \cell{2}{4} \cell{1}{4} \cell{0}{4} \cell{6}{4} \cell{5}{4} \tall{4}{3} \cell{3}{3} \cell{2}{3} \cell{1}{3} \cell{0}{3} \short{3}{6} \short{2}{6} \short{1}{6} \short{0}{6} \cell{2}{2} \cell{1}{2} \cell{0}{2} \cell{1}{1} \cell{0}{1} \cell{1}{0} \cell{0}{0} \node at (0.500000000000000, -0.200000000000000) {\tiny 13};
\node at (1.50000000000000, -0.200000000000000) {\tiny 12};
\node at (2.20000000000000, 0.500000000000000) {\tiny 11};
\node at (2.20000000000000, 1.50000000000000) {\tiny 10};
\node at (2.50000000000000, 1.80000000000000) {\tiny 9};
\node at (3.20000000000000, 2.50000000000000) {\tiny 8};
\node at (3.50000000000000, 2.80000000000000) {\tiny 7};
\node at (4.60000000000000, 3.30000000000000) {\tiny 6};
\node at (5.50000000000000, 3.80000000000000) {\tiny 5};
\node at (6.50000000000000, 3.80000000000000) {\tiny 4};
\node at (7.20000000000000, 4.50000000000000) {\tiny 3};
\node at (7.20000000000000, 5.50000000000000) {\tiny 2};
\node at (7.20000000000000, 6.50000000000000) {\tiny 1};
 \upArr{5}{6} \upArr{1}{5} \upArr{3}{4} \upArr{6}{4} \slantUpArr{2}{6} \slantUpArr{0}{6} \leftArr{0}{5} \leftArr{1}{4} \slantLeftArr{4}{3} \leftArr{2}{2} \leftArr{0}{1} \leftArr{1}{0}
 \draw[->] (9,3) -- (12,3) node [midway,above] {Straight};
\end{tikzpicture}
\qquad 
\begin{tikzpicture}[scale=1.5]
\cell{6}{5} \cell{5}{5} \cellt{4}{5} \cell{3}{5} \cell{2}{5} \cell{1}{5} \cell{0}{5} \cell{6}{4} \cell{5}{4} \cellt{4}{4} \cell{3}{4} \cell{2}{4} \cell{1}{4} \cell{0}{4} \cell{6}{3} \cell{5}{3} \cellt{4}{3} \cell{3}{3} \cell{2}{3} \cell{1}{3} \cell{0}{3} \cell{2}{2} \cell{1}{2} \cell{0}{2} \cell{1}{1} \cell{0}{1} \cell{1}{0} \cell{0}{0} \node at (0.500000000000000, -0.200000000000000) {\tiny 13};
\node at (1.50000000000000, -0.200000000000000) {\tiny 12};
\node at (2.20000000000000, 0.500000000000000) {\tiny 11};
\node at (2.20000000000000, 1.50000000000000) {\tiny 10};
\node at (2.50000000000000, 1.80000000000000) {\tiny 9};
\node at (3.20000000000000, 2.50000000000000) {\tiny 8};
\node at (3.50000000000000, 2.80000000000000) {\tiny 7};
\node at (4.50000000000000, 2.80000000000000) {\tiny 6};
\node at (5.50000000000000, 2.80000000000000) {\tiny 5};
\node at (6.50000000000000, 2.80000000000000) {\tiny 4};
\node at (7.20000000000000, 3.50000000000000) {\tiny 3};
\node at (7.20000000000000, 4.50000000000000) {\tiny 2};
\node at (7.20000000000000, 5.50000000000000) {\tiny 1};
 \upArr{5}{5} \upArr{1}{5} \upArr{3}{4} \upArr{6}{3} \leftArr{0}{5} \leftArr{1}{4} \leftArr{4}{3} \leftArr{2}{2} \leftArr{0}{1} \leftArr{1}{0}
\end{tikzpicture}
\caption{The bijection in \cref{lemma:straight}. 
\label{fig:StraightRAT}}
\end{figure}
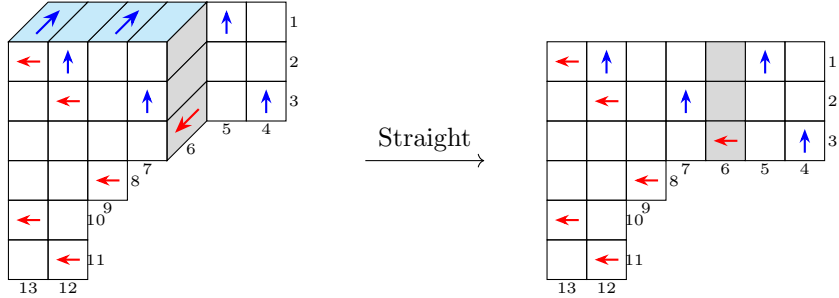

\begin{lemma}
\label{lemma:straight}
Let \( n \) be a positive integer.
The map \(T\mapsto \str(T)\) is a bijection from the set of packed
RAT of diagonal type and size \((n,1)\) to the set of alternative
tableaux of size \(n\) with no free rows.
\end{lemma}

\begin{proof}
Let \(U\) be an alternative tableau of size \( n \) with no free rows (which hence must contain at least one free column). Let \(C\) be the rightmost free column. Transform the cells of \(C\) to tall rhombi preserving the contents, which is possible since \(C\) contains no up-arrows. Then add a short rhombus to the top of each column left of \(C\), and fill it with an up-arrow if its column is free. 
By construction, this defines the inverse of the map \(\str\).
\end{proof} 

From this lemma, it follows that the number \( d(n) \) of packed RAT
of diagonal type and size \( (n,1) \) is equal to \(n!\) for
\(n\geq 1\), since the set of alternative tableaux of size \(n\) with
no free rows is in bijection with the set of alternative tableaux of
size \(n-1\), by simply deleting the leftmost column. Therefore
\(D(x)=\frac{1}{1-x}-1=\frac{x}{1-x}\). We thus obtain
from~\eqref{eq:split} that
\begin{align*}
  \sum_{n\geq r}\Y_{n,r}(\alpha,\beta,1)\frac{x^n}{n!}
&=\frac{1}{r!}\left(\frac{x}{1-x}\right)^r\exp(-\alpha\log(1-x))\exp(-\beta\log(1-x))\\
&=\frac{x^r}{r!}(1-x)^{-r-\alpha-\beta}\\
&=\frac{x^r}{r!}\sum_{j\geq 0}(\alpha+\beta+r)_j\frac{x^j}{j!},
\end{align*}
and so
\[
\Y_{n,r}(\alpha,\beta,1)=\frac{n!}{r!}\frac{(\alpha+\beta+r)_{n-r}}{(n-r)!}={n \choose r}(\alpha+\beta+r)_{n-r},
\]
 which completes the proof of \eqref{Z_nr}.

\subsection{Flattening map}\label{sec:flatten}

In this subsection, we construct a map that converts a RAT into an
alternative tableau, thereby embedding the set of RAT into a subset of
alternative tableaux. This embedding allows us to extend known bijections on alternative tableaux to RAT
in a natural way.

The map we define requires us to append auxiliary strips to the top of
the tableau. To preserve the original edge labels, we assign to the
added edges a new collection of labels that are strictly smaller than
the labels in $[n]$ used for the original RAT. To that end, fix an increasing
sequence \( \epsilon = (\epsilon_1<\epsilon_2<\cdots) \) of positive
real numbers less than \( 1 \) to serve as the labels of the auxiliary
edges.

\begin{defn}[Flattening map]
  Let \( R\in \RAT^+(n+1,r+1) \) and suppose that
  \( d_1,\dots,d_{r+1} \) are the labels of the diagonal edges. The
  \emph{flattening} \( \flat(R) \) of \( R \) is the alternative
  tableau \( T \in\AT(n+r+2) \) obtained as follows.
\begin{itemize}
\item The southeast border of \( T \) begins with $r+1$ vertical edges
  labeled \( \epsilon_1,\dots,\epsilon_{r+1} \). The remaining edges
  are labeled \( 1,\dots,n+1 \). The edge of $T$ with label \( i\in[n+1] \) is
  vertical if the \( i \)th edge of \( R \) is vertical, and
  horizontal otherwise.
\item For each \( i\in [r+1] \), add an up-arrow in the cell \( (\epsilon_i,d_i) \) in \( T \).
\item For every arrow of \( R \) in the cell \( (i,j) \), add the same
  arrow in the cell \( (i,j) \) in \( T \) if $i\not\in\{d_1,\ldots,d_{r+1}\}$. If \( i=d_t \) for some
  \( t \), then add the arrow in the cell \( (\epsilon_t,j) \) instead of
  \( (i,j) \).
\end{itemize}
\end{defn}

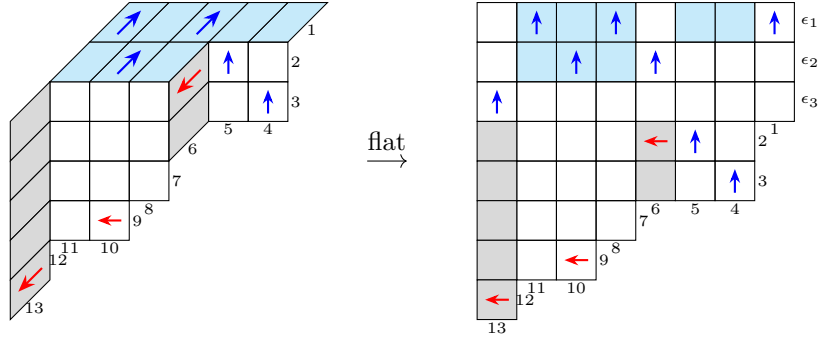
\begin{figure}
  \centering
\begin{tikzpicture}[scale=1.5]
    \short{6}{7} \short{5}{7} \short{4}{7} \short{3}{7} \short{2}{7} \cell{6}{6} \cell{5}{6} \tall{4}{5} \cell{3}{5} \cell{2}{5} \cell{1}{5} \tall{0}{4} \cell{6}{5} \cell{5}{5} \tall{4}{4} \cell{3}{4} \cell{2}{4} \cell{1}{4} \tall{0}{3} \short{3}{6} \short{2}{6} \short{1}{6} \cell{3}{3} \cell{2}{3} \cell{1}{3} \tall{0}{2} \cell{2}{2} \cell{1}{2} \tall{0}{1} \tall{0}{0} \node at (0.600000000000000, 0.300000000000000) {\tiny 13};
\node at (1.20000000000000, 1.50000000000000) {\tiny 12};
\node at (1.50000000000000, 1.80000000000000) {\tiny 11};
\node at (2.50000000000000, 1.80000000000000) {\tiny 10};
\node at (3.20000000000000, 2.50000000000000) {\tiny 9};
\node at (3.50000000000000, 2.80000000000000) {\tiny 8};
\node at (4.20000000000000, 3.50000000000000) {\tiny 7};
\node at (4.60000000000000, 4.30000000000000) {\tiny 6};
\node at (5.50000000000000, 4.80000000000000) {\tiny 5};
\node at (6.50000000000000, 4.80000000000000) {\tiny 4};
\node at (7.20000000000000, 5.50000000000000) {\tiny 3};
\node at (7.20000000000000, 6.50000000000000) {\tiny 2};
\node at (7.60000000000000, 7.30000000000000) {\tiny 1};
 \slantUpArr{2}{7} \slantUpArr{4}{7} \slantUpArr{2}{6} \upArr{5}{6} \upArr{6}{5} \slantLeftArr{4}{5} \leftArr{2}{2} \slantLeftArr{0}{0}
 \draw[->] (9,4) -- (10,4);
 \node at (9.5,4.5) {flat};
 \phantom{\node at (0.500000000000000, -0.200000000000000) {\tiny 13};}
\end{tikzpicture}
\qquad 
\begin{tikzpicture}[scale=1.5]
\cell{7}{7} \cells{6}{7} \cells{5}{7} \cell{4}{7} \cells{3}{7} \cells{2}{7} \cells{1}{7} \cell{0}{7} \cell{7}{6} \cell{6}{6} \cell{5}{6} \cell{4}{6} \cells{3}{6} \cells{2}{6} \cells{1}{6} \cell{0}{6} \cell{7}{5} \cell{6}{5} \cell{5}{5} \cell{4}{5} \cell{3}{5} \cell{2}{5} \cell{1}{5} \cell{0}{5} \cell{6}{4} \cell{5}{4} \cellt{4}{4} \cell{3}{4} \cell{2}{4} \cell{1}{4} \cellt{0}{4} \cell{6}{3} \cell{5}{3} \cellt{4}{3} \cell{3}{3} \cell{2}{3} \cell{1}{3} \cellt{0}{3} \cell{3}{2} \cell{2}{2} \cell{1}{2} \cellt{0}{2} \cell{2}{1} \cell{1}{1} \cellt{0}{1} \cellt{0}{0} \node at (0.500000000000000, -0.200000000000000) {\tiny 13};
\node at (1.20000000000000, 0.500000000000000) {\tiny 12};
\node at (1.50000000000000, 0.800000000000000) {\tiny 11};
\node at (2.50000000000000, 0.800000000000000) {\tiny 10};
\node at (3.20000000000000, 1.50000000000000) {\tiny 9};
\node at (3.50000000000000, 1.80000000000000) {\tiny 8};
\node at (4.20000000000000, 2.50000000000000) {\tiny 7};
\node at (4.50000000000000, 2.80000000000000) {\tiny 6};
\node at (5.50000000000000, 2.80000000000000) {\tiny 5};
\node at (6.50000000000000, 2.80000000000000) {\tiny 4};
\node at (7.20000000000000, 3.50000000000000) {\tiny 3};
\node at (7.20000000000000, 4.50000000000000) {\tiny 2};
\node at (7.50000000000000, 4.80000000000000) {\tiny 1};
\node at (8.40000000000000, 5.50000000000000) {\tiny \( \epsilon_3 \)};
\node at (8.40000000000000, 6.50000000000000) {\tiny \( \epsilon_2 \)};
\node at (8.40000000000000, 7.50000000000000) {\tiny \( \epsilon_1 \)};
 \upArr{7}{7} \upArr{3}{7} \upArr{1}{7} \upArr{4}{6} \upArr{2}{6} \upArr{0}{5} \upArr{5}{4} \upArr{6}{3} \leftArr{4}{4} \leftArr{2}{1} \leftArr{0}{0}
\end{tikzpicture}
\caption{The flattening map sends \( R\in\RAT^+(13,3) \) on the left to the
  alternative tableau on the right.
The cells coming from short rhombi and tall rhombi are colored accordingly.}
  \label{fig:flat}
\end{figure}

See \Cref{fig:flat} for an example of the flattening map. In order to characterize the images of the flattening map, we define
\( \AT^+(n+1,r+1) \) to be the set of alternative tableaux \( T \) in
\( \AT^+(n+1) \), with southeast border labeled by \((\epsilon_1,\dots,\epsilon_{r+1}, 1,\dots,n+1) \),  satisfying the following properties.
\begin{itemize}
\item The southeast border of \( T \) begins with \( r+1 \) vertical
  edges.
\item The top \(r+1\) rows of \(T\) contain no left-arrows.
\item Each of the top \(r+1\) rows of \(T\) contains at least one
  up-arrow. If \(d_1,\ldots,d_{r+1}\) are the labels of the columns
  containing the rightmost up-arrows in the first \( r+1 \) rows,
  respectively, we must have \(d_1<\cdots<d_{r+1}\).
\end{itemize}

By construction, for \( R\in \RAT^+(n+1,r+1) \), we have
\( \flat(R) = T \in \AT^+(n+1,r+1) \). Conversely, given
\( T\in \AT^+(n+1,r+1) \), we can reconstruct \( R \) as follows. Let
\(d_1<\dots<d_{r+1}\) be the labels of the columns containing the
rightmost up-arrows in the first \( r+1 \) rows in \( T \).
Then the shape of \( R \) is \( w=w_1 \cdots w_{n+1} \),
where 
\[
  w_i =
  \begin{cases}
   1 & \mbox{if \( i\in \{d_1,\dots,d_{r+1}\} \)},\\
    0 & \mbox{if \( i \) is the label of a column in \( T \)},\\
    2 & \mbox{if \( i \) is the label of a row in \( T \)}.
  \end{cases}
\]
For each arrow of \( T \) in cell \( (i,j) \) that is not the rightmost
up-arrow within the first \( r+1 \) rows, place the same type of arrow in the
cell \( (i',j) \) in \( R \), where \( i'=d_k \) if
\( i=\epsilon_k \), and \( i'=i \) otherwise. This shows the following
proposition.

\begin{prop}\label{lem:6}
  The map \( \flat: \RAT^+(n+1,r+1) \to \AT^+(n+1,r+1) \) is a
  bijection.
\end{prop}

This bijective construction now allows us to consider RAT as alternative tableaux.

\section{Three bijections between RAT and assembl\'ees}\label{sec:bijections}

There are three well-known bijections that directly relate alternative
tableaux and permutations: the \emph{insertion bijection} of Corteel
and Nadeau \cite{Corteel2009}, the \emph{zigzag bijection} of Corteel
and Kim \cite{CorteelKim}, and the \emph{fusion-exchange bijection} of
Viennot \cite{Viennota}. Mandelshtam and Viennot
\cite{Mandelshtam2018} generalized the fusion-exchange bijection to
RAT. In this section, we generalize the insertion map and the zigzag
map to RAT. We show that these three bijections between RAT and
assembl\'ees are equivalent up to a simple transformation.

\subsection{The insertion map}\label{sec:phi_i}

In this subsection we define the insertion map
\[\Phi_I: \RAT^+(n+1,r+1) \rightarrow \A(n+1,r+1)\,.\] 
For this map, we no longer consider the canonical order of an
assembl\'{e}e. Instead, it is convenient to represent an assembl\'ee
as a certain reading word, which we call its \(\epsilon\)-word. This
representation allows us to view an assembl\'ee as a permutation on the alphabet $\{\epsilon_1,\ldots,\epsilon_{r+1},1,\ldots,n+1\}$, where \(\epsilon_1<\epsilon_2<\cdots<\epsilon_{r+1}<1\) are fixed
positive real numbers. This will give a natural way to extend results for alternative tableaux on permutations to rhombic
alternative tableaux.

\begin{defn}\label{def:e-word}
  Let \( \pi\in \A(n+1,r+1) \). The \emph{\( \epsilon \)-word} of
  \( \pi \), denoted $e(\pi)$, is the word on the alphabet
\(\{\epsilon_1,\ldots,\epsilon_{r+1},1,\ldots,n+1\}\), constructed as follows. First, we arrange the blocks of \( \pi \) as
  \( B_1,\dots,B_{r+1} \) such that \( d_1=1\in B_1 \) and if
  \(d_2,\ldots,d_{r+1}\) are the last elements of the blocks
  \(B_2,\ldots,B_{r+1}\), respectively, then \( d_2<\dots<d_{r+1} \).
  We then write $B_1=C_1\,1\,C_2$ and concatenate \( C_1,B_2,\dots, B_{r+1},C_2 \) as words. 
  Insert \( \epsilon_i \) after \( d_i \) for each \( i\in[r+1] \).
  Then \( e(\pi) \) is defined to be the resulting word.
  
  An assembl\'{e}e \( \pi\in\A(n+1,r+1) \) can be reconstructed from $e(\pi)$, by concatenating the subword before \(\epsilon_1\) with the subword after \(\epsilon_{r+1}\) to recover the first block $B_1$. We use the notation $e^{-1}(v)$ for the assembl\'ee whose $\epsilon$-word is $v$.

  Let \( \pi\in\A(n+1,r+1) \) and \( v=e(\pi) \).
  For an integer \( i\in [n+1] \), we say that
  \begin{itemize}
  \item \( i \) is an \emph{ascent} of \( v \) if it is followed by a
    larger element or if it is the last element of \( v \),
\item \( i \) is a \emph{descent} of \( v \) if it is followed by a
  smaller element.
  \end{itemize}
  Moreover, we say that a descent \( i \) is \emph{special} if it is
  followed by \( \epsilon_j \) for some \( j \), and \emph{regular}
  otherwise.
\end{defn}

\begin{example}
  Let \( \pi= [9\,\,10\,\,5\,\,2\,\,6]\, [11\,\,8\,\,1\,\,4\,\,3\,\,7]\, [12\,\,13]\in \A(13,3) \). Then
  \( d_1=1 \), \( d_2=6 \), \( d_3=13 \), and the \( \epsilon \)-word
  of \( \pi \) is \( e(\pi)=v \), where
  \[
    v = 
11\,\, 8\,\, 1\,\, \epsilon_1\,\, 9\,\,10\,\,5\,\, 2\,\,6\,\,\epsilon_2\,\, 12\,\,13\,\,\epsilon_3\,\, 4\,\, 3\,\, 7,
  \]
  and \( \pi = e^{-1}(v) \).
  The ascents of \( v \) are \( 2,3,7,9,12 \), the regular descents of
  \( v \) are \( 4,5,8,10,11 \), and the special descents of \( v \)
  are \( 1,6,13 \).
\end{example}

\begin{algorithm}[Insertion map]\label{alg:insertion}
Let \(R\in\RAT^+(n+1,r+1)\) be given as an input.
Then \( \Phi_I(R)\in\A(n+1,r+1)\) is constructed as follows.
  \begin{enumerate}
  \item Suppose that \( 1=d_1 < \cdots < d_{r+1} \) are the labels of
    diagonal strips and \( a_1 < \cdots < a_t \) are the labels of
    free rows in \( R \). Initialize
    \( v = \epsilon_1\, \cdots\, \epsilon_{r+1}\, a_1\, a_2\, \cdots\, a_t \).
  \item Suppose that \(R\) has \(m\) non-horizontal strips (i.e.,
    vertical or diagonal strips), and let \(\ell_1>\cdots>\ell_m\) be
    their labels from left to right. Then for \(i=1,\ldots,m\), we do
    the following.

\begin{enumerate}
\item   Let \(r_1<\cdots<r_k\) be the labels of the strips containing a left-arrow at the intersection with strip \(\ell_i\).
\item If strip \( \ell_i \) is vertical and it contains an up-arrow in
  a horizontal strip \( h \), then let \( j=h \). If strip
  \( \ell_i \) is vertical and it contains an up-arrow in a diagonal
  strip \( d_t \), then let \( j=\epsilon_t \). If \( \ell_i=d_t \),
  then let \( j=\epsilon_t \).
\item Insert \(r_1\,\cdots\, r_k\,\ell_i\) to the left of \(j\) in \( v \).
\end{enumerate}
\item Define \(\Phi_I(R)=e^{-1}(v)\) to be the assembl\'{e}e whose
  \( \epsilon \)-word is \( v \).
  \end{enumerate}
\end{algorithm}

\begin{example}\label{exa:2}
Let \(R\) be the RAT in \Cref{fig:flat}. The steps of the insertion to form $v=e(\Phi_I(R))$ are as follows.
\begin{center}
  \begin{tabular}{ll}
Initialization:& \( v=\epsilon_1\,\, \epsilon_2\,\, \epsilon_3\,\, 3\,\, 7 \)\\
Strip 13:& \( v=\epsilon_1\,\, \epsilon_2\,\, 12\,\,13\,\,\epsilon_3\,\, 3\,\, 7 \)\\
Strip 11:& \( v=11\,\, \epsilon_1\,\, \epsilon_2\,\, 12\,\,13\,\,\epsilon_3\,\, 3\,\, 7 \)\\
Strip 10:& \( v=11\,\, \epsilon_1\,\, 9\,\,10\,\,\epsilon_2\,\, 12\,\,13\,\,\epsilon_3\,\, 3\,\, 7 \)\\
Strip 8:& \( v=11\,\, 8\,\, \epsilon_1\,\, 9\,\,10\,\,\epsilon_2\,\, 12\,\,13\,\,\epsilon_3\,\, 3\,\, 7 \)\\
Strip 6:& \( v=11\,\, 8\,\, \epsilon_1\,\, 9\,\,10\,\,2\,\,6\,\,\epsilon_2\,\, 12\,\,13\,\,\epsilon_3\,\, 3\,\, 7 \)\\
Strip 5:& \( v=11\,\, 8\,\, \epsilon_1\,\, 9\,\,10\,\,5\,\, 2\,\,6\,\,\epsilon_2\,\, 12\,\,13\,\,\epsilon_3\,\, 3\,\, 7 \)\\
Strip 4:& \( v=11\,\, 8\,\, \epsilon_1\,\, 9\,\,10\,\,5\,\, 2\,\,6\,\,\epsilon_2\,\, 12\,\,13\,\,\epsilon_3\,\, 4\,\, 3\,\, 7 \) \\
Strip 1:& \( v=11\,\, 8\,\, 1\,\, \epsilon_1\,\, 9\,\,10\,\,5\,\, 2\,\,6\,\,\epsilon_2\,\, 12\,\,13\,\,\epsilon_3\,\, 4\,\, 3\,\, 7 \)\\
\end{tabular}
\end{center}
Thus
\[
  \Phi_I(R) = e^{-1}(v)= [11\,\,8\,\,1\,\,4\,\,3\,\,7]\,\, [9\,\,10\,\,5\,\,2\,\,6]\ [12\,\,13].
\]
\end{example}

The following lemmas follow from the construction of the map \( \Phi_I \) and \Cref{thm:CN}.

\begin{lemma}\label{lem:I-flat}
  Let \(R\in\RAT^+(n+1,r+1)\). Then the \( \epsilon \)-word of
  \( \Phi_I(R) \) is 
  \[
  e(\Phi_I(R)) = \Phi^{\AT}_I(\flat(R)).
  \]
\end{lemma}

\begin{lemma}\label{lem:shape}
  Suppose \(R\in\RAT^+(n+1,r+1)\) corresponds to \(\pi=\Phi_I(R)\in\A(n+1,r+1)\).
  Let \( v=e(\pi) \). For \(i\in [n+1]\setminus \{1\}\), we have the following.
  \begin{itemize}
  \item Strip \(i\) is a diagonal strip if and only if \(i\) is a special descent of \( v \).
  \item Strip \(i\) is a horizontal strip if and only if \(i\) is an ascent of \(v\).
  \item Strip \(i\) is a vertical strip if and only if \(i\) is a
    regular descent of \(v\).
  \item The labels of the free rows of \( R \) are the RL-minima of
    \( v \) except for \( \epsilon_1,\dots,\epsilon_{r+1} \).
  \end{itemize}
\end{lemma}

Note that \cref{lem:shape} enables us to determine the shape of
\( R \) from \( \pi \) and reverse the construction, which will be done in the
following subsection to obtain the inverse map.

\subsection{The inverse of the insertion map} \label{sec:insertion_reverse}

In this subsection we describe the inverse map of $\Phi_I$: 
\[
\Phi^{-1}_I : \A(n+1,r+1)\to \RAT^+(n+1,r+1).
\] 
For $\pi\in \A(n+1,r+1)$, we reverse the procedure of
\cref{alg:insertion} to construct the corresponding RAT $\Phi^{-1}_I(\pi)$ by
sequentially filling a RAT of shape given by $\pi$ according to
\cref{lem:shape} with
arrows. 

\begin{algorithm}[Inverse of insertion]\label{alg:ins_reverse}
  For an assembl\'ee $\pi \in \A(n+1,r+1)$, the RAT \( R=\Phi^{-1}_I(\pi) \) is
  constructed as follows, using the $\epsilon$-word $e(\pi)$.
  \begin{enumerate}
  \item We first construct the shape of \( R \). For each
    \( i\in [n+1] \), the \( i \)th step of the southeast border is
    vertical, diagonal, or horizontal if \( i \) is an ascent, a
    special descent, or a regular descent of \( e(\pi) \), respectively. Label the border edges with $[n+1]$ and let $d_1<\cdots<d_{r+1}$ be the labels of the diagonal strips.
  \item Suppose that \(R\) has \(m\) non-horizontal strips and let
    \(\ell_1<\cdots<\ell_m\) be their labels from right to left. Then
    for \(i=1,\ldots,m\), we do the following.
\begin{itemize}
\item [(i)] If strip \(\ell_i\) is vertical, let \( j \) be the
  element to the right of \( \ell_i \) in \( e(\pi) \). Let
  \( r_1,\dots,r_k,\ell_i \) be the maximal consecutive subsequence of
  \( e \) ending with \( \ell_i \) such that
  \( j<r_1 < \cdots <r_k<\ell_i \). Add an up-arrow in strip \( j \)
  and a left-arrow in each of strips \( r_1,\dots,r_k \) at the
  intersections of strip \( \ell_i \). (If \( j=\epsilon_t \), then
  strip \( j \) means the diagonal strip \( d_t \).) Delete
  \( r_1,\dots,r_k, \ell_i \) from \( e(\pi) \).
\item [(ii)] If strip \(\ell_i\) is diagonal, let
  \( r_1,\dots,r_k,\ell_i \) be the maximal consecutive subsequence of
  \( w \) ending with \( \ell_i \) such that
  \( r_1 < \cdots <r_k<\ell_i \). Add a left-arrow in each of the strips
  \( r_1,\dots,r_k \) at the intersections of strip \( \ell_i \).
  Delete \( r_1,\dots,r_k \) from \( e(\pi) \).
\end{itemize}
\item Define \(\Phi^{-1}_I(\pi)\) to be the resulting object \(R\). 
  \end{enumerate}
\end{algorithm}

\begin{example}\label{ex:psiI}
  Let \( \pi= [11\,\,8\,\,1\,\,4\,\,3\,\,7]\,[9\,\,10\,\,5\,\,2\,\,6]\, [12\,\,13]\in \A(13,3) \). Then
  \( d_1=1 \), \( d_2=6 \), \( d_3=13 \), and the \( \epsilon \)-word
  is
  \[
    e(\pi) = 11\,\, 8\,\, 1\,\, \epsilon_1\,\, 9\,\,10\,\,5\,\, 2\,\,6\,\,\epsilon_2\,\, 12\,\,13\,\,\epsilon_3\,\, 4\,\, 3\,\, 7.
  \]
  We reverse the process in \Cref{exa:2} using the above
  algorithm. Then \( \Phi^{-1}_I(\pi) \) is the RAT in \Cref{fig:flat}.
\end{example}

The maps $\Phi_I$ and $\Phi^{-1}_I$ are well-defined and are inverses of
each other by construction, giving the following theorem.

\begin{thm}
The map \(\Phi_I: \RAT^+(n+1,r+1)\to \A(n+1,r+1)\) is a bijection with inverse map
\(\Phi^{-1}_I : \A(n+1,r+1) \to \RAT^+(n+1,r+1)\) described in \Cref{alg:ins_reverse}.
\end{thm}

\subsection{An application of the insertion map}
\label{sec:z_1-ldots}

In this subsection, as an application of the insertion map
\( \Phi_I \), we prove the following generalization of the enumeration
formula \eqref{Z_nr}.

\begin{thm}\label{thm:MV_gen}
We have
\[
\sum_{R \in \RAT(n,r)} \alpha^{\frow(R)} \beta^{\fcol(R)}z_1^{a_1(R)}\cdots z_{r}^{a_{r}(R)} = 
{n \choose r} (\alpha+\beta+z_1+ \dots + z_{r})_{n-r},
\]
where \(a_i(R)\) is the number of arrows in the \(i\)th topmost diagonal
strip in \(R\) for \(1\leq i \leq r\).
\end{thm}

Recall the bijection \( R\mapsto R^+ \) from \( \RAT(n,r) \) to
\( \RAT^+(n+1,r+1) \). Under this bijection, a free column of \( R \)
corresponds to a column that has an up-arrow in diagonal strip
\( 1 \). Thus, by replacing the parameters
\( (\beta,z_1,\dots,z_r) \) by \( (z_1,z_2,\dots,z_{r+1}) \), we can
restate \Cref{thm:MV_gen} as follows:
\begin{equation}\label{eq:7}
\sum_{R \in \RAT^+(n+1,r+1)} \alpha^{\frow(R)} z_1^{a_1(R)}z_2^{a_2(R)}\cdots z_{r+1}^{a_{r+1}(R)} = 
{n \choose r} (\alpha+z_1+z_2+ \dots + z_{r+1})_{n-r}.
\end{equation}

To prove \eqref{eq:7}, we need a definition and a lemma. Let
\( B=b_1\cdots b_k \) be a word of integers. The \emph{decomposition}
of \(B\) is the triple \((X,Y,b_k)\) defined as follows:
\( X=b_1 \cdots b_t \) and \( Y= b_{t+1} \cdots b_{k-1} \), where
\( t \) is the largest integer with \( 1\le t\le k-1 \) such that
\( b_t>b_k \); if there is no such \( t \), then \( X=\emptyset \) and
\( Y=b_1 \cdots b_{k-1} \). Note that every integer in \( Y \) is
smaller than \( b_k \).

\begin{lemma}\label{lem:diag-strip}
  Let \( R\in \RAT^+(n+1,r+1) \). Suppose that the \( \epsilon \)-word of
  \( \pi = \Phi_I(R)\in \A(n+1,r+1) \) is given by
  \begin{equation}\label{eq:10}
    e(\pi)= B_1 \epsilon_1 B_2 \epsilon_2 \cdots B_{r+1} \epsilon_{r+1} B_{r+2}.
  \end{equation}
  For \( i\in [r+1] \), let \( (X_i,Y_i,d_i) \) be the decomposition
  of \( B_i \). Then the number of up- (resp.~left-) arrows in the
  diagonal strip \( d_i \) is equal to \( \RLmax(X_i) \)
  (resp.~\( \RLmin(Y_i) \)). In particular, the total number of arrows
  in the diagonal strip \( d_i \) is equal to \( \RLmax(X_i)+\RLmin(Y_i) \).
\end{lemma}

\begin{proof}
  Let \( 1=d_1<d_2<\dots<d_{r+1}\le n+1 \) be the labels of the
  diagonal strips of \( R \). Recall the construction of the word
  \( e(\pi) \) in \Cref{alg:insertion}. At each step of the process, \( e(\pi) \) is
  written in the form of \eqref{eq:10}, where each word \( B_i \) may be empty.
  At each stage, if $d_i\in B_i$, we write \( B_i = X_iY_id_i \), where
  \( (X_i,Y_i,d_i) \) is the decomposition of \( B_i \), and if $d_i\not\in B_i$ we write \( B_i=X_i \).
 
  Fix \( i\in [r+1] \), and let \(c_1>\cdots>c_m\) be the labels of
  the columns in \( R \) that contain an up-arrow in diagonal strip
  \( d_i \). By the construction in \Cref{alg:insertion}, we have
  \( B_i=X_i=\emptyset \) before we process strip \( c_1 \). At the
  stage of strip \( c_1 \), \( B_i=X_i \) becomes a nonempty word,
  where \( c_1 \) is the last element and also the largest element in
  \( X_i \), i.e., \( c_1 \) is the only RL-maximum of \( X_i \).
  During the stages between strips \(c_1\) and \( c_2 \), \( B_i=X_i \) may be modified: with each modification,
  \( c_1 \) remains at the end of \( X_i \), and only
  integers smaller than \( c_1 \) may be inserted into \( X_i \), thus maintaining $c_1$ as the only RL-maximum of $X_1$. At the stage of strip
  \( c_2 \), we add a word whose last element is \( c_2 \) and whose
  remaining elements are less than \( c_2 \) to the end of
  \( B_i=X_i \). Then \( c_1 \) and \( c_2 \) are the only RL-maxima
  of \( X_i \). Continuing in this way, after strip \( c_m \), we
  obtain that \( B_i=X_i \) and its RL-maxima are \( c_1,\dots,c_m \).
  Note that \( c_m>d_i \). Now let \(r_1<\cdots<r_\ell\) be the labels
  of the rows in \( R \) that contain a left-arrow in diagonal strip
  \( d_i \). Then, at the stage of strip \( d_i \), we insert
  \( r_1\, \cdots\, r_\ell\, d_i \) at the end of \( B_i \) so that
  \( B_i=X_iY_id_i \) and \( Y_i=r_1 \cdots r_\ell \). Note that
  \( r_1,\dots,r_\ell \) are the RL-minima of \( Y_i \). After this
  stage, the word \( B_i=X_iY_id_i \) may be modified further, but with each modification, the
  RL-maxima of \( X_i \) remain \( c_1,\dots,c_m \) and the
  RL-minima of \( Y_i \) remain \( r_1,\dots,r_\ell \). This
  completes the proof.
\end{proof}

We are now ready to prove \eqref{eq:7}, which is equivalent to \Cref{thm:MV_gen}.

\begin{proof}[Proof of \eqref{eq:7}]
  Let \( R\in \RAT^+(n+1,r+1) \) and
  \( \pi = \Phi_I(R)\in \A(n+1,r+1) \). Then we can write the
  \( \epsilon \)-word \( e(\pi) \) of \( \pi \) as
  \[
    e=e(\pi) = B_1\, \epsilon_1\, B_2\, \epsilon_2\, \cdots\, B_{r+1}\, \epsilon_{r+1}\, B_{r+2}.
  \]
  For \( i\in [r+1] \), let \((X_i,Y_i,d_i)\) be the decomposition of
  \(B_i\), so that
  \begin{equation}\label{eq:9}
    e= X_1\,Y_1\,d_1\, \epsilon_1\, X_2\,Y_2\,d_2\, \epsilon_2\, \cdots\, X_{r+1}\,Y_{r+1}\,d_{r+1}\, \epsilon_{r+1}\, B_{r+2}.
  \end{equation}
  By \Cref{thm:CN}, we have
  \[
    \frow(R) = \RLmin(e) -r-1 = \RLmin(B_{r+2}).
  \]
  Hence, by \Cref{lem:diag-strip}, it suffices to show that
\begin{equation}\label{eq:EA_gf}
  \sum_{e \in E} \alpha^{\RLmin(B_{r+2})} \prod_{i=1}^{r+1}  z_i^{\RLmax(X_i)+\RLmin(Y_i)}= 
\binom{n}{r} (\alpha+z_1+ \dots + z_{r+1})_{n-r},
\end{equation}
where \( E=\{e(\pi):\pi\in\A(n+1,r+1)\} \)
with the notation in \eqref{eq:9} used for \( e\in E \).

For a permutation \( u=u_1 \cdots u_k \) on a set
\( \{b_1 < \cdots < b_k\} \), let \( u'=u'_1 \cdots u'_k \), where
\( u'_i = b_{k+1-j} \) if \( u_i=b_j \). Then
\( \RLmin(u) = \RLmax(u') \).
Define 
\begin{equation}\label{eq:8}
  \tilde{e}= U_1\,d_1\, \epsilon_1\, U_2\,d_2\, \epsilon_2\, \cdots\, U_{r+1}\,d_{r+1}\, \epsilon_{r+1}\, U_{r+2},
\end{equation}
where \( U_i=X_iY'_i \) for \( i\in [r+1] \) and
\( U_{r+2}=B'_{r+2} \).
Note that 
\[
  \RLmax(X_i)+\RLmin(Y_i) = \RLmax(X_i)+\RLmax(Y'_i) = \RLmax(U_i),
\]
and \( \RLmin(B_{r+2}) = \RLmax(U_{r+2}) \). Since the map
\( e\mapsto \tilde{e} \) is an involution on \( E \) with these
properties, \eqref{eq:EA_gf} is equivalent to
\begin{equation}\label{eq:EA_gf2}
  \sum_{\tilde{e} \in E} \alpha^{\RLmax(U_{r+2})} \prod_{i=1}^{r+1}  z_i^{\RLmax(U_i)}= 
\binom{n}{r} (\alpha+z_1+ \dots + z_{r+1})_{n-r},
\end{equation}
where we use the notation in \eqref{eq:8} for \( \tilde{e}\in E \).

Now we prove \eqref{eq:EA_gf2} by giving an algorithm to construct the
\( \epsilon \)-words \( \tilde{e}\in E \) in the form \eqref{eq:8}.
First, choose the integers \( 1=d_1<d_2<\dots<d_{r+1}\le n+1 \), which
can be done in \( \binom{n}{r} \) ways. We initialize \( \tilde{e} \)
as in \eqref{eq:8}, where each \( U_i \) is the empty word. Let
\( k_1>\dots>k_{n-r} \) be the integers in
\( [n+1]\setminus\{d_1,\dots,d_{r+1}\} \). For each
\( i=1,2,\dots,n-r \), we insert \( k_i \) into one of
\( U_1,\dots,U_{r+2} \), which can be done in \( r+i+1 \) ways. For a
fixed \( j\in [r+2] \), \( \RLmax(U_j) \) is increased by \( 1 \) if
and only if \( k_i \) is added at the end of \( U_j \). This shows
that the insertion of \( k_i \) contributes the factor
\( \alpha+z_1 + \cdots + z_{r+1} + i-1 \) to the left-hand side of
\eqref{eq:EA_gf2}. Therefore, we obtain \eqref{eq:EA_gf2}, which
completes the proof.
\end{proof}

\subsection{A zigzag map from RAT to assembl\'ees}
\label{sec:phi_Z}

In this subsection, we give a zigzag map \( \Phi_Z \) on RAT, which
generalizes the zigzag map \( \Phi_Z^{\AT} \) on alternative tableaux
introduced in \cref{def:zigzag AT}. We then show that $\Phi_Z$ agrees
with the insertion map $\Phi_I$ up to a simple transformation,
paralleling \cref{eq:CK}.

Since it is most natural to define a zigzag map so that its images are
permutations, we will represent assembl\'ees as permutations. To this
end, we make the following definition.
\begin{defn}\label{def:A-epsilon}
Define \( \A^\epsilon(n+1,r+1) \) to be the set of
permutations
\( \pi \in S_{\{\epsilon_1,\ldots,\epsilon_{r+1},1,\ldots,n+1\}} \)
satisfying the following conditions:
\begin{enumerate}
\item Every cycle of \( \pi \) contains at most one \( \epsilon_i \).
\item If \( d_i = \pi^{-1}(\epsilon_i) \), then
  \( 1=d_1 < \cdots < d_{r+1} \).
\end{enumerate}
\end{defn}
Recalling \cref{def:phi}, as $\epsilon_1<\cdots<\epsilon_{r+1}$ are necessarily RL-minima of $e(\pi)$ for $\pi\in\A(n+1,r+1)$, we obtain that
\[
  \phi\circ e:\A(n+1,r+1)\to \A^\epsilon(n+1,r+1)
\]
is a bijection whose inverse is \(e^{-1}\circ\phi^{-1}\). See \cref{ex:phi of e}.

\begin{example}\label{ex:phi of e}
  Let
  \( \pi= [11\,\,8\,\,1\,\,4\,\,3\,\,7]\,[9\,\,10\,\,5\,\,2\,\,6]\,
  [12\,\,13]\in \A(13,3) \). Then the \( \epsilon \)-word of \( \pi \)
  is
  \[
    e(\pi) = 11\, 8\, 1\, \epsilon_1\, 9\,10\,5\, 2\,6\,\epsilon_2\,
    12\,13\,\epsilon_3\, 4\, 3\, 7.
  \]
  The RL-minima of \( e(\pi) \) are
  \( \epsilon_1,\epsilon_2,\epsilon_3, 3, 7 \). Hence, we have
 \[
   \phi(e(\pi))=(11, 8, 1, \epsilon_1)
   (9,10,5, 2,6,\epsilon_2)
   (12,13,\epsilon_3) (4, 3) (7) \in \A^\epsilon(13,3).
  \] 
\end{example}

For a tile \( \tau \) with its edges $\{N(\tau),E(\tau),S(\tau),W(\tau)\}$ defined as in \cref{fig:NESW}, we will identify the N-side \( N(\tau) \), S-side \( S(\tau) \), E-side \( E(\tau) \), and W-side \( W(\tau) \) of \( \tau \) with their midpoints, respectively.

\begin{defn}[Zigzag paths]
  \label{def:zigzag}
  Let \( T\in \RAT^+ \). Let \( U \) be a set of cells in \( T \). A
  cell of \( T \) is called a \emph{turning cell} (with respect to
  \( U \)) if it is contained in \( U \), and a \emph{straight cell}
  (with respect to \( U \)) otherwise. A \emph{zigzag path} (with
  respect to \( U \)) is a sequence \( (u_0,u_1,\dots,u_k) \) of
  points satisfying the following conditions:
  \begin{itemize}
  \item The first point \( u_0 \) is the midpoint of the north or west
    side of a tile on the northwest border of \( T \), and the last
    point \( u_k \) is the midpoint of the south or east side of a
    tile on the southeast border of \( T \).
  \item For each \( i\in [k] \), the consecutive points \( u_{i-1} \)
    and \( u_{i} \) lie on two sides of a tile \( \tau \). If
    \( \tau \) is a straight cell, then either
    \( (u_{i-1},u_i) = (N(\tau),S(\tau)) \) or
    \( (u_{i-1},u_i) = (W(\tau),E(\tau)) \). If \( \tau \) is a
    turning cell, then either \( (u_{i-1},u_i) = (N(\tau),E(\tau)) \)
    or \( (u_{i-1},u_i) = (W(\tau),S(\tau)) \).
  \end{itemize}
  In particular, if \( U \) is the set of cells of \( T \) containing
  an arrow, we call a zigzag path with respect to \( U \) an
  \emph{arrow-zigzag path}. If \( U \) is the set of cells of \( T \)
  that contain an up-arrow or that are free cells, we call a zigzag
  path with respect to \( U \) an \emph{up-free-zigzag path}.
\end{defn}

We can visualize a zigzag path as a path moving southeast that enters from the
northwest border and exits from the southeast border, and changes direction at each turning cell. See \Cref{fig:zigzag2} for an example of an arrow-zigzag path. In this section, we only consider arrow-zigzag paths. Up-free-zigzag paths will be considered in
\Cref{sec:zigzag-RPT}.

\begin{figure}
  \centering
\begin{tikzpicture}[scale=1.5]
    \short{6}{7} \short{5}{7} \short{4}{7} \short{3}{7} \short{2}{7} \cell{6}{6} \cell{5}{6} \tall{4}{5} \cell{3}{5} \cell{2}{5} \cell{1}{5} \tall{0}{4} \cell{6}{5} \cell{5}{5} \tall{4}{4} \cell{3}{4} \cell{2}{4} \cell{1}{4} \tall{0}{3} \short{3}{6} \short{2}{6} \short{1}{6} \cell{3}{3} \cell{2}{3} \cell{1}{3} \tall{0}{2} \cell{2}{2} \cell{1}{2} \tall{0}{1} \tall{0}{0} \node at (0.600000000000000, 0.300000000000000) {\tiny 13};
\node at (1.20000000000000, 1.50000000000000) {\tiny 12};
\node at (1.50000000000000, 1.80000000000000) {\tiny 11};
\node at (2.50000000000000, 1.80000000000000) {\tiny 10};
\node at (3.20000000000000, 2.50000000000000) {\tiny 9};
\node at (3.50000000000000, 2.80000000000000) {\tiny 8};
\node at (4.20000000000000, 3.50000000000000) {\tiny 7};
\node at (4.60000000000000, 4.30000000000000) {\tiny 6};
\node at (5.50000000000000, 4.80000000000000) {\tiny 5};
\node at (6.50000000000000, 4.80000000000000) {\tiny 4};
\node at (7.20000000000000, 5.50000000000000) {\tiny 3};
\node at (7.20000000000000, 6.50000000000000) {\tiny 2};
\node at (7.60000000000000, 7.30000000000000) {\tiny 1};
 \slantUpArr{2}{7} \slantUpArr{4}{7} \slantUpArr{2}{6} \upArr{5}{6} \upArr{6}{5} \slantLeftArr{4}{5} \leftArr{2}{2} \slantLeftArr{0}{0}
 \draw[->] (9,4) -- (10,4);
 \node at (9.5,4.5) {flat};
 \phantom{\node at (0.500000000000000, -0.200000000000000) {\tiny 13};}
  \draw[very thick,red,dashed] (1.5,6.5) to (2.5,6.5) to[out=0, in=45] (2.5,6) to (2.5,3) to[out=270, in=180] (3,2.5);
  \draw[very thick,purple] (4.5,8) to (3.5,7) to[out=225, in=180] (3.5,6.5) to (4.5,6.5) to[out=270, in=225] (5,6.5)
  to[out=0,in=90] (5.5,6) to (5.5,5);
  \draw[very thick,blue,dotted] (0,4.5) to (1,5.5) to (4,5.5) to[out=45,in=90] (4.5,5.5) to (4.5,4.5);
  \node at (-.3,.5) {\tiny 12};
  \node at (-.3,1.5) {\tiny 9};
  \node at (-.3,2.5) {\tiny 7};
  \node at (-.3,3.5) {\tiny 3};
  \node at (-.3,4.5) {\tiny 2};
  \node at (.2,5.7) {\tiny \( 13 \)};
  \node at (1.2,6.7) {\tiny \( 6 \)};
  \node at (2.2,7.7) {\tiny \( 1 \)};
  \node at (3.5,8.3) {\tiny 11};
  \node at (4.5,8.3) {\tiny 10};
  \node at (5.5,8.3) {\tiny 8};
  \node at (6.5,8.3) {\tiny 5};
  \node at (7.5,8.3) {\tiny 4};
\end{tikzpicture}
\qquad 
\begin{tikzpicture}[scale=1.5]
\cell{7}{7} \cells{6}{7} \cells{5}{7} \cell{4}{7} \cells{3}{7} \cells{2}{7} \cells{1}{7} \cell{0}{7} \cell{7}{6} \cell{6}{6} \cell{5}{6} \cell{4}{6} \cells{3}{6} \cells{2}{6} \cells{1}{6} \cell{0}{6} \cell{7}{5} \cell{6}{5} \cell{5}{5} \cell{4}{5} \cell{3}{5} \cell{2}{5} \cell{1}{5} \cell{0}{5} \cell{6}{4} \cell{5}{4} \cellt{4}{4} \cell{3}{4} \cell{2}{4} \cell{1}{4} \cellt{0}{4} \cell{6}{3} \cell{5}{3} \cellt{4}{3} \cell{3}{3} \cell{2}{3} \cell{1}{3} \cellt{0}{3} \cell{3}{2} \cell{2}{2} \cell{1}{2} \cellt{0}{2} \cell{2}{1} \cell{1}{1} \cellt{0}{1} \cellt{0}{0} \node at (0.500000000000000, -0.200000000000000) {\tiny 13};
\node at (1.20000000000000, 0.500000000000000) {\tiny 12};
\node at (1.50000000000000, 0.800000000000000) {\tiny 11};
\node at (2.50000000000000, 0.800000000000000) {\tiny 10};
\node at (3.20000000000000, 1.50000000000000) {\tiny 9};
\node at (3.50000000000000, 1.80000000000000) {\tiny 8};
\node at (4.20000000000000, 2.50000000000000) {\tiny 7};
\node at (4.50000000000000, 2.80000000000000) {\tiny 6};
\node at (5.50000000000000, 2.80000000000000) {\tiny 5};
\node at (6.50000000000000, 2.80000000000000) {\tiny 4};
\node at (7.20000000000000, 3.50000000000000) {\tiny 3};
\node at (7.20000000000000, 4.50000000000000) {\tiny 2};
\node at (7.50000000000000, 4.80000000000000) {\tiny 1};
\node at (8.40000000000000, 5.50000000000000) {\tiny \( \epsilon_3 \)};
\node at (8.40000000000000, 6.50000000000000) {\tiny \( \epsilon_2 \)};
\node at (8.40000000000000, 7.50000000000000) {\tiny \( \epsilon_1 \)};
 \upArr{7}{7} \upArr{3}{7} \upArr{1}{7} \upArr{4}{6} \upArr{2}{6} \upArr{0}{5} \upArr{5}{4} \upArr{6}{3} \leftArr{4}{4} \leftArr{2}{1} \leftArr{0}{0}
 \draw[very thick,purple] (2.5,8) to (2.5,7) to[out=270, in=180] (3,6.5) to (4,6.5) to[out=0, in=90] (4.5,6)
 to (4.5,5) to[out=270,in=180] (5,4.5) to[out=0,in=90] (5.5,4) to (5.5,3);
  \draw[very thick,red,dashed] (0,6.5) to (2,6.5) to[out=0, in=90] (2.5,6) to (2.5,2) to[out=270, in=180] (3,1.5);
  \draw[very thick,blue,dotted] (0,4.5) to (4,4.5) to[out=0,in=90] (4.5,4) to (4.5,3);
  \node at (-.3,.5) {\tiny 12};
  \node at (-.3,1.5) {\tiny 9};
  \node at (-.3,2.5) {\tiny 7};
  \node at (-.3,3.5) {\tiny 3};
  \node at (-.3,4.5) {\tiny 2};
  \node at (-.4,5.5) {\tiny \( \epsilon_3 \)};
  \node at (-.4,6.5) {\tiny \( \epsilon_2 \)};
  \node at (-.4,7.5) {\tiny \( \epsilon_1 \)};
  \node at (.5,8.3) {\tiny 13};
  \node at (1.5,8.3) {\tiny 11};
  \node at (2.5,8.3) {\tiny 10};
  \node at (3.5,8.3) {\tiny 8};
  \node at (4.5,8.3) {\tiny 6};
  \node at (5.5,8.3) {\tiny 5};
  \node at (6.5,8.3) {\tiny 4};
  \node at (7.5,8.3) {\tiny 1};
\end{tikzpicture}
\caption{The left diagram show the arrow-zigzag path from \( 10 \) to
  \( 5 \), the arrow-zigzag path from \( 2 \) to \( 6 \), and the arrow-zigzag
  path from \( 6 \) to \( 9 \). The right diagram show the arrow-zigzag path
  from \( 10 \) to \( 5 \), the arrow-zigzag path from \( 2 \) to \( 6 \),
  and the arrow-zigzag path from \( \epsilon_2 \) to \( 9 \).}
  \label{fig:zigzag}
\end{figure}
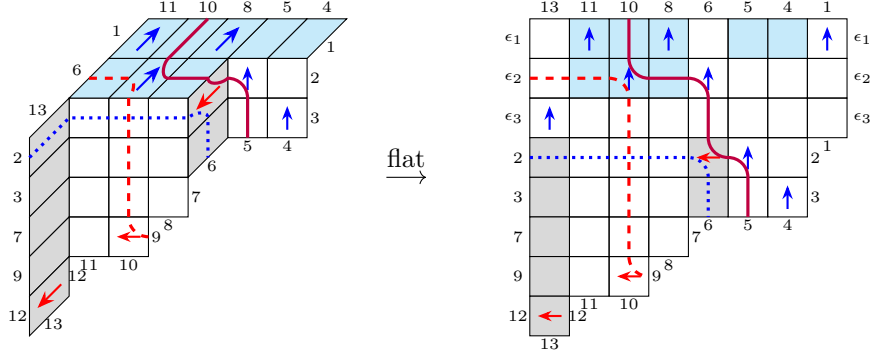

\begin{defn}[Arrow-zigzag map]
The \emph{arrow-zigzag map}
\[
  \Phi_Z: \RAT^+(n+1,r+1)\to\A^\epsilon(n+1,r+1)
\]
is defined as follows. Let \( R\in \RAT^+(n+1,r+1) \). Suppose that
\( 1=d_1 < \cdots < d_{r+1} \) are the labels of diagonal strips.
First, we set \( \sigma \) to be the permutation on \( [n+1] \) such
that for each \( i\in [n+1] \), we have \( \sigma(i) = j \), where the
arrow-zigzag path starting from strip \( i \) exits at strip \( j \).
Then write \( \sigma \) in cycle notation and insert \( \epsilon_i \)
after \( d_i \) for each \( i\in [r+1] \). We define \( \Phi_Z(R) \)
to be the resulting permutation \( \sigma\in S_{\{\epsilon_1\ldots,\epsilon_{r+1},1,\ldots,n+1\}} \).
\end{defn}

Note that \cref{thm:CN} combined with \cref{lem:Z-flat} below establishes that $\Phi_Z(R)\in\A^\epsilon(n+1,r+1)$: each cycle of $\Phi_Z(R)$ corresponds to a free row or a diagonal strip in $R$, satisfying the conditions of \cref{def:A-epsilon}. See \cref{exa:3}.

\begin{example}\label{exa:3}
  Let \( R \) be the RAT in \Cref{fig:zigzag}. The labels of the
  diagonal strips are \( (d_1,d_2,d_3) = (1,6,13) \). Using
  arrow-zigzag paths, we first set
\[
\sigma=  \begin{pmatrix}
 1  & 2 & 3 & 4 & 5 & 6 & 7 & 8 & 9  & 10 & 11 & 12 & 13\\
 11 & 6 & 4 & 3 & 2 & 9 & 7 & 1 & 10 & 5  & 8  & 13 & 12\\
\end{pmatrix}.
\]
In cycle notation,
\[
\sigma= (11,8,1)(6,9,10,5,2)(4,3)(7)(13,12).
\]
Inserting \( \epsilon_i \) after each \( d_i \), we obtain
\[
\Phi_Z(R) = (11,8,1,\epsilon_1)(9,10,5,2,6,\epsilon_2)(12,13,\epsilon_3)(4,3)(7) \in \A^\epsilon(13,3).
\]
Then the corresponding assembl\'{e}e is
\[
e^{-1}(\phi^{-1}(\Phi_Z(R))) = [11\,\,8\,\,1\,\,4\,\,3\,\,7]\,[9\,\,10\,\,5\,\,2\,\,6]\,[12\,\,13]\in \A(13,3),
\]
which coincides with the output $\Phi_I(R)$ in \cref{exa:2}.
\end{example}

The arrow-zigzag map \( \Phi_Z \) is naturally considered as the zigzag map on alternative tableaux via the flattening map, as shown in \Cref{fig:zigzag}, giving the following lemma.

\begin{lemma}\label{lem:Z-flat}
  Let \( R\in \RAT^+(n+1,r+1) \). Then
  \[
    \Phi_Z(R) = \Phi^{\AT}_Z(\flat(R)).
  \]
\end{lemma}

Viewing both the insertion map \(\Phi_I\) and the arrow–zigzag map \(\Phi_Z\) as maps on alternative tableaux after flattening, we conclude that the two maps are equivalent up to the transformation \(\phi\); see \Cref{exa:3}. Here, we apply the map \(e\) to \(\Phi_I(R)\) in order to represent it as a permutation in \( S_{\{\epsilon_1,\ldots,\epsilon_{r+1},1,\ldots,n+1\}} \).

\begin{prop}\label{prop:I is Z}
  For \(R\in\RAT^+(n+1,r+1)\), we have 
  \[
    \Phi_Z(R) = \phi(e(\Phi_I(R))).
  \]
\end{prop}

\begin{proof}
  By \Cref{lem:Z-flat}, \cref{eq:CK}, and \Cref{lem:I-flat},
  \[
    \Phi_Z(R) = \Phi^{\AT}_Z(\flat(R)) = \phi(\Phi^{\AT}_I(\flat(R))) = \phi(e(\Phi_I(R))).
  \]
\end{proof}

\subsection{The fusion-exchange map from assembl\'ees to RAT}\label{sec:FE}

We now describe a variant of the fusion-exchange (FE) map
\[
\ferat:\A(n+1,r+1)\to\RAT^+(n+1,r+1),
\]
which was introduced in \cite{Mandelshtam2018}. Our definition is the
transposed version of the original. Hence, in our convention,
assemblée blocks are ordered by their heads rather than their ends.
Furthermore, we add a diagonal edge to the northeast corner of the
rhombic alternative tableau, working in \(\RAT^+(n+1,r+1)\) rather
than \(\RAT(n,r)\). Apart from this modification, the fusion-exchange
rules are unchanged. See the proof of \cref{lem:labels} for more details on the correspondence between the two algorithms.

\begin{defn}
  Let \(a=a_1\cdots a_n\in S_n\). An entry \(a_i\in[n-1]\) is an
  \emph{inverse descent} if \(a^{-1}(a_i+1)<i\), or equivalently, if
  \(a_i+1\) is to the left of \( a_i \) in \(a\). 
  We write \(\iDes(a)\) for the set of inverse descents of \(a\).
  By convention, \(n\not\in\iDes(a)\).
\end{defn}

Recall the definition of the map \( f \) in \Cref{def:f}. 
Let \(\pi\in\A(n+1,r+1)\) be an assembl\'ee with canonical permutation \(f(\pi)=a_1\ldots a_{n+1}\in S_{n+1}\). We construct the word \(w=w_1\cdots w_{n+1}\in\{2,1,0\}^{n+1}\) by 
\begin{equation}\label{eq:ides}
w_i=\begin{cases}
1, & a_i\in\heads(\pi),\\
0, & a_i\in\iDes(f(\pi))\text{ and }a_i\not\in\heads(\pi),\\
2, & a_i\notin\iDes(f(\pi))\text{ and }a_i\notin\heads(\pi),
\end{cases}
\qquad i\in[n+1].
\end{equation}
In \cref{def:FE} below, we will construct a filling \(R^+=\ferat(\pi)\) on the rhombic diagram
\(\Gamma_w\) with its southeast border labeled by \( f(\pi)\) by iteratively
passing the edge labels from the southeast border of \(\Gamma_w\) to
the northwest border along the strips perpendicular to the edges at
which the labels originate. Certain interactions of labels will result
in \emph{fusing} or \emph{exchanging} labels along their trajectories,
thus giving the name of fusion-exchange. The filling will be
determined by the interactions of the labels within each cell of
\(\Gamma_w\).

A \emph{label set} is a (possibly empty) set \(I\subseteq [n]\) of
consecutive integers. If \(I,J\subseteq [n]\) are two disjoint
nonempty labels, then we define the relation \(I\prec J\) if
\(I\cup J\) is a label set and \( 1+\max I = \min J \). By convention,
\(\emptyset\not \prec I\) for any label set \(I\).

\begin{defn}[Fusion-exchange algorithm]\label{def:FE}
Let \(\pi\in\A(n+1,r+1)\). Then \(\ferat(\pi)\) is constructed as follows.
\begin{itemize} 
\item[(1)] For the canonical permutation \(f(\pi)=a_1\cdots a_{n+1}\),
  let \(w\in\{2,1,0\}^{n+1}\) be the word obtained according to
  \eqref{eq:ides}, and let \(D\) be the rhombic diagram \( \Gamma_w \)
  whose \(i\)th southeast border edge from northeast to southwest is
  labeled by the label set \(\{a_i\}\).
\item[(2)] We iteratively pass the label sets along trajectories
  determined by the strips of \(D\) to label all the edges contained
  in \(D\), as follows. Let \(\tau\) be a cell such that
  \(\{E(\tau),S(\tau)\}=\{I,J\}\) and that \( N(\tau) \) and
  \( W(\tau) \) have no labels. If \(I\prec J\) and \( I \) is not the
  label of a diagonal edge, the label set \(I\cup J\) is passed to the
  edge of \(\tau\) opposite the one containing \(J\) and \(\emptyset\)
  is passed to the edge of \(\tau\) opposite to the one containing
  \(I\), and an arrow pointing towards the label set \(\emptyset\) is
  placed inside \(\tau\). Concretely, this yields either
  \((S(\tau),E(\tau),N(\tau),W(\tau))=(I, J, \emptyset,I\cup J)\) with
  an up-arrow, or
  \((S(\tau),E(\tau),N(\tau),W(\tau))=(J, I, I\cup J, \emptyset)\)
  with a left-arrow, as shown in \cref{fig:fusion}. Otherwise, i.e.,
  if \(I\not \prec J\) or \( I \) is the label of a diagonal edge,
  then the labels simply pass through each other, and are transferred
  to the opposite edges of \(\tau\) along their trajectories.
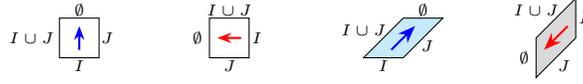
\begin{figure}
  \centering
\begin{tikzpicture}[scale=1.5]
   \cell{0}{0} \upArr{0}{0}
  \node at (0.5, -0.2) {\tiny \(I\)};
  \node at (0.5, 1.2) {\tiny \(\emptyset\)};
  \node at (1.2, 0.5) {\tiny \(J\)};
  \node at (-.7, 0.5) {\tiny \(I\cup J\)};
\end{tikzpicture}
\qquad 
\begin{tikzpicture}[scale=1.5]
  \cell{0}{0} \leftArr{0}{0}
  \node at (0.5, -0.2) {\tiny \(J\)};
  \node at (0.5, 1.2) {\tiny \(I\cup J\)};
  \node at (1.2, 0.5) {\tiny \(I\)};
  \node at (-.3, 0.5) {\tiny \(\emptyset\)};
\end{tikzpicture}
\qquad 
\begin{tikzpicture}[scale=1.5]
  \short{0}{0} \slantUpArr{0}{0}
  \node at (1.5, 1.2) {\tiny \(\emptyset\)};
  \node at (0.5, -0.2) {\tiny \(I\)};
  \node at (1.6, 0.3) {\tiny \(J\)};
  \node at (0, 0.7) {\tiny \(I\cup J\)};
\end{tikzpicture}
\qquad 
\begin{tikzpicture}[scale=1.5]
  \tall{0}{0} \slantLeftArr{0}{0}
  \node at (0, 1.7) {\tiny \(I\cup J\)};
  \node at (0.7, 0.3) {\tiny \(J\)};
  \node at (1.2, 1.5) {\tiny \(I\)};
  \node at (-.3, 0.5) {\tiny \(\emptyset\)};
\end{tikzpicture}
\caption{All possible configurations in which fusion occurs are shown.
  Here, each cell \(\tau\) satisfies \(\{E(\tau),S(\tau)\}=\{I,J\}\)
  for label sets \(I\prec J\) such that \( I \) is not the label of a
  diagonal edge.}
  \label{fig:fusion}
\end{figure}

\item[(3)] The procedure is completed once all edges in \(D\) are
  labeled, i.e., the labeled edges reach the northwest border. Define
  \(\ferat(\pi)\) to be the RAT obtained from this procedure.
\end{itemize}
\end{defn}

\begin{lemma}\label{lem:labels}
  Let \(\pi\in\A(n+1,r+1)\) and \( R= \ferat(\pi)\). Then the following hold:
  \begin{enumerate}
  \item The label of every diagonal step in the northeast border is nonempty.
  \item The label of a vertical step in the northeast border is nonempty
    if and only if the corresponding row is free.
  \item The label of every horizontal step in the northeast border is \( \emptyset \).
  \item If \( I_1,\dots,I_\ell \) are the nonempty labels of the
  diagonal and vertical edges in the northwest border from north to
  south, then \( I_1 \prec \cdots \prec I_\ell \).
  \end{enumerate}
\end{lemma}

\begin{proof}
  By the construction of \( \ferat \), the statements (1) and (2) are
  immediate. To prove (3) and (4), we compare our map \(\ferat\) with
  the fusion-exchange map in \cite{Mandelshtam2018}, which we refer to
  as \(\ferat^{MV}\). Let \( R_1 \) be the RAT obtained from \( R \)
  by removing the first diagonal strip. Let \( R'_1 \) be the
  ``transpose'' of \( R_1 \), that is, the RAT obtained by reflecting
  \(R_1\) across the line \(y=-x\), so that horizontal and vertical
  edges, as well as left- and up-arrows, are
  interchanged.\footnote{The tiling of the transpose of \( R_1 \) is no
    longer a maximal tiling that we consider in this paper; see
    \Cref{rem:1}. However, in \cite{Mandelshtam2018}, the
    fusion-exchange map is defined for arbitrary tilings.} Then, by
  the definitions of the maps \( \ferat \) and \(\ferat^{MV}\), we
  have \( R'_1 = \ferat^{MV}(\pi) \).

  Note that \( I_2,\dots,I_\ell \) are the nonempty labels of the
  diagonal and vertical edges in the northwest border from north to
  south in \( R_1 \). Let \( J_1,\dots,J_t \) be the nonempty
  labels of the horizontal edges in the northwest border from west to
  east in \( R_1 \). Let \( \{d\} \) be the head of the first block in
  \( \pi \). Then, by applying \cite[Lemmas~3.4,~3.5]{Mandelshtam2018}
  to \( R'_1 = \ferat^{MV}(\pi) \), we obtain
  \begin{equation}\label{eq:16}
    J_1 \prec \cdots \prec J_t \prec \{d\} \prec I_2 \prec \cdots \prec I_\ell.
  \end{equation}

  Now observe that \( R=\ferat(\pi) \) is constructed from \( R_1 \)
  by applying the fusion-exchange algorithm to the cells in the first
  diagonal strip. By \eqref{eq:16}, the label set \( \{d\} \) is
  merged with \( J_t, J_{t-1},\dots,J_1 \) in this order, which gives
  \( I_1=J_1 \cup \cdots \cup J_t \cup \{d\} \prec I_2 \). Moreover,
  this process also creates an up-arrow in each vertical strip without
  an up-arrow. Hence, we obtain the statements (3) and (4).
\end{proof}

By \cref{lem:labels}, every vertical strip in $\ferat(\pi)$ contains
an up-arrow, and thus $\ferat(\pi)\in\RAT^+(n+1,r+1)$.

We now compare the zigzag map
\( \Phi_Z:\RAT^+(n+1,r+1)\to \A^\epsilon(n+1,r+1) \) and the inverse of the fusion-exchange map
\( \ferat^{-1}: \RAT^+(n+1,r+1)\to \A(n+1,r+1) \) to show they are equivalent up to some simple transformations. 

\begin{defn}\label{def:iota}
  For \(\pi\in\A(n+1,r+1)\), define the word \(\iota(\pi)\) as
  follows. Let \(\sigma=f(\pi)\in S_{n+1}\), let
  \(h_1<\cdots<h_{r+1}\) be the block heads of \(\pi\), and let
  \(\sigma^{-1}=u_1\cdots u_{n+1}\). Define \(\iota(\pi)\)
  to be the word obtained from \( u_1 \cdots u_{n+1} \)
  by inserting \(\epsilon_i\) immediately after \(u_{h_i}\) for each
  \(i\in[r+1]\).
\end{defn}

Note that, by construction, we have
\( 1=u_{h_1}<u_{h_2}<\cdots<u_{h_{r+1}}\) and each \( \epsilon_i \) is
placed immediately after \( u_i \). Therefore, there is an
assembl\'{e}e whose \( \epsilon \)-word is \( \iota(\pi) \), and we
can define the map
\[e^{-1}\circ\iota:\A(n+1,r+1)\to\A(n+1,r+1).\] Observe that
\(e^{-1}\circ\iota\) inverts the descent structure in the resulting
assembl\'ee: the positions of the inverse descents of \(f(\pi)\)
become descents in \(f(e^{-1}(\iota(\pi)))\), and the heads of
\(\pi\), encoding the block delineations of \(e^{-1}(\iota(\pi))\),
become special descents. Moreover, the following proposition shows that
\(e^{-1}\circ\iota\) coincides with the classical inverse on
permutations under the natural identification
\(\A(n+1,1)\cong S_{n+1}\) via the map \(f\).

\begin{prop}
For \(\pi\in\A(n+1,1)\), we have
\[
f(e^{-1}(\iota(\pi))) = f(\pi)^{-1}.
\]
\end{prop}

\begin{proof}
Write \(\sigma=f(\pi)\). The \(\epsilon\)-word \(\iota(\pi)\) is given by \(\sigma^{-1}\) with \(\epsilon_1\) inserted after the entry \(\sigma_1\), and thus \(e^{-1}(\iota(\pi))\in\A(n+1,1)\) is the assembl\'ee obtained by deleting the \(\epsilon_1\). Thus \(f(e^{-1}(\iota(\pi)))=\sigma^{-1}\) as claimed.
\end{proof}

\begin{example}
  Let \(\pi=[2\,8\,5]\,[3\,9\,7\,1]\,[4]\,[6]\in\A(9,4)\) and  
  \(\sigma=f(\pi)\). Then \( \sigma=2\,8\,5\,3\,9\,7\,1\,4\,6\in S_9 \) and 
  \(\sigma^{-1}=7\,1\,4\,8\,3\,9\,6\,2\,5\). As
  \(\heads(\pi)=\{2,3,4,6\}\), we place
  \(\epsilon_1,\epsilon_2,\epsilon_3,\epsilon_4\) after
  \(\sigma^{-1}_2,\sigma^{-1}_3,\sigma^{-1}_4,\sigma^{-1}_{6}\),
  respectively, in \( \sigma^{-1} \) to obtain
\[
\iota(\pi)=7\,1\,\epsilon_1\,4\,\epsilon_2\,8\,\epsilon_3\,3\,9\,\epsilon_4\,6\,2\,5\,,
\]
which is the \(\epsilon\)-word of \(e^{-1}(\iota(\pi))=[7\,1\,6\,2\,5]\,[4]\,[8]\,[3\,9]\in\A(9,4)\).
\end{example}

\begin{prop}\label{pro:FE=zigzag2}
Let \(R\in\RAT^+(n+1,r+1)\) and let \(\pi=\ferat^{-1}(R)\in\A(n+1,r+1)\). Then
\[
\Phi_Z(R)=\phi(\iota(\pi)).
\]
In particular, if \(\theta\in S_{n+1}\) is the permutation obtained by removing the \(\epsilon_i\)'s from \(\phi(\iota(\pi))\), then \(f(\pi)=\theta^{-1}\).
\end{prop}

\begin{proof}
  First consider the zigzag map \(\Phi_Z\) on \(R\) with the standard
  labeling of the southeast border by \(1,2,\dots,n+1\), and suppose that
   \(\Phi_Z(R) \in S_{\{1,\ldots,n+1,\epsilon_1,\ldots,\epsilon_{r+1}\}} \) is given by
\begin{multline}\label{eq:14}
\Phi_Z(R) = (u_1, \dots, u_{b_1},\epsilon_1)(u_{b_1+1},\dots, u_{b_2},\epsilon_2)
\cdots(u_{b_r+1},\dots, u_{b_{r+1}},\epsilon_{r+1})\\
(u_{b_{r+1}+1},\dots, u_{b_{r+2}})\cdots (u_{b_{r+k}+1},\dots, u_{b_{r+k+1}}),
\end{multline}
where \( b_{r+k+1}=n+1 \) and
\(\epsilon_1<\dots<\epsilon_{r+1}<u_{b_{r+2}}<\cdots<u_{b_{r+k+1}} \)
are the RL-minima of the word
\[
 v := u_1\cdots u_{b_1}\epsilon_1 u_{b_1+1}\cdots u_{b_2}\epsilon_2 \cdots
  u_{b_r+1}\cdots u_{b_{r+1}}\epsilon_{r+1} u_{b_{r+1}+1}\cdots
  u_{b_{r+2}} \cdots u_{b_{r+k}+1}\cdots u_{b_{r+k+1}}.
\]
In particular, \(\Phi_Z(R)\) determines where the zigzag path starting
at each label \(u_i\) terminates: if
\(i\not\in \{b_1,\ldots,b_{r+k+1}\}\), then the zigzag path from
\(u_i\) terminates at \(u_{i+1}\); whereas if \(i=b_j\) for some
\(j\in[r+k+1]\), then the path terminates at \(u_{b_{j-1}+1}\), where
we set \( b_0=0 \).

Since \( \Phi_Z(R) = \phi(v) \), by \Cref{prop:I is Z},
we have \( v = e(\Phi_I(R)) \).
Then, by \Cref{lem:shape}, the labels
of the free rows of \( R \) are \( b_{r+2}<\dots<b_{r+k+1} \).

Let \(\sigma=f(\pi)\in S_{n+1}\) be the canonical permutation of
\(\pi\), and let \(\Phi_Z(R;\sigma)\) denote the zigzag map obtained
by labeling the southeast border of \(R\) by
\(\sigma_1,\dots,\sigma_{n+1}\). In other words, \(\Phi_Z(R;\sigma)\)
is obtained from \(\Phi_Z(R)\) by replacing each entry \(i\) with
\(\sigma(i)\).

We will consider two labelings of the same underlying rhombic
alternative tableau. Let \(R\) denote the tableau whose southeast
border is labeled by \(\sigma_1,\dots,\sigma_{n+1}\). Let \(R'\)
denote the same tableau, but decorated with the labeling produced by
the FE algorithm, in which every edge (including internal edges)
carries a label set. By our definition of
\(\sigma=f(\Phi^{-1}_{FE}(R))\), the southeast border edges of \(R'\)
are labeled by singleton sets
\(\{\sigma_1\},\ldots,\{\sigma_{n+1}\}\), and these labels agree with
those of \(R\) after identifying \(\{\ell\}\) with \(\ell\) for
$\ell\in[n+1]$.

For \(\ell\in[n+1]\), we refer to the \emph{northwest edge in the strip labeled \(\ell\) in \(R'\)} as the northwest border edge lying in the same strip as the southeast border edge of \(R'\) labeled \(\{\ell\}\).
This lets us identify zigzag paths in \(R\) that determine $\Phi_Z(R;\sigma)$ with paths traced through \(R'\). Specifically, a zigzag path in \(\Phi_Z(R;\sigma)\) originating at \(\ell\) and terminating at \(j\) corresponds exactly to a path in \(R'\)
that starts at the northwest edge in the strip labeled \(\ell\) and terminates at the southeast border edge labeled \(\{j\}\).

The northwest border edges of \(R'\) carrying nonempty label sets correspond precisely to diagonal strips and free rows. These label sets are intervals, and by 
\cref{lem:labels} they are totally ordered from north to south by the relation \(\prec\).

We now analyze the zigzag paths on $R'$. See \cref{exa:4} for explicit examples of the following observations. 

The turning cells of \(R'\) are the cells containing arrows, and those are precisely the cells where fusion occurs. Thus we can trace the trajectories of the zigzag paths through the movement of the labels through these cells. First, observe that fusion-exchange interactions preserve the maximal element of a label set along its trajectory. Thus, whenever some edge in \(R'\) carries the label set \(I=[i,j]\), the corresponding southeast border edge carries the label set \(\{j\}\).

On the other hand, at each turning cell \(\tau\) of \(R'\), the minimal element carried by the union of the label sets on the edges \(N(\tau)\) and \(W(\tau)\) necessarily turns. Consequently, a zigzag path passing through an edge entering a turning cell in \(R'\) that carries the label set \(I=[i,j]\), terminates at the southeast edge labeled \(\{i\}\). 

We now determine the endpoints of a zigzag path originating at a northwest edge in the strip labeled $\ell$. If a southeast edge labeled \(\ell\) is not diagonal and not a free row, then it lies in either a column containing an up-arrow or a row containing a left-arrow. Let \(\tau\) be the cell containing this arrow. Then we have that the FE label sets on \((W(\tau),N(\tau),E(\tau),S(\tau))\) in \(R'\) are \((\emptyset,[i,j],[i,\ell],[\ell+1,j])\) if \(\tau\) has a left-arrow and \(([i,j],\emptyset,[\ell+1,j],[i,\ell])\) if \(\tau\) has an up-arrow, for some \(i\leq \ell<j\).  In this case, the label set on the turning edge of that cell is \([\ell+1,j]\), and thus the zigzag path originating at the northwest edge in strip \(\ell\) terminates at the southeast edge with the reduced label set
\[
I'=\{k\in I:k>\ell\}=[\ell+1,j],
\]
so the path terminates at \(\min I'=\ell+1\). Thus in $\Phi_Z(R;\sigma)$, the zigzag path originating at \(\ell\) ends at \(\ell+1\). Similarly, the zigzag path originating at a diagonal or free row edge \(j\) ends at \(\min I=i\), where \(I=[i,j]\) is the interval labeling that edge in \(R'\).

Thus each diagonal strip or each free row of \(R'\) labeled by an
interval \(I=[i,j]\) contributes exactly one cycle
\((i,i+1,\ldots,j,\epsilon_k)\) or \((i,i+1,\ldots,j)\), respectively,
to \(\Phi_Z(R;\sigma)\in\A^\epsilon(n+1,r+1)\). Thus, by
\Cref{lem:labels}, we have
\begin{multline}\label{eq:15}
\Phi_Z(R;\sigma) = (1,\dots,c_1,\epsilon_1)(c_1+1,\dots,c_2,\epsilon_2)
\cdots(c_r+1,\dots,c_{r+1}\epsilon_{r+1})\\
(c_{r+1}+1,\dots, c_{r+2})\cdots (c_{r+k}+1,\dots, c_{r+k+1}),
\end{multline}
and the \( i \)th free row has a label set with largest element
\( c_{r+1+i} \). Since \( \Phi_Z(R;\sigma) \) is obtained from
\( \Phi_Z(R) \) by replacing every integer \( i \) with
\( \sigma(i) \), comparing \eqref{eq:14} and \eqref{eq:15} gives
\( \sigma(u_i)= i \) for all \(i\in[n+1]\), or equivalently,
\(\sigma^{-1}=u_1\cdots u_{n+1}\). Moreover, the entries
\(\epsilon_j\) in the word \(\phi^{-1}(\Phi_Z(R))\) follow the entries
\(\sigma^{-1}(h_j)\), where \(h_j\) is the \(j\)th head in \(\pi\), so
indeed we have \(\iota(\pi)=\phi^{-1}(\Phi_Z(R))\). This completes the proof.
\end{proof}

Finally, we obtain the following from \cref{prop:I is Z}.

\begin{cor}\label{cor:FE=I}
For \(\pi \in \A(n+1,r+1)\), we have
\[
\Phi_Z^{-1}(\pi)=\ferat(\iota^{-1}(e(\pi))).
\]
\end{cor}

By identifying $S_{n+1}$ with $\A(n+1,1)$, we obtain an analogous statement for the corresponding maps on alternative tableaux.
\begin{cor}
Define $\Phi^{AT}_{FE}:S_{n+1}\to\AT(n)$ by setting $\Phi_{FE}(\sigma)$ to be the alternative tableau obtained by removing the diagonal strip from $\Phi_{FE}(\pi)$, where $\pi\in\A(n+1,1)$ is the assembl\'ee consisting of the single block $\sigma$. Then
\[
\ferat^{AT}(\sigma)=(\Phi_I^{\AT})^{-1}(\sigma^{-1})=(\Phi_Z^{AT})^{-1}(\phi(\sigma^{-1})).
\]
\end{cor}
Note that due to the proof of \cref{lem:labels}, our definition of
$\Phi_{FE}^{AT}(\sigma)$ for \( \sigma=\sigma_1 \cdots \sigma_n \) is equivalent to the transpose of the
classical fusion-exchange map in \cite{Viennota} applied to the reversed word \( \sigma_n \cdots \sigma_1 \).

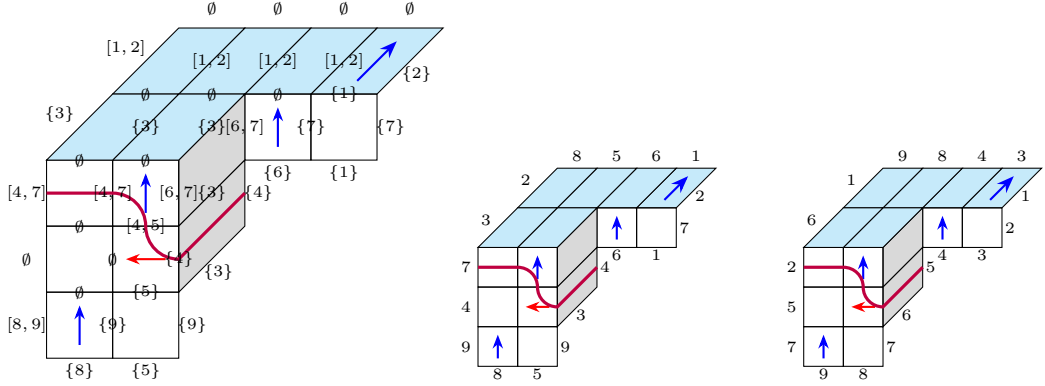
\begin{figure}
    \centering
\begin{tikzpicture}[scale=2.5]
  \short{4}{4} \short{3}{4} \short{2}{4} \short{1}{4} \cell{4}{3} \cell{3}{3} \tall{2}{2} \cell{1}{2} \cell{0}{2} \tall{2}{1} \cell{1}{1} \cell{0}{1} \short{1}{3} \short{0}{3} \cell{1}{0} \cell{0}{0}
\node at (0.500000000000000, -0.200000000000000) {\tiny \( \{8\} \)};
\node at (1.50000000000000, -0.200000000000000) {\tiny \( \{5\} \)};
\node at (2.20000000000000, 0.500000000000000) {\tiny \( \{9\} \)};
\node at (2.60000000000000, 1.30000000000000) {\tiny \( \{3\} \)};
\node at (3.20000000000000, 2.50000000000000) {\tiny \( \{4\} \)};
\node at (3.50000000000000, 2.80000000000000) {\tiny \( \{6\} \)};
\node at (4.50000000000000, 2.80000000000000) {\tiny \( \{1\} \)};
\node at (5.20000000000000, 3.50000000000000) {\tiny \( \{7\} \)};
\node at (5.60000000000000, 4.30000000000000) {\tiny \( \{2\} \)};
 \upArr{0}{0} \upArr{1}{2} \upArr{3}{3} \slantUpArr{4}{4} \leftArr{1}{1}
  \draw[very thick,purple] (0,2.5) to (1,2.5) to[out=0, in=90] (1.5,2) to[out=270, in=180] (2,1.5) to (3,2.5);
   \node at (-.3,.5) {\tiny \( [8,9] \)};
  \node at (-.3,1.5) {\tiny \( \emptyset \)};
  \node at (-.3,2.5) {\tiny \( [4,7] \)};
  \node at (.2,3.7) {\tiny \( \{3\} \)};
  \node at (1.2,4.7) {\tiny \( [1,2] \)};
  \node at (2.5,5.3) {\tiny \( \emptyset \)};
  \node at (3.5,5.3) {\tiny \( \emptyset \)};
  \node at (4.5,5.3) {\tiny \( \emptyset \)};
  \node at (5.5,5.3) {\tiny \( \emptyset \)};
  \node at (1.5,4) {\tiny \( \emptyset \)};
  \node at (.5,3) {\tiny \( \emptyset \)};
  \node at (.5,2) {\tiny \( \emptyset \)};
  \node at (.5,1) {\tiny \( \emptyset \)};
  \node at (1.5,1) {\tiny \( \{5\} \)};
  \node at (1,.5) {\tiny \( \{9\} \)};
  \node at (1,1.5) {\tiny \( \emptyset \)};
  \node at (1,2.5) {\tiny \( [4,7] \)};
  \node at (1.5,3) {\tiny \( \emptyset \)};
  \node at (2.5,4) {\tiny \( \emptyset \)};
  \node at (1.5,2) {\tiny \( [4,5] \)};
  \node at (2,1.5) {\tiny \( \{4\} \)};
  \node at (2.5,2.5) {\tiny \( \{3\} \)};
  \node at (2.5,3.5) {\tiny \( \{3\} \)};
  \node at (1.5,3.5) {\tiny \( \{3\} \)};
  \node at (2,2.5) {\tiny \( [6,7] \)};
  \node at (3,3.5) {\tiny \( [6,7] \)};
  \node at (4,3.5) {\tiny \( \{7\} \)};
  \node at (4.5,4) {\tiny \( \{1\} \)};
  \node at (3.5,4) {\tiny \( \emptyset \)};
  \node at (4.5,4.5) {\tiny \( [1,2] \)};
  \node at (3.5,4.5) {\tiny \( [1,2] \)};
  \node at (2.5,4.5) {\tiny \( [1,2] \)};
\end{tikzpicture}
\begin{tikzpicture}[scale=1.5]
  \short{4}{4} \short{3}{4} \short{2}{4} \short{1}{4} \cell{4}{3} \cell{3}{3} \tall{2}{2} \cell{1}{2} \cell{0}{2} \tall{2}{1} \cell{1}{1} \cell{0}{1} \short{1}{3} \short{0}{3} \cell{1}{0} \cell{0}{0}
\node at (0.500000000000000, -0.200000000000000) {\tiny 8};
\node at (1.50000000000000, -0.200000000000000) {\tiny 5};
\node at (2.20000000000000, 0.500000000000000) {\tiny 9};
\node at (2.60000000000000, 1.30000000000000) {\tiny 3};
\node at (3.20000000000000, 2.50000000000000) {\tiny 4};
\node at (3.50000000000000, 2.80000000000000) {\tiny 6};
\node at (4.50000000000000, 2.80000000000000) {\tiny 1};
\node at (5.20000000000000, 3.50000000000000) {\tiny 7};
\node at (5.60000000000000, 4.30000000000000) {\tiny 2};
 \upArr{0}{0} \upArr{1}{2} \upArr{3}{3} \slantUpArr{4}{4} \leftArr{1}{1}
  \draw[very thick,purple] (0,2.5) to (1,2.5) to[out=0, in=90] (1.5,2) to[out=270, in=180] (2,1.5) to (3,2.5);
   \node at (-.3,.5) {\tiny 9};
  \node at (-.3,1.5) {\tiny 4};
  \node at (-.3,2.5) {\tiny 7};
  \node at (.2,3.7) {\tiny 3};
  \node at (1.2,4.7) {\tiny 2};
  \node at (2.5,5.3) {\tiny 8};
  \node at (3.5,5.3) {\tiny 5};
  \node at (4.5,5.3) {\tiny 6};
  \node at (5.5,5.3) {\tiny 1};
\end{tikzpicture}
\qquad
\begin{tikzpicture}[scale=1.5]
  \short{4}{4} \short{3}{4} \short{2}{4} \short{1}{4} \cell{4}{3} \cell{3}{3} \tall{2}{2} \cell{1}{2} \cell{0}{2} \tall{2}{1} \cell{1}{1} \cell{0}{1} \short{1}{3} \short{0}{3} \cell{1}{0} \cell{0}{0}
\node at (0.500000000000000, -0.200000000000000) {\tiny 9};
\node at (1.50000000000000, -0.200000000000000) {\tiny 8};
\node at (2.20000000000000, 0.500000000000000) {\tiny 7};
\node at (2.60000000000000, 1.30000000000000) {\tiny 6};
\node at (3.20000000000000, 2.50000000000000) {\tiny 5};
\node at (3.50000000000000, 2.80000000000000) {\tiny 4};
\node at (4.50000000000000, 2.80000000000000) {\tiny 3};
\node at (5.20000000000000, 3.50000000000000) {\tiny 2};
\node at (5.60000000000000, 4.30000000000000) {\tiny 1};
 \upArr{0}{0} \upArr{1}{2} \upArr{3}{3} \slantUpArr{4}{4} \leftArr{1}{1}
  \draw[very thick,purple] (0,2.5) to (1,2.5) to[out=0, in=90] (1.5,2) to[out=270, in=180] (2,1.5) to (3,2.5);
   \node at (-.3,.5) {\tiny 7};
  \node at (-.3,1.5) {\tiny 5};
  \node at (-.3,2.5) {\tiny 2};
  \node at (.2,3.7) {\tiny 6};
  \node at (1.2,4.7) {\tiny 1};
  \node at (2.5,5.3) {\tiny 9};
  \node at (3.5,5.3) {\tiny 8};
  \node at (4.5,5.3) {\tiny 4};
  \node at (5.5,5.3) {\tiny 3};
\end{tikzpicture}
    \caption{The same arrow-zigzag path for a RAT $R\in\RAT^+(9,2)$ with edges labeled by the FE procedure, labeled by $\{1,\ldots,9\}$, and labeled by $\sigma$.}
    \label{fig:1}
  \end{figure}

\begin{example}\label{exa:4}
For the RAT \(R\) this example, we compare the outputs of all three
maps \(\ferat(R)\), \(\Phi_Z(R)\), and \(\Phi_I(R)\); see \Cref{fig:1}. We have
\begin{align*}
\ferat(R)&=[2\,7\,1\,6\,4][3\,9\,8],\\
\phi(e(\Phi_Z(R)))&=(3,1,\epsilon_1)(6,\epsilon_2)(5,8,4,2)(9,7),\\
\Phi_I(R)=\Phi_Z(R)&=[3\,1\,5\,8\,4\,2\,9\,7][6].
\end{align*}
Then for \(\pi=\ferat(R)=[2\,7\,1\,6\,4][3\,9\,8]\), we have
\[
  f(\pi)^{-1}=3\,1\,6\,5\,8\,4\,2\,9\,7, \qquad
  \iota(\pi)=3\,1\,\epsilon_1\,6\,\epsilon_2\,5\,8\,4\,2\,9\,7=e(\Phi_Z(R)),
\]
and
\[
  e^{-1}(\iota(\pi))=[3\,1\,5\,8\,4\,2\,9\,7][6]=\Phi_I(R),
\]
as claimed in \cref{pro:FE=zigzag2} and \cref{cor:FE=I}.
\end{example}

\section{Another zigzag map from RAT to assembl\'ees}
\label{sec:zigzag-RPT}

In this section, we give a different zigzag map from RAT to assembl\'ees,
which is a generalization of the zigzag map due to Steingr\'imsson and
Williams \cite{Steingrimsson2007} on permutation tableaux. As a
consequence, we obtain an interpretation of the parameter \( q \) of
the partition function \( \Y_{n,r}(\alpha,\beta,q) \) in terms of
assembl\'{e}es, which was proposed as an open problem in
\cite[Section~5]{Mandelshtam2018}.
More generally, we will consider
 \begin{equation}\label{eq:11}
   \Y_{n,r}(\alpha,\beta,q,y) :=\sum_{T\in \RAT(n,r)}
   \alpha^{\frow(T)} \beta^{\fcol(T)} q^{\fcell(T)}y^{\row(T)},
\end{equation}
which is a generalization of \( \Y_{n,r}(\alpha,\beta,q) \) in
\eqref{eq:Y} with an extra parameter \( y \).

Before stating the results, we introduce some definitions. Throughout
this section, \( X = \{a_1<\dots<a_n\} \) denotes a nonempty set of positive
integers and \( \overline{X} = \{-a_1,\dots,-a_n\} \). We
write \( \overline{i} = -i \).

\subsection{Signed permutations and their arc diagrams}

\newcommand{\Neg}{\text{neg}}

We first introduce signed permutations, which will be used to
interpret assembl\'{e}es.

\begin{defn}
  A \emph{signed permutation} on \( X \) is a bijection
  \( \tau:X\cup \overline{X} \to X\cup \overline{X} \) such that
  \( \tau(-i) = -\tau(i) \) for all \( i\in X \). Denote the set of signed permutations by $S^\pm_X$, and when $X=[n]$ we simply write $S^\pm_n$.
\end{defn}

For \(X=\{a_1<\cdots<a_n\}\) and \(\tau\in S_X^\pm\), define $\neg(\tau)=\{i\in X:\tau(i)\in\overline{X}\}$. The word \(|\tau(a_1)|\cdots|\tau(a_n)|\) is a permutation \(\sigma\in S_X\), and \(\tau\) is uniquely determined by \(\sigma\) together with \(\neg(\tau)\).
Accordingly, we write \(\tau\) in one-line notation as \(\tau_1\cdots\tau_n\), where \(\tau_i=\tau(a_i)\). 

\begin{defn}\label{def:13}
  Let \( \tau\in S^\pm_X \). For distinct integers
  \( i,j\in\neg(\tau) \), define \( i<_\tau j \) if
\[
  \min(i,|\tau(i)|) < \min(j,|\tau(j)|)
\qquad \mbox{or} \qquad
  \min(i,|\tau(i)|) = i = |\tau(j)|= \min(j,|\tau(j)|).
\]
\end{defn}

For any two distinct integers \( i,j\in\neg(\tau) \), we have either
\( i<_\tau j \) or \( j<_\tau i \). Thus $<_\tau$ defines a total
order on $\neg(\tau)$. For example, if
\( \tau= \overline{4}\, 5\, 3\,\overline{6}\, \overline{2}\, 1 \),
then \( 1<_\tau 5 <_\tau 4 \).

For integers \( a\le b \) (resp.~\( a<b \)), an \emph{upper
  arc} (resp.~a \emph{lower arc}) \( (a,b) \) is a curve
connecting the points labeled \( a \) and \( b \) on the \( x \)-axis
that lies above (resp.~below) the \( x \)-axis. The upper arc \((a,a)\) is also called a \emph{loop} at \( a \).

\begin{defn}\label{def:14}
  Let \( X = \{a_1<\dots<a_n\} \) be a set of positive integers. The
  \emph{arc diagram} of a signed permutation
  \( \tau \in S^\pm_X \) is defined as follows. Let
  \( \neg(\tau)=\{d_1 <_\tau \cdots <_\tau d_{r}\} \). Arrange \( n+r \) vertices labeled
  \( -r, -r+1,\dots,-1 , a_1,a_2,\dots,a_n \) on the \( x \)-axis in
  this order from left to right. For each \( i\in X \), do the following.
\begin{itemize}
\item If \( \tau(i)\ge i \), draw an upper arc
  \( (i,\tau(i)) \).
\item If \( 0<\tau(i)< i \), draw a lower arc
  \( (\tau(i),i) \).
\item If \( i=d_j\in\neg(\tau)\), draw a lower arc \( (-j,i) \) and an
  upper arc \( (-j,|\tau(i)|) \).
\end{itemize}
\end{defn}

We draw the arc diagram so that every pair of arcs crosses at most
once. See \Cref{fig:arc_diagram} for an example.

\begin{figure}
  \centering
\begin{tikzpicture}[scale=1.5]
\vertex{1}{1} \vertex{2}{2} \vertex{3}{3} \vertex{4}{4} \vertex{5}{5} \vertex{6}{6} \vertex{0}{-1} \vertex{-1}{-2} \vertex{-2}{-3} \upperEdge{2}{5} \upperLoop{3}{3} \lowerEdge{6}{1} \lowerEdge{1}{0} \upperEdge{0}{4} \lowerEdge{4}{-2} \upperEdge{-2}{6} \lowerEdge{5}{-1} \upperEdge{-1}{2} 
\node at (2,0) {$\square$};
\node at (.8,.7) {$\square$};
\node at (3.2,.7) {$\square$};
\node at (2.65,-1.15) {$\square$};
\node at (3.45,-1.3) {$\square$};
\node at (1.5,-1.5) {$\square$};
\end{tikzpicture}
\caption{The arc diagram of \( \tau= \bar4\, 5\, 3\,\bar6\, \bar2\, 1\in S^\pm_6 \). The six crossings are indicated with $\square$s.}
  \label{fig:arc_diagram}
\end{figure}
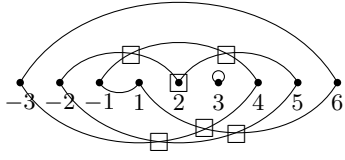

\begin{defn}
    Let \( \tau\in S^\pm_X \) be a signed permutation. An \emph{upper
    crossing} of \( \tau \) is defined as a pair \( ((a,b),(c,d))\) of
  upper arcs in the arc diagram of \( \tau \) such that
  \( a<c\le b<d \). A \emph{lower crossing} of \( \tau \) is defined
  as a pair \( ((a,b),(c,d))\) of lower arcs in the arc diagram of
  \( \tau \) such that \( a<c< b<d \). A \emph{crossing} of \( \tau \)
  is either an upper or a lower crossing of \( \tau \). We denote by
  \( \cro(\tau) \) the number of crossings of \( \tau \).
\end{defn}

For example, the signed permutation \(\tau\) shown in
\cref{fig:arc_diagram} has three upper crossings
\( ((-2,2),(-1,4)) \), \( ((-2,2),(2,5)) \), and \( ((-1,4),(2,5)) \),
three lower crossings \( ((-3,4),(-2,5)) \), \( ((-3,4),(1,6)) \), and
\( ((-2,5),(1,6)) \), hence \(\cro(\tau)=6\).

If a signed permutation is a permutation, i.e., if it has no negative
entries, then the definition of a crossing coincides with the classical one for
permutations in~\cite{Corteel2007a}.

\subsection{Assembl\'{e}es as signed permutations}

An assembl\'{e}e \( \pi\in \A(n,r) \) can be viewed as the signed
permutation \( \tau\in S^\pm_n\) obtained from the canonical
permutation \( f(\pi) \) by replacing each block head
\( a\in\heads(\pi) \) with \( -a \). Such a signed permutation can be
characterized as follows.

\begin{defn}\label{def:5}
  Define $\AS(n,r)$ to be the set of 
  signed permutations \( \tau\in S^\pm_n \) satisfying the following properties: 
\begin{itemize}
\item[(1)] \(|\neg(\tau)|=r \), 
\item[(2)] \( 1\in\neg(\tau) \), and 
\item[(3)] if \(\neg(\tau)=\{d_1 < \cdots <d_r\}\), then
\( |\tau(d_1)| < \cdots < |\tau(d_r)| \). 
\end{itemize}
\end{defn}

Any signed permutation with the properties above corresponds to an
assembl\'{e}e in \( \A(n,r) \). For example, the signed permutation
\(\tau= \bar5\,4\,9\,\bar6\,1\,\bar7\,3\,8\,2\in \AS(9,3) \)
corresponds to the assembl\'ee
\( \pi = [5\,\,4\,\,9]\,\, [6\,\,1]\,\, [7\,\,3\,\,8\,\,2] \in \A(9,3)
\). It is useful to extend this identification between \( \A(n,r) \)
and \( \AS(n,r) \) to assembl\'{e}es on an arbitrary finite set of
positive integers.

\begin{defn}\label{def:7}
  Let \( X = \{a_1<\dots<a_n\} \) be a set of positive integers. An
  \emph{assembl\'{e}e} on \( X \) is a collection
  \( \{\sigma^{(1)},\dots,\sigma^{(r)}\} \) of \( r \) permutations
  \( \sigma^{(i)}\in S_{B_i} \) such that \( \{B_1,\dots,B_r\} \) is a
  set partition of \( X \). Let \( \A(X) \) denote the set of
  assembl\'{e}es on \( X \). Define \(\AS(X)\) to be the set of signed
  permutations \( \tau\in S^\pm_X \) such that
  \( |\tau(d_1)| < \cdots < |\tau(d_r)| \) and \(d_1=a_1\), where
  \( \neg(\tau)=\{d_1 < \cdots <d_r\} \).
\end{defn}

Note that the above bijection between \( \A(n,r) \) and \( \AS(n,r) \)
naturally extends to a bijection between \(\A(X)\) and \(\AS(X)\).

Observe that for a signed permutation $\tau\in\AS(X)$ with $\neg(\tau)=\{d_1<\cdots<d_r\}$, since \( |\tau(d_i)|<|\tau(d_j)| \) for all \( 1\le i<j\le r \), we have
\( d_1 <_\tau \cdots <_\tau d_r \). Consequently, for each \( i\in [r] \), there
are a lower arc \( (-i,d_i) \) and an upper arc \( (-i,|\tau(d_i)|) \), and no crossings occur among these lower arcs or among these upper arcs.
We may therefore represent the pair of arcs \( (-i,d_i) \) and \( (-i,|\tau(d_i)|) \) by a single \emph{spiral arc} from \( d_i \) to \( |\tau(d_i)| \), omitting the intermediate vertex \( -i \), such that the resulting spiral arcs do not cross, as shown in \Cref{fig:signed_perm}.

\begin{defn}
  A \emph{spiral arc diagram} of a signed permutation
  \(\tau\in\AS(X)\) is an arc diagram consisting of upper arcs
  \(\{(i,\tau(i)):i\not\in\neg(\tau), i\le \tau(i)\}\), lower arcs
  \(\{(\tau(i),i):i\not\in\neg(\tau), \tau(i)<i\}\), and spiral arcs
  \(\{(i,|\tau(i)|):i\in\neg(\tau)\}\) for \(i\in X\), such that the
  spiral arcs are pairwise noncrossing.
\end{defn}

Note that for a given signed permutation \(\tau\in\AS(X)\), the spiral arc diagram constructed above is uniquely determined: the spiral arcs \((d_i,|\tau(d_i)|)\) are nested in the order $d_1<\cdots< d_r$, where \(\neg(\tau)=\{d_1,\ldots, d_r\}\). Equivalently, a signed permutation $\tau\in S^\pm_X$ is in the set $\AS(X)$ if and only if it admits a spiral arc diagram, and has a spiral arc incident to the leftmost vertex.

\begin{figure}
  \centering
  \begin{tikzpicture}[scale=1.5]
  \def\n{-2};
  \def\k{-1};
    \vertex{0}{-1}\vertex{\k}{-2}\vertex{\n}{-3}\vertex{1}{1} \vertex{1.8}{2} \vertex{3}{3} \vertex{4.2}{4} \vertex{5}{5} \vertex{6}{6} \vertex{7}{7} \vertex{8}{8} \vertex{9}{9} \upperEdge{0}{5}\lowerEdge{1}{0} \upperEdge{1.8}{4.2} \upperEdge{3}{9} \lowerEdge{4.2}{\k} \upperEdge{\k}{6} \lowerEdge{5}{1} \lowerEdge{6}{\n} \upperEdge{\n}{7} \lowerEdge{7}{3} \upperLoop{8}{8} \lowerEdge{9}{1.8} 
  \end{tikzpicture}\hspace{0.5in}
  \begin{tikzpicture}[scale=1.5]
    \vertex{1}{1} \vertex{2}{2} \vertex{3}{3} \vertex{4}{4} \vertex{5}{5} \vertex{6}{6} \vertex{7}{7} \vertex{8}{8} \vertex{9}{9} \oneSpiralEdge{1}{0.500000000000000}{5} \upperEdge{2}{4} \upperEdge{3}{9} \spiralEdge{4}{0.000000000000000}{6} \lowerEdge{5}{1} \spiralEdge{6}{-0.500000000000000}{7} \lowerEdge{7}{3} \upperLoop{8}{8} \lowerEdge{9}{2} 
  \end{tikzpicture}
  \caption{We show the standard arc diagram (left) and the spiral arc diagram (right) of the signed permutation
    \( \bar5\,4\,9\,\bar6\,1\,\bar7\,3\,8\,2\in\AS(9,3) \). The number of upper and lower arc crossings is 4 and 7, respectively, for a total of 11 crossings.}
  \label{fig:signed_perm}
\end{figure}
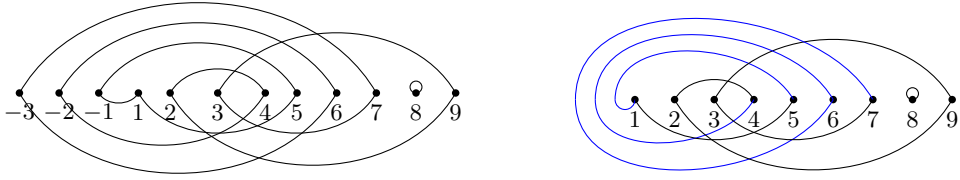

\begin{defn}
  Let \( X = \{a_1<\dots<a_n\} \) be a set of positive integers and
  let $\tau\in \AS(X)$. The \emph{shape} \( \sh(\tau) \) of \( \tau \)
  is defined to be the word \( w_1\cdots w_n \), where 
  \[
    w_i =
    \begin{cases}
     2 & \mbox{if \( \tau(a_i)\ge a_i \)},\\
     1 & \mbox{if \( \tau(a_i)<0 \)},\\
     0 & \mbox{if \( 0<\tau(a_i)<a_i \).}
    \end{cases}
  \]
  A \emph{shape-inversion} of \( \tau \) is a pair \( (a_i,a_{j}) \)
  such that \( 1\le i<j\le n \) and \( w_i>w_j \). A \emph{weak
    excedance} of \( \tau \) is an integer \( a_i\in X \) such that
  \( \tau(a_i) \ge a_i \). We denote the number of shape-inversions
  and weak excedances in \( \tau \) by \( \sinv(\tau) \) and
  \( \wex(\tau) \), respectively.
\end{defn}

\begin{defn}\label{def:1}
  Let \( X = \{a_1<\dots<a_n\} \) be a set of positive integers and
  let $\tau\in \AS(X)$ with \( \neg(\tau)=\{d_1<\cdots<d_r\} \). For \( a_i \in X \), we say
  that
  \begin{itemize}
  \item \( \tau(a_i) \) is a \emph{special LR-maximum}
  if \( \tau(a_i) > |\tau(k)| \) for all \( k\in \{a_1,\dots,a_{i-1},d_1,\dots,d_r\} \), 
  \item \( \tau(a_i) \) is a \emph{special RL-minimum}
  if \( 0<\tau(a_i) < |\tau(k)| \) for all \( k\in \{a_{i+1},\dots,a_{n},d_1,\dots,d_r\} \).
  \end{itemize}
  We denote the number of special LR-maxima and special RL-minima of \( \tau \) by \( \LRmax^*(\tau) \) and  \( \RLmin^*(\tau) \).
\end{defn}

In particular, if $a_i\in\neg(\tau)$, then $\tau(a_i)$ cannot be a special LR-maximum or a special RL-minimum. As we identify an assembl\'ee \(\pi\in\A(n,r)\) with its corresponding signed permutation \(\tau\in \AS(n,r)\), we may equivalently write $\stat(\pi)=\stat(\tau)$ for $\stat\in\{\wex,\sinv,\sh,\LRmax^*,\RLmin^*\}$.

\begin{example}\label{ex:stats}
The assembl\'ee \( \pi = [5\,\,4\,\,9]\,\, [6\,\,1]\,\, [7\,\,3\,\,8\,\,2] \in \A(9,3) \) is represented by the signed permutation \(\tau= \bar5\,4\,9\,\bar6\,1\,\bar7\,3\,8\,2\in \AS(9,3) \). Its arc diagram is shown in \cref{fig:signed_perm}. We have that \( 9 \) is the only  special LR-maximum, and the RL-minima are \( 2 \) and \( 1 \). Thus
 $\neg(\tau)=\{5,6,7\}$, $\sh(\tau)=122101020$, $\sinv(\tau)=19$, $\wex(\tau)=3$, $\RLmin^*(\tau)=2$, and $\RLmax^*(\tau)=1$.
\end{example}

\subsection{A zigzag map \( \zeta \) that preserves the $q$ statistic}\label{sec:zeta}

In this subsection, we define another zigzag map \( \zeta \) on \( \RAT^+ \).
First we introduce some notation.

\begin{defn}
  Let \( X = \{a_1<\dots<a_n\} \) be a set of positive integers. We
  denote by \( \RAT(X) \) and $\RAT^+(X)$ the sets of RAT and extended
  RAT whose southeast border is labeled by \( a_1,\dots,a_n \). Let
  \( \A \) denote the set of all assembl\'{e}es (on any finite set
  \( X \)), and let \( \RAT^+ \) denote the set of all extended RAT
  (on any finite set \( X \)).

  For \( R\in \RAT^+(X) \), let \( \topup(R) \) denote the number of
  up-arrows in diagonal strip \( a_1 \), and let \( \tile(R) \) denote
  the number of tiles in \( R \).
\end{defn}

Note that under the bijection \( R\mapsto R^+ \), we have
\( \fcol(R) = \topup(R^+) \). Hence,
we can rewrite \eqref{eq:11} as
\begin{equation}\label{eq:12}
   \Y_{n,r}(\alpha,\beta,q,y)
  =  \sum_{R\in \RAT^+(n+1,r+1)}  \alpha^{\frow(R)} \beta^{\topup(R)} q^{\fcell(R)}y^{\row(R)}.
\end{equation}

The following theorem is the main result in this section.

\begin{thm}\label{thm:zigzag-RPT}
  Let \( X  \) be a finite set of positive integers.
  There is a bijection \( \zrpt: \RAT^+(X)\to \A(X) \) such that
  if \( \zrpt(R) = \pi \), then
  \begin{align}
    \diag(R) &= \neg(\pi),\label{eq:diag}\\
    \tile(R) &= \sinv(\pi),\label{eq:sinv}\\
    \row(R) &= \wex(\pi),\label{eq:wex}\\
    \topup(R) &= \RLmin^*(\pi),\label{eq:RLmin}\\
    \frow(R) &= \LRmax^*(\pi),\label{eq:LRmax}\\
    \fcell(R) &= \cro(\pi).\label{eq:crossing}
  \end{align}
  In particular, when $X=[n+1]$, we obtain the bijection  \( \zrpt: \RAT^+(n+1,r+1)\to \A(n+1,r+1) \) with the properties above.
\end{thm}

By \Cref{thm:zigzag-RPT} and \eqref{eq:12}, we have another
description of \( \Y_{n,r}(\alpha,\beta,q,y) \):
\begin{equation}\label{eq:13}
   \Y_{n,r}(\alpha,\beta,q,y)
  = \sum_{\pi\in \A(n+1,r+1)} \alpha^{\LRmax^*(\pi)}
  \beta^{\RLmin^*(\pi)} q^{\cro(\pi)}y^{\wex(\pi)}.
\end{equation}
If \( r=1 \) in \eqref{eq:13}, then we obtain the results of
\cite{Steingrimsson2007}; see also \cite[(5)]{Josuat-Verges2011}.

Recall up-free-zigzag paths in \Cref{def:zigzag}. For the rest of this
section, for brevity, up-free-zigzag paths will be called simply
zigzag paths. Accordingly, from now on, a \emph{turning cell} is either a free cell
or a cell containing an up-arrow, and a \emph{straight cell} is a cell
that is not a turning cell. 

\begin{defn}
The \emph{initial substrip} of a zigzag path \( p \) is the maximal set \( \{c_1,\dots,c_m\} \) of straight cells such that \( c_i \) is the \( i \)th cell that \( p \) passes
through for all \( i\in [m] \).
\end{defn}
Observe that the initial substrip of a zigzag path $p$ which starts in strip $k$ is necessarily contained within strip $k$. Moreover, if a cell $c$ contains an arrow pointing in the direction of strip $i$, then the initial substrip of the zigzag path that starts from strip $i$ must contain every cell the arrow in cell $c$ points to, and if the arrow is a left-arrow, the initial substrip also contains $c$ itself.

\begin{example}\label{ex:substrips}
In \cref{fig:zigzag2}, the initial substrip of the zigzag path from strip 2 to strip 4 is $\{(2,9), (2,7), (2,6)\}$, corresponding to the set of cells pointed to by the arrow in the cell $(2,6)$, including $(2,6)$ itself. For another example, the initial substrip of the zigzag path from strip 7 to strip 3 is $\{(1,7)\}$, corresponding to the cell pointed to by the arrow in the cell $(4,7)$. The initial substrips of the zigzag paths from strip 9 to strip 2 and from strip 1 to strip 5 are empty. 
\end{example}

We say that two zigzag paths \emph{intersect} if there exists a straight cell
\( c \) such that both paths pass through it. 
For example, in \Cref{fig:zigzag2}, the path from strip 1 to strip 5 intersects the path from strip 2 to strip 4 at the cell $(2,6)$, but does not intersect the path from strip 9 to strip 2.

\begin{figure}
  \centering
\begin{tikzpicture}[scale=1.5]
  \short{4}{5} \short{3}{5} \short{2}{5} \tall{4}{3} \cell{3}{3} \tall{2}{2} \cell{1}{2} \cell{0}{2} \tall{4}{2} \cell{3}{2} \tall{2}{1} \cell{1}{1} \cell{0}{1} \short{3}{4} \short{2}{4} \short{1}{4} \short{1}{3} \short{0}{3} \cell{0}{0} 
  \node at (0.500000000000000, -0.200000000000000) {\tiny 9};
\node at (1.20000000000000, 0.500000000000000) {\tiny 8};
\node at (1.50000000000000, 0.800000000000000) {\tiny 7};
\node at (2.60000000000000, 1.30000000000000) {\tiny 6};
\node at (3.50000000000000, 1.80000000000000) {\tiny 5};
\node at (4.60000000000000, 2.30000000000000) {\tiny 4};
\node at (5.20000000000000, 3.50000000000000) {\tiny 3};
\node at (5.20000000000000, 4.50000000000000) {\tiny 2};
\node at (5.60000000000000, 5.30000000000000) {\tiny 1};
\node at (3.5, 6.3) {\tiny 9};
\node at (-.3, 0.5) {\tiny 8};
\node at (4.5, 6.3) {\tiny 7};
\node at (.3, 3.7) {\tiny 6};
\node at (5.5, 6.3) {\tiny 5};
\node at (1.3, 4.7) {\tiny 4};
\node at (-.3, 1.5) {\tiny 3};
\node at (-.3, 2.5) {\tiny 2};
\node at (2.3, 5.7) {\tiny 1};
 \slantUpArr{4}{5} \slantUpArr{2}{4} \slantUpArr{2}{5} \slantLeftArr{2}{2} \leftArr{0}{0}
 \cellDot01 \cellDot11 \tallDot21 \cellDot32 \tallDot42
 \cellDot33 \tallDot43 \shortDot34 \shortDot03 \shortDot13 \shortDot14
 \draw[very thick,red] (2.5,5.5) to[out=0, in=45] (2.5,5) to[out=225, in=180] (2.5,4.5) to[out=0, in=45] (2.5,4)
  to[out=225, in=180] (2.5,3.5) -- (2.5,2.5) to[out=270, in=225] (3,2.5) to[out=0,in=90] (3.5,2);
  \draw[very thick,blue,dotted] (0,2.5) to (2,2.5) to (3,3.5) to[out=0,in=90] (3.5,3)
  to[out=270,in=180] (4,2.5) to[out=45,in=90] (4.5,2.5);
  \draw[very thick,purple,dashed] (3.5,6) to[out=270,in=180] (3.5,5.5) to (4.5,5.5)
  to[out=0,in=45] (4.5,5) to[out=225,in=0] (4.5,4.5) to[out=90,in=225] (5,4.5);
\end{tikzpicture}
\caption{An extended RAT whose free cells are marked with dots. The
  red zigzag path enters from strip \( 1 \) and exits from strip
  \( 5 \). The blue (dotted) zigzag path enters from strip \( 2 \) and
  exits from strip \( 4 \). The purple (dashed) zigzag path enters
  from strip \( 9 \) and exits from strip \( 2 \).}
  \label{fig:zigzag2}
\end{figure}
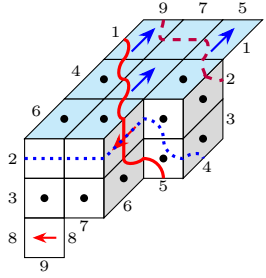

We now define the map $\zeta$ and show that it has the properties in
\Cref{thm:zigzag-RPT}. Although \(\zeta\) is naturally defined as a
map from \(\RAT^+(X) \) to \( \AS(X) \), by the identification between
the signed permutations in \( \AS(X) \) and the assembl\'{e}es in
\( \A(X) \), we also view \(\zeta\) as a map from \(\RAT^+(X)\) to
\( \A(X) \).

\begin{defn}\label{def:2}
  Let \( R\in \RAT^+(X) \). We define \( \zrpt(R) \) to be the signed permutation $\tau\in \AS(X)$ such that for all \( i\in X \),
  we have \( \tau(i) = \epsilon j \), where the zigzag path starting
  from strip \( i \) exits from strip \( j \), and \( \epsilon=-1 \) if
  strip \( i \) is a diagonal strip, and \( \epsilon=1 \) otherwise.
\end{defn}

For example, if \( R \) is the RAT in \Cref{fig:zigzag2}, then
\( \zrpt(R) \) is the signed permutation in \Cref{fig:signed_perm}.
The following lemma shows that \(\zeta\) indeed defines a map from
\(\RAT^+(X)\) to \( \AS(X) \).

\begin{lemma}\label{lem:3}
  For \( R\in \RAT^+(X) \), we have  \( \zrpt(R) \in \AS(X) \).
\end{lemma}

\begin{proof}
  Let \( R\in \RAT^+(X) \) and \( \zrpt(R) = \tau \). Since
  \( R\in \RAT^+(X) \), we have \( \min(X)\in \neg(\tau) \). Thus, it
  suffices to show that if \( d \) and \( d' \) are labels
  of diagonal strips in \( R \) such that \( d<d' \), then
  \( |\tau(d)| < |\tau(d')| \).

  Let \( p \) and \( p' \) be the zigzag paths starting from diagonal
  strips \( d \) and \( d' \), respectively. In particular, the
  initial substrips of $p$ and $p'$ are contained in the diagonal
  strips $d$ and $d'$, respectively. It suffices to show that \( p \)
  and \( p' \) do not intersect, as this would imply that \( p \)
  always stays to the northeast of \( p' \), and hence
  \(|\tau(d)|<|\tau(d')|\) since the labels on \(R\) increase from
  northeast to southwest. To this end, suppose by way of contradiction
  that \( p \) and \( p' \) intersect. Let \(c\) be the first cell
  where they intersect, which is necessarily a short or tall rhombus.
  There are two cases.

  \begin{description}
  \item[Case 1] \( c \) is a tall rhombus, so \(N(c)\) is diagonal and
    \(W(c)\) is vertical. Since \( c \) is a straight cell, either it contains a left-arrow, or there must be a left-arrow pointing to \(c\). Then the cells that are pointed to by this left-arrow are contained in the initial substrip of \(p'\). However, this implies that the initial substrip of \( p' \) is in a horizontal strip, a contradiction.
  \item[Case 2] \( c \) is a short rhombus, so \(N(c)\) is horizontal
    and \(W(c)\) is diagonal. Since \( c \) is a straight cell, there must be an up-arrow pointing to \( c \). Then the cells that are pointed to by this up-arrow are contained in the initial substrip of \(p\). However, this implies that the initial substrip of \(p\) is in a vertical strip, a contradiction.
  \end{description}
  Since both cases imply a contradiction, \( p \) and \( p' \) do not
  intersect. Hence we obtain \(|\tau(d)|<|\tau(d')|\), completing the
  proof.
\end{proof}

In what follows, we prove several properties of the map \( \zrpt \).
The following lemma verifies \eqref{eq:diag} and \eqref{eq:wex}, and
is an immediate consequence of the construction of $\zrpt$.

\begin{lemma}\label{lem:4}
  Let \( R\in \RAT^+(X) \), \( \zrpt(R) = \tau \), and \( i\in X \).
  Then the following hold:
\begin{itemize}
\item \( i \) is the label of a diagonal strip in \( R \) if and only if \( i\in\neg(\tau) \);
\item \( i \) is the label of a row strip in \( R \) if and only if \( i \) is a weak excedance of \( \tau \);
\item \( i \) is the label of a column strip in \( R \) if and only if \( 0<\tau(i)<i \).
\end{itemize}
\end{lemma}

\begin{lemma}[Proof of \eqref{eq:sinv}]\label{lem:5}
  Let \( R\in \RAT^+(X) \), \( \zrpt(R) = \tau \), and \( i,j\in X \),
  where \( i<j \). Then $(i,j)$ is a cell in $R$
  if and only if \( (i,j) \) is a shape-inversion of \( \tau \).
\end{lemma}

\begin{proof}
  There is a cell in strips \( i \) and \( j \) if and only if one of
  the following conditions holds:
  \begin{itemize}
  \item strip \( i \) is a row strip and strip \( j \) is a diagonal strip (in this case, \( (i,j) \) is a tall rhombus);
  \item strip \( i \) is a row strip and strip \( j \) is a column strip (in this case, \( (i,j) \) is a square);
  \item strip \( i \) is a diagonal strip and strip \( j \) is a
    column strip (in this case, \( (i,j) \) is a short rhombus).
  \end{itemize}
  By \Cref{lem:4}, this is equivalent to the condition that
  \( (i,j) \) is a shape-inversion of \( \tau \).
\end{proof}

\begin{lemma}\label{lem:z+intersect}
  Let \( R\in \RAT^+(X) \). Suppose that \( p_1 \) and \( p_2 \) are
  zigzag paths in \( R \) that intersect at a cell \( c \).
  Then, \( c \) is a straight cell contained in at least one
  of the initial substrips of \( p_1 \) and \( p_2 \).
\end{lemma}

\begin{proof}
 As $p_1$ and $p_2$ intersect at the cell $c$, it must be a straight cell. Suppose that \( c=(i,j) \) with $i<j$, and that $p_1$ and $p_2$ passing through $c$ in strip $i$ and strip $j$, respectively. Since \( c \) is not a free cell or contains an up-arrow, either there is an arrow pointing to \( c \) in strip \( i \) or strip \( j \), or \( c \) itself contains a left-arrow. If there is an arrow in strip \( i \) (resp.~\( j \)) pointing to \( c \), then this strip must contain the initial substrip of \( p_1 \) (resp.~\( p_2 \)). If \( c \) contains a left-arrow, then strip $i$ is horizontal, and contains the initial substrip of $p_1$, which in turn contains \( c \).
\end{proof}

\begin{lemma}[Proof of \eqref{eq:RLmin}]\label{lem:RLmin}
  Let \( R\in \RAT^+(X) \), \( \zrpt(R) = \tau \), and let \( i\in X \).
  Then \( i \) is the label of a free column in \( R \) if and only if
  \( \tau(i) \) is a special RL-minimum.
\end{lemma}

\begin{proof}
  For the ``only if'' part of the statement, suppose that \( i \) is
  the label of a free column in \( R \). Then the short tile
  \( (1,i) \) has an up-arrow. Thus, the zigzag path \( p \) starts
  from the column strip \( i \). Now let \( p' \) be a zigzag path
  starting from a strip \( j \) with \( j>i \) that is either a column strip or a diagonal
  strip. To show that \( \tau(i) \) is a special RL-minimum, it
  suffices to show that \( p \) and \( p' \) do not intersect.

  Suppose, for a contradiction, that \( p \) and \( p' \) intersect at
  a cell \( c \). By \Cref{lem:z+intersect}, \( c \) is contained in
  the initial substrip of \( p \) or \( p' \). Since the initial
  substrip of \( p \) is empty, \( c \) is contained in the initial
  substrip of \( p' \). Note that if \( j \) is the label of a column
  strip with \( j>i \), then the whole column strip \( j \) is
  strictly to the left of \( p \). Hence, \( p' \) must start from a
  diagonal strip whose initial substrip contains \( c \). There are
  two cases.

  \begin{description}
  \item[Case 1] \( c \) is a short rhombus. Then \( p \) enters \( c \)
    at the horizontal edge $N(c)$ and \( p' \) enters \( c \) at the diagonal edge $W(c)$.
    Since \( c \) is a straight cell, there must be an up-arrow pointing to
    it. Then, \( c \) must also be contained in the initial substrip
    of \( p \), a contradiction.

  \item[Case 2] \( c \) is a tall rhombus. Then \( p \) enters \( c \)
    at the diagonal edge $N(c)$ and \( p' \) enters \( c \) at the vertical edge $W(c)$.
    However, this is impossible because every initial substrip
    starting from a diagonal strip always enters any tile at a diagonal edge.
  \end{description}

  In either case, we reach a contradiction. Hence \( p \) and \( p' \)
  do not intersect. This shows the ``only if'' part of the statement.

  For the ``if'' part of the statement, suppose that \( i \) is not
  the label of a free column in \( R \). If \( i \) is the label of a
  diagonal strip, then  \(i\in\neg(\tau)\), so \( i \) cannot be a special RL-minimum of \( \tau \). If \( i \) is the label of a row strip, then
  again \( i \) is not a special RL-minimum of \( \tau \). To see this,
  observe that in this case, we have \( i\le \tau(i)<|\tau(j)| \) for
  all \( j\in \{1,i+1,i+2,\dots,n\} \) since \( 1 \) is always the
  label of a diagonal strip. However, this implies
  \( \{|\tau(1)|,|\tau(i+1)|,|\tau(i+2)|,\dots,|\tau(n)|\}\subseteq
  \{i+1,\dots,n\} \), which is impossible. It remains to check the
  case that \( i \) is the label of a column containing an up-arrow. In this case,
  \( (1,i) \) does not contain an up-arrow. Let \( t \) be the largest
  integer with \( t<i \) such that \( (1,t) \) contains an up-arrow.
  Then \( t \) is the label of a column strip, and let \( p' \) be the
  zigzag path starting from strip \( t \). (If there is no such
  \( t \), then let \( p' \) be the diagonal strip containing
  \( (1,i) \).) Then \( p' \) changes direction at \( (1,t) \) and at
  \( (1,i) \). By \Cref{lem:z+intersect}, \( p \) and \( p' \) do not
  intersect, which implies that \( \tau(t)>\tau(i) \). Therefore,
  \( i \) is not a special RL-minimum of \( \tau \). This completes the proof.
\end{proof}

\begin{lemma}[Proof of \eqref{eq:LRmax}]\label{lem:LRmax}
  Let \( R\in \RAT^+(X) \), \( \zrpt(R) = \tau \), and \( i\in X \).
  Then \( i \) is the label of a free row in \( R \) if and only if
  \( \tau(i) \) is a special LR-maximum.
\end{lemma}

\begin{proof}
  For the ``only if'' part of the statement, suppose that \( i \) is
  the label of a free row in \( R \) and let \( p \) be the zigzag
  path starting from row strip \( i \). Let \( p' \) be a zigzag path
  starting from row \( j \) with \( j<i \) or from a diagonal strip.
  To show that \( \tau(i) \) is a special LR-maximum, it suffices to
  show that \( p \) and \( p' \) do not intersect.

  Let \( p_0 \) be the initial substrip of \( p \). Then every cell in
  \( p_0 \) is straight. Since row \( i \) is free, every cell in
  \( p_0 \) has an up-arrow pointing to it. Since row strip \( i \)
  has only squares and tall rhombi and there is no up-arrow pointing
  to a tall rhombus, every cell in \( p_0 \) is a square. Let
  \( (i,c_1),\dots,(i,c_k) \) be the squares in \( p_0 \). Suppose
  that strip \( j \) is a row strip with \( j<i \) or a diagonal
  strip. Let \( p' \) be the zigzag path starting from strip \( j \).
  Then, the cells \( (j,c_1),\dots,(j,c_k) \) are straight cells.
  Thus, \( p' \) cannot visit any cell in \( p_0 \). By
  \Cref{lem:z+intersect}, \( p \) and \( p' \) do not intersect. This
  shows the ``only if'' part of the statement.

  For the ``if'' part of the statement, suppose that \( i \) is not
  the label of a free row in \( R \). If \( i \) is the label of a
  diagonal strip, then \( i\in\neg(\tau) \), so $i$ is not a special LR-maximum of
  \( \tau \). If \( i \) is the label of a column strip, then again \( i \) is not a special
  RL-minimum of \( \tau \). Indeed, if \( i \) is both the label of a
  column strip and a special RL-minimum, we have
  \( i> \tau(i)>|\tau(j)| \) for all \( j\in \{1,\dots,i\} \). However,
  this implies
  \( \{|\tau(1)|,\dots,|\tau(i)|\}\subseteq \{1,\dots,i-1\} \), which is
  impossible. It remains to check the case that \( i \) is the label
  of a (non-free) row.

  Suppose that \( p \) is the zigzag path starting from row strip
  \( i \). Let \( (i,c) \) be the leftmost cell in row \( i \) such
  that either \( c \) is the label of a diagonal strip or \( c \) is
  the label of a column strip that contains an up-arrow above
  \( (i,c) \). Such a cell exists because the cell in row \( i \)
  containing a left-arrow satisfies the condition. Let \( (j,c) \) be
  the last turning cell in strip \( c \) before \( (i,c) \). Then the
  zigzag path \( p' \) starting from strip \( j \) makes the first
  turn at \( (j,c) \) and passes through \( (i,c) \). By
  \Cref{lem:z+intersect}, \( p \) and \( p' \) intersect only once at
  \( (i,c) \), which shows \( |\tau(j)|>\tau(i) \). Since strip \( j \)
  is either a row strip with \( j<i \) or a diagonal strip, this
  means that \( i \) is not a special LR-maximum. This completes the
  proof.
\end{proof}
\newcommand{\tp}{\gamma}

By \Cref{lem:4,lem:5,lem:RLmin,lem:LRmax}, for the proof of
\Cref{thm:zigzag-RPT}, it remains to show that \( \zrpt \) is a
bijection and that \( \fcell(R) = \cro(\zrpt(R)) \). To establish the
bijectivity of \( \zeta \), we construct an explicit inverse map. This inverse is defined recursively using a family of invertible local operators $\{D_{i,j}^{\tp}\}$, which act on $R\in\RAT^+$ by reducing the number of
cells of $R$. A parallel family of operators similarly acts on signed permutations $\tau\in \AS$ by reducing the number of shape inversions, thereby allowing us to prove our results by induction on the number
of cells. 

\begin{defn}\label{def:6}
  Let \( X=\{a_1 < \cdots <a_n\} \) and \( R\in \RAT^+(X) \). A
  \emph{corner} of \( R \) is a cell \( c=(a_k,a_{k+1}) \) for some
  \( 1\le k\le n-1 \). The \emph{type} of the corner \( c \) is
  \[
    \tp =
    \begin{cases}
     L & \mbox{if the cell \( c \) contains a left-arrow},\\
     U & \mbox{if the cell \( c \) contains an up-arrow},\\
     \emptyset{} & \mbox{otherwise.}
    \end{cases}
  \]
\end{defn}

Note that a corner \(c=(a_k,a_{k+1})\) has its edges \(E(c)\) and
\(S(c)\) lying on the southeast border of $R$, corresponding to the
southeast border edges labeled \(a_k\) and \(a_{k+1}\),
respectively. 

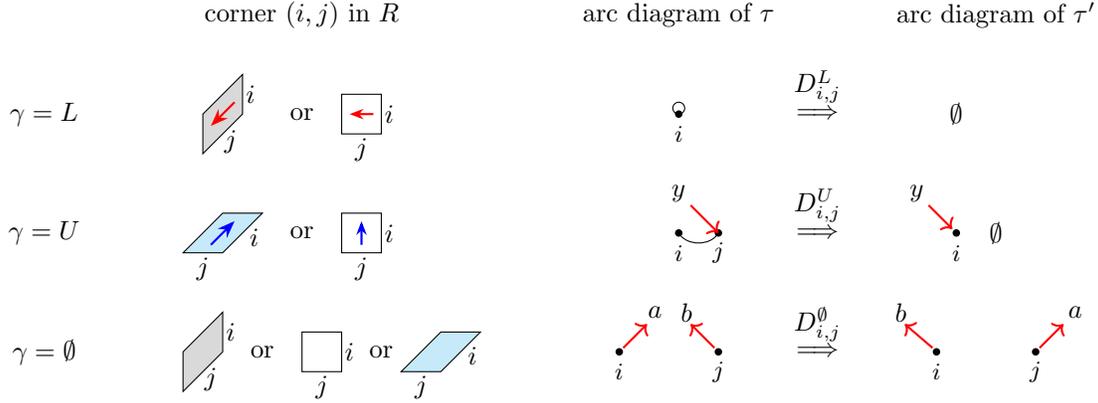
\begin{figure}
  \centering
\begin{tikzpicture}[scale=1.5]
\def\w{11}
\node at (2.5,3) {corner $(i,j)$ in $R$};
\node at (\w+1,3) {arc diagram of $\tau$};
\node at (\w+6+3,3) {arc diagram of $\tau'$};

  \node at (-4,.5) {\( \tp=L \)};
  \tall{0}{-.5}  \cell{3.5}{0}
  \node at (2.5,.5) {or};
  \node at (0.7, -0.2) {\( j \)};
  \node at (1.2, 1) {\( i \)};
    \node at (4, -0.4) {\( j \)};
  \node at (4.7, .5) {\( i \)};
  \slantLeftArr{0}{-.5}
  \leftArr{3.5}{0}
  \begin{scope} [shift={(\w,.5)}]
    \vertex{1}{i} \upperLoop{1} 
    \node at (4.5,0) {\( \Longrightarrow \)};
    \node at (4.5,.7) {\( D^L_{i,j} \)};
  \end{scope}
  \begin{scope} [shift={(\w+6,.5)}]
  \node at (2,0) {$\emptyset$};
  \end{scope}

\begin{scope}[shift={(0,-3)}]
  \node at (-4,.5) {\( \tp=U \)};
 \short{-.5}{0}
  \node at (0, -0.4) {\( j \)};
  \node at (1.3, 0.4) {\( i \)};
  \slantUpArr{-.5}{0}
    \cell{3.5}{0}
  \node at (2.5,.5) {or};
    \node at (4, -0.4) {\( j \)};
  \node at (4.7, .5) {\( i \)};
  \upArr{3.5}{0}
  \begin{scope} [shift={(\w,.5)}]
   \vertex{1}{i} \vertex{2}{j}
    \lowerEdge{2}{1}
    \draw[<-,red,thick] (2,0) to (1.3,.7); \node at (1,1) {\( y \)};
    \node at (4.5,0) {\( \Longrightarrow \)};
    \node at (4.5,.7) {\( D^U_{i,j} \)};
  \end{scope}
  \begin{scope} [shift={(\w+6,.5)}]
 \vertex{2}{i}
  \node at (3,0) {$\emptyset$};
    \draw[<-,red,thick] (1.9,0.1) to (1.3,.7); \node at (1,1) {\( y \)};
  \end{scope}
\end{scope}

\begin{scope}[shift={(0,-6)}]
  \node at (-4,.5) {\( \tp=\emptyset \)};
  \begin{scope}[shift={(-.5,0)}]
  \tall{0}{-.5}  \cell{3}{0}
  \node at (2,.5) {or};
  \node at (0.7, -0.2) {\( j \)};
  \node at (1.2, 1) {\( i \)};
    \node at (3.5, -0.4) {\( j \)};
  \node at (4.2, .5) {\( i \)};
  \end{scope}
  \begin{scope}[shift={(5.5,0)}]
    \node at (-1,.5) {or};
   \short{-.5}{0}
  \node at (0, -0.4) {\( j \)};
  \node at (1.3, 0.4) {\( i \)};
  \end{scope}
  \begin{scope} [shift={(\w,.5)}]
\vertex{-.5}{i} \vertex{2}{j} 
    \draw[->,red,thick] (1.9,0.1) to (1.3,.7); \node at (1.2,1) {\( b \)};
      \draw[->,red,thick] (-.4,0.1) to (.2,.7); \node at (.4,1) {\( a \)};
    \node at (4.5,0) {\( \Longrightarrow \)};
    \node at (4.5,.7) {\( D^\emptyset_{i,j} \)};
  \end{scope}
  \begin{scope} [shift={(\w+8,.5)}]
 \vertex{-.5}{i} \vertex{2}{j} 
    \draw[->,red,thick] (2.1,0.1) to (2.7,.7); \node at (3,1) {\( a \)};
    \draw[->,red,thick] (-.6,0.1) to (-1.3,.7); \node at (-1.4,1) {\( b \)};
  \end{scope}
\end{scope}
\end{tikzpicture}
\caption{We give an illustration of the operator \( D^\tp_{i,j} \) acting on the arc diagram of $\tau$  to obtain $\tau'=D^\gamma_{i,j}(\tau)$ for each type  \( \tp \) for the corresponding corner $(i,j)$ in the RAT $R$. Let $a=\tau(i)$, $b=\tau(j)$, and $y=\tau^{-1}(j)$; arrows may represent upper arcs, lower arcs, or spiral arcs. All other arcs are left unchanged.}
  \label{fig:D tau}
\end{figure}

\begin{defn}\label{def:4}
  Suppose that \( R\in \RAT^+ \) has a corner \( (i,j) \) of type
  \( \tp \). We define \( D^\tp_{i,j}(R) \) to be the extended RAT
  \( R' \) obtained from \( R \) by doing the following:
  \begin{enumerate}
  \item If \( \gamma=L \), delete row strip $i$ containing it.
  \item If \( \gamma=U \), delete column strip $j$ containing it.
  \item If \( \gamma=\emptyset \), delete the corner $(i,j)$.
  \end{enumerate}
  For a fixed triple $(i,j,\tp)$, the operator $D_{i,j}^\tp$ is invertible; we denote its inverse $I_{i,j}^\tp$, which is obtained by reversing the procedure. 
\end{defn}

Note that if \(R\in\RAT^+(X)\) with \(X=\{a_1<\cdots<a_n\}\), then any
corner has the form \((i,j)=(a_k,a_{k+1})\) for some \(1\le k<n\). 
If \((i,j)\) has type \(\tp\), then \(D^\tp_{i,j}(R)\) lies in
\(\RAT^+(X\setminus\{a_k\})\), \(\RAT^+(X\setminus\{a_{k+1}\})\), or
\(\RAT^+(X)\) in the cases \( \gamma=L \), \( \gamma=U \), or \( \gamma=\emptyset \), respectively.

The notion of a corner in a RAT has a direct analogue for the signed permutation \(\tau\) associated to an assembl\'ee. Following the definition of \(\sh(\tau)\), corners correspond to adjacent shape-inversions, which we call \emph{shape-descents}.

\begin{defn}\label{def:3}
  Let \( X=\{a_1 < \cdots <a_n\} \), and let $\tau\in \AS(X)$ with
  \( \sh(\tau) = (w_1,\dots,w_n) \). A \emph{shape-descent} of
  \( \tau \) is a pair \( (a_k,a_{k+1}) \) for some \( 1\le k\le n-1 \)
  with \( w_k>w_{k+1} \). The \emph{type} \( \tp \) of the
  shape-descent \( (a_k,a_{k+1}) \) is defined by
  \[
    \tp =
    \begin{cases}
     L & \mbox{if \( \tau(a_{k}) = a_{k} \)},\\
     U & \mbox{if \( \tau(a_{k+1}) = a_k \)},\\
     \emptyset{} & \mbox{otherwise.}
    \end{cases}
  \]
\end{defn}

We now define an operator on signed permutations that mirrors the action of \(D^\tp_{i,j}\) on RAT. To emphasize the parallel definition, we use the same notation for this operator.

\begin{defn}\label{def:9}
  Suppose that \( \tau\in\AS(X)\) is a signed permutation that has a shape-descent \( (i,j) \) of type $\tp$. 
  Suppose \( \tau(i)=a \) and \( \tau(j)=b \):
  \[
  \tau =
  \begin{pmatrix}
\cdots & i & j & \cdots \\
\cdots & a & b & \cdots 
\end{pmatrix}.
\]
We define \( D^\tp_{i,j}(\tau) \) to be the signed permutation
  \( \tau' \) obtained from \( \tau \) by doing the following:
  \begin{enumerate}
  \item If \( \tp=L \), we must have \(a=i\). In this case, delete the bi-letter \( \biletter{i}{i} \).
  \item If \( \tp=U \), we must have \( b=i \), and \( \tau \) has a
    bi-letter \( \biletter{y}{\epsilon j} \), where
    \( \epsilon\in \{+1,-1\} \). In this case, delete the bi-letter
    \( \biletter{j}{i} \) and replace \( \biletter{y}{\epsilon j} \)
    by \( \biletter{y}{\epsilon i} \).
  \item If \( \tp=\emptyset \), then replace \( \biletter{i}{a} \) and
    \( \biletter{j}{b} \) by \( \biletter{i}{b} \) and
    \( \biletter{j}{a} \), respectively.
  \end{enumerate}
  See \cref{fig:D tau} for an illustration of the arc diagrams of $\tau$ and $\tau'$ under this operator. Denote by \( I^\tp_{i,j}(\tau') = \tau \), the inverse process
  of the operator \( D^\tp_{i,j} \).
  
\end{defn}

\begin{example}
  The following RAT \( R\in\RAT^+(X) \) and signed permutation \( \tau\in\AS(X) \) for $X=\{2,4,7,8\}$ satisfy \( \tau=\zrpt(R) \):
\begin{center}
  \begin{tikzpicture}[scale=1.5]
  \short{1}{1} \short{0}{1} \cell{1}{0} \cell{0}{0} \node at (0.500000000000000, -0.200000000000000) {\tiny 8};
\node at (1.50000000000000, -0.200000000000000) {\tiny 7};
\node at (2.20000000000000, 0.500000000000000) {\tiny 4};
\node at (2.60000000000000, 1.30000000000000) {\tiny 2};
 \slantUpArr{0}{1} \upArr{1}{0}
 \node[left] at (-1,1) {\( R= \)};
 \node[left] at (13,1) {\(\tau = \begin{pmatrix} 2 & 4 & 7 & 8 \\ -7 & 8 & 4 & 2 \end{pmatrix}=\)};
 \begin{scope}[shift={(13,1)}]
\vertex{1}{2} \vertex{2}{4} \vertex{3}{7} \vertex{4}{8} \oneSpiralEdge{1}{0.500000000000000}{3} \upperEdge{2}{4} \lowerEdge{3}{2} \lowerEdge{4}{1} 
 \end{scope}
\end{tikzpicture}
\end{center}  
Applying \( D^{U}_{4,7} \) to \( R \) and \( \tau \) yields the following results: 
\begin{center}
  \begin{tikzpicture}[scale=1.5]
\short{0}{1} \cell{0}{0} \node at (0.500000000000000, -0.200000000000000) {\tiny 8};
\node at (1.20000000000000, 0.500000000000000) {\tiny 4};
\node at (1.60000000000000, 1.30000000000000) {\tiny 2};
 \slantUpArr{0}{1}
 \node[left] at (-1,1) {\( D^{U}_{4,7} (R)= \)};
 \node[left] at (12,1) {\(D^{U}_{4,7} (\tau) = \begin{pmatrix} 2 & 4 & 8 \\ -4 & 8 & 2 \end{pmatrix}=\)};
 \begin{scope}[shift={(12,1)}]
\vertex{1}{2} \vertex{2}{4} \vertex{4}{8} \oneSpiralEdge{1}{0.500000000000000}{2} \upperEdge{2}{4} \lowerEdge{4}{1} 
 \end{scope}
\end{tikzpicture}
\end{center}  
\end{example}

\begin{lemma}\label{lem:2}
  Let \( R\in \RAT^+ \) and \( \tau=\zrpt(R) \). Suppose that the first
  corner of \( R \) is the cell \( (i,j) \) of type \( \tp \). Let
  \( R'=D_{i,j}^\tp(R) \) and \( \tau'=\zrpt(R') \). Then
  \( \tau'=D_{i,j}^\tp(\tau) \), i.e., the following diagram commutes:
\[
\begin{tikzcd}[row sep=large, column sep = huge]
&R \arrow[r, "D^{\tp}_{i,j}"] \arrow[d,"\zrpt"] & R' \arrow[d,"\zrpt"]\\
&\tau \arrow[r,"D^{\tp}_{i,j}"] & \tau'
\end{tikzcd}
\]
Moreover, the first shape-descent of \( \tau \) is the pair \( (i,j) \)
of type \( \tp \), and
\begin{align}
\label{eq:2}  \fcell(R) &= \fcell(R') + \epsilon, \\
\label{eq:3}  \cro(\tau) &= \cro(\tau') + \epsilon,
\end{align}
where \( \epsilon=1 \) if the cell \( (i,j) \) is a free cell and
\( \epsilon=0 \) otherwise.
\end{lemma}

\begin{proof}
 First, as \(\sh(\tau)\) matches the boundary of \(R\), \((i,j)\) is a corner of \(R\) if and only if \((i,j)\) is a shape-descent of \(\tau\). We compare the type \(\tp'\) of \((i,j)\) in \(\tau\) to \(\tp\): 
 \begin{itemize}
 \item if \(\tp=L\), then strip \(i\) in \(R\) has no turning cells, hence \(\tau(i)=i\) and \( \gamma'=L \);
 \item if \(\tp=U\), then \((i,j)\) is the only turning cell in strip \(j\) of \(R\), hence \(\tau(j)=i\) and \( \gamma'=U \);
 \item if \( \tp=\emptyset \), then \(\tau(i)\ne i\) and \(\tau(j)\ne i\), hence \( \tp'=\emptyset \).
 \end{itemize}
 Therefore, we always have \(\tp=\tp'\), and \(D_{i,j}^\tp(\tau)\) is well-defined. 
 
The operator \(D^\tp_{i,j}\) acts on \(R\) by deleting either a row strip ending with a left-arrow, a column strip ending with an up-arrow, or a free corner, and hence affects only the zigzag paths interacting with the corner \((i,j)\). Tracing these paths shows that the resulting signed permutation \(\zeta(R')\) is obtained from \(\tau\) through exactly the same case (1), (2), or (3) of the procedure in \cref{def:9}. Thus the diagram commutes.

We now prove \eqref{eq:2}. In cases (1) and (2) of \cref{def:4}, a strip with no free cells is removed from \(R\) to form \(R'\), so \(\fcell(R)=\fcell(R')\). On the other hand, case (3) occurs if and only if \((i,j)\) is a free cell, in which case deleting it removes exactly one free cell: \(\fcell(R')=\fcell(R)-1\).

Finally, we show \eqref{eq:3}. Since \(i\) and \(j\) are adjacent vertices in the arc diagram, only the crossings arising from two arcs incident to one or both of them may be affected by \(D_{i,j}^\gamma\). In case (1) of \cref{def:9}, removing the loop at \(i\) leaves all other arcs and their crossings unchanged. Case (2) corresponds to merging the vertices \(i\) and \(j\) in the arc diagram of \(\tau\) to obtain \(\tau'\), which again does not change the number of crossings, since the lower arc from \(j\) to \(i\) has no crossings in \(\tau\). Thus \(\cro(\tau)=\cro(\tau')\) in both cases. 

It remains to analyze case (3) of \cref{def:9} with \(\tp=\emptyset\). In this case, we have one of the following possibilities for \(\tau(i)\) and \( \tau(j) \), corresponding to whether the corner \((i,j)\) is a square, tall rhombus, or short rhombus, respectively:
\[
(\tau(i)>i,0<\tau(j)<j)\qquad\text{or}\qquad (\tau(i)>i,\tau(j)<0)\qquad\text{or}\qquad(\tau(i)<0,0<\tau(j)<j). 
\] 
Note that the first two cases can be combined into the conditions
\( \tau(i)>i \) and \( \tau(j)<j \). We write \(\tilde{\tau}(j)\) to
mean \(\tau(j)\) if \(j\not\in\neg(\tau)\), and \(-d\) if
\(j\in\neg(\tau)\), where \((-d,0)\) is the intermediate vertex of the
spiral arc from \(j\) to \(|\tau(j)|\). Similarly, we write
\(\tilde{\tau}^{-1}(i)\) to mean \(\tau^{-1}(i)\) if
\(|\tau^{-1}(i)|\not\in\neg(\tau)\), and \(-d'\) if
\(|\tau^{-1}(i)|\in\neg(\tau)\), where \((-d',0)\) is the intermediate
vertex of the spiral arc from \( |\tau^{-1}(i)| \) to \( i \). In particular, if
\( \tau(j)<j \), there is a lower arc \( (\tilde{\tau}(j),j) \). Then
we have the following cases:
\begin{itemize}
\item Suppose \(\tau(i)>i\) and \( \tau(j)<j \). If
  \(\tau^{-1}(i)<i\), then \(((\tilde{\tau}^{-1}(i),i),(i,\tau(i)))\)
  is an upper crossing in \(\tau\). This becomes the pair of arcs
  \(((\tilde{\tau}^{-1}(i),i),(j,\tau(i)))\) in \(\tau'\), which is no
  longer a crossing:
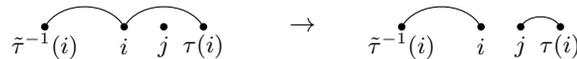
\begin{center}
  \begin{tikzpicture}[scale=1.5]
  \begin{scope} [shift={(6,.5)}]
    \vertex{0}{\tilde{\tau}^{-1}(i)} \vertex{2}{i} \vertex{3}{j} \vertex{4}{\tau(i)}
    \upperEdge{0}{2} \upperEdge{2}{4}
    \node at (6.5,0) {\( \rightarrow \)};
  \end{scope}
  \begin{scope} [shift={(15,.5)}]
    \vertex{0}{\tilde{\tau}^{-1}(i)} \vertex{2}{i} \vertex{3}{j} \vertex{4}{\tau(i)}
    \upperEdge{0}{2} \upperEdge{3}{4}
  \end{scope}
\end{tikzpicture}
\end{center}
  If \(\tau^{-1}(i)>i\), then
  \(((\tilde{\tau}(j),j),(i,\tau^{-1}(i)))\) is a lower crossing in
  \(\tau\). This becomes the pair of arcs
  \(((\tilde{\tau}(j),i),(i,\tau^{-1}(i)))\) in \(\tau'\), which is no
  longer a crossing:
  
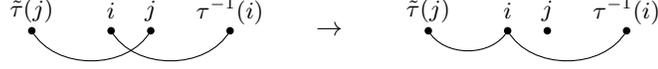
\begin{center}
  \begin{tikzpicture}[scale=1.5]
  \begin{scope} [shift={(6,.5)}]
    \Uvertex{0}{\tilde{\tau}(j)} \Uvertex{2}{i} \Uvertex{3}{j} \Uvertex{5}{\tau^{-1}(i)}
    \lowerEdge{3}{0} \lowerEdge{5}{2}
    \node at (7.5,0) {\( \rightarrow \)};
  \end{scope}
  \begin{scope} [shift={(16,.5)}]
    \Uvertex{0}{\tilde{\tau}(j)} \Uvertex{2}{i} \Uvertex{3}{j} \Uvertex{5}{\tau^{-1}(i)}
    \lowerEdge{2}{0} \lowerEdge{5}{2}
  \end{scope}
\end{tikzpicture}
\end{center}

\item Suppose \( \tau(i)<0 \) and \( 0<\tau(j)<j \). Then
  \(((\tilde{\tau}(i),i),(\tau(j),j))\) is a lower crossing in
  \(\tau\). This becomes the pair of arcs
  \(((\tilde{\tau}(i),j),(\tau(j),i)) \) in \(\tau'\), which is no
  longer a crossing:
  
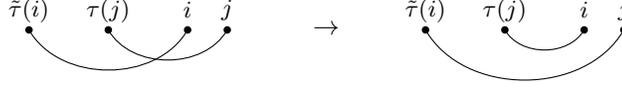
\begin{center}
  \begin{tikzpicture}[scale=1.5]
  \begin{scope} [shift={(6,.5)}]
    \Uvertex{0}{\tilde{\tau}(i)} \Uvertex{2}{\tau(j)} \Uvertex{4}{i} \Uvertex{5}{j}
    \lowerEdge{4}{0} \lowerEdge{5}{2}
    \node at (7.5,0) {\( \rightarrow \)};
  \end{scope}
  \begin{scope} [shift={(16,.5)}]
    \Uvertex{0}{\tilde{\tau}(i)} \Uvertex{2}{\tau(j)} \Uvertex{4}{i} \Uvertex{5}{j}
    \lowerEdge{4}{2} \lowerEdge{5}{0}
  \end{scope}
\end{tikzpicture}
\end{center}
\end{itemize}
It is straightforward to verify that any other crossings arising from two arcs incident to \(i\) and \(j\) remain unchanged. Thus we obtain \(\cro(\tau)=\cro(\tau')+1\) in this case as well, completing the proof of \eqref{eq:3}.
\end{proof}

\begin{prop}\label{pro:1}
Let \(X=\{a_1<\cdots<a_n\}\) be a set of positive integers. 
  The map \( \zrpt:\RAT^+(X)\to \AS(X) \) is a bijection. Moreover,  we have
  \[
    \fcell(R)=\cro(\tau).
  \]
\end{prop}

\begin{proof}
  Since
  \[
    |\RAT^+(X)| = \sum_{r\ge1} |\RAT^+(n,r)|= \sum_{r\ge1} |\A(n,r)|
    = \sum_{r\ge1} |\AS(n,r)| = |\AS(X)|,
  \]
  it suffices to show that \( \zrpt \) is injective. Let
  \( R\in \RAT^+(X) \) and set \( \tau=\zrpt(R) \).
  
  We define a sequence \((R^{(0)},R^{(1)},\ldots)\) of RAT by
  sequentially applying the operators \(D_{i,j}^\tp\) to reduce the
  number of tiles in \(R\). Define \(R^{(0)}=R\), and for \(k\ge 0\),
  if \(R^{(k)}\) has at least one tile, let \((i_k,j_k)\) be the
  topmost corner of \(R^{(k)}\) with type \(\tp_k\), and set
\[
R^{(k+1)}=D^{\tp_k}_{i_k,j_k}(R^{(k)}).
\]
Let \(m\) be the first index such that \(R^{(m)}\) has no tiles.

Define similarly a sequence \((\tau^{(0)},\tau^{(1)},\dots)\)
of signed permutations by \(\tau^{(0)}=\tau\), and for
\(k\ge 0\), if \(\tau^{(k)}\) has a shape-descent, let \((i_k,j_k)\) be its
leftmost shape-descent, and set
\[
\tau^{(k+1)}:=D^{\tp_k}_{i_k,j_k}(\tau^{(k)}),
\]
where \(\tp_k\) is the type of \(\tau^{(k)}\).

By \cref{lem:2}, for each \(k<m\),  assuming \(\zeta(R^{(k)})=\tau^{(k)}\), the topmost corner of \(R^{(k)}\)
corresponds to the leftmost shape-descent of \(\tau^{(k)}\) with the same labels
\((i_k,j_k)\) and type \(\tp_k\), and hence
\[
\zeta(R^{(k+1)})=\zeta(D^{\tp_k}_{i_k,j_k}(R^{(k)}))
= D^{\tp_k}_{i_k,j_k}(\zeta(R^{(k)}))=D^{\tp_k}_{i_k,j_k}(\tau^{(k)})
=\tau^{(k+1)}.
\]
Thus by induction on \(k\), we have \(\tau^{(k)}=\zeta(R^{(k)})\) for all
\(0\le k\le m\), and in particular, \(\tau^{(m)}=\zeta(R^{(m)})\).

Suppose \(R^{(m)}\in\RAT^+(X')\) for \(X'=\{b_1<\cdots<b_\ell\}\).
Since \( R^{(m)} \) has no tiles, it is the unique extended
RAT whose border consists of \(r\) diagonal edges followed by
\(\ell-r\) vertical edges (from northeast to southwest). By definition
of \(\zeta\), we have
\[
\tau^{(m)}=\zeta(R^{(m)})=\bar b_1\,\cdots\,\bar b_r\,b_{r+1}\,\cdots\,b_\ell.
\]
In particular, \(\tau^{(m)}\) has no shape-descents and \(\cro(\tau^{(m)})=0\).

Since each \(D^{\tp_k}_{i_k,j_k}\) is invertible, we recover \(R\) from \(R^{(m)}\) by applying the inverses \(I^{\tp_{m-1}}_{i_{m-1},j_{m-1}}, \ldots,I^{\tp_0}_{i_0,j_0}\). Therefore \(\tau\) determines \(R\), and \(\zeta\) is injective, and hence a bijection.

Finally, iterating \cref{lem:2} for \(k=1,2,\ldots,m\)
yields 
\[
\cro(\tau) =\cro(\tau^{(m)})+\fcell(R)-\fcell(R^{(m)}).
\]
Because \(R^{(m)}\) has no tiles, we have \(\fcell(R^{(m)})=0\), and from the above, \(\cro(\tau^{(m)})=0\). Hence \(\cro(\tau)=\fcell(R)\), as claimed.
\end{proof}

Combining \cref{lem:4,lem:5,lem:RLmin,lem:LRmax}, and \cref{pro:1}, we complete the proof of \Cref{thm:zigzag-RPT}.

\section{Marked Laguerre histories}
\label{sec:laguerre-histories}

In \cite{Corteel2020a}, Corteel and Nunge gave a bijection between
marked Laguerre histories and signed permutations, generalizing the
Fran\c{c}on--Viennot bijection \cite{Francon1979}. In this section, we
propose a new bijection between these objects by generalizing the
Foata--Zeilberger bijection \cite{Foata1990a} as presented in
\cite{Corteel2007a}. We also show that a signed permutation admits a natural decomposition into a pair consisting of an assembl\'{e}e and a
permutation. As a consequence, we show that the generating function
for the marked Laguerre histories has the factor
\( [r]_q! := [1]_q[2]_q \cdots [r]_q \), where
\( [i]_q = (1-q^i)/(1-q) \). This answers a question of
\cite{Corteel2020a}.

\subsection{Definitions}

A \emph{Motzkin path} of length \(n\) is a lattice path from \((0,0)\)
to \((n,0)\), consisting of \emph{up steps} \( (1,1) \),
\emph{horizontal steps} \( (1,0) \), and \emph{down steps}
\( (1,-1) \), that never goes below the \( x \)-axis. We denote by
\(\nearrow\), \(\longrightarrow\), and \(\searrow\) an up step, a
horizontal step, and a down step, respectively. A \emph{2-Motzkin
  path} is a Motzkin path such that every horizontal step is either a
\emph{solid} horizontal step \(\longrightarrow\) or a \emph{dashed}
horizontal step \(\dashrightarrow\).

A Laguerre history is a labeled 2-Motzkin path introduced by Viennot
in~\cite{ViennotLN} (referred to as a \emph{restricted Laguerre history} therein); see also~\cite{Medicis1994}.

\begin{defn}
  A \emph{Laguerre history} of length \(n\) is a labeled
  2-Motzkin path of length \(n\) such that
  \begin{itemize}
  \item each solid horizontal step \(\longrightarrow\) starting at height \(h\) has
    a label \( i\in [0,h]\),
  \item each dashed horizontal step \(\dashrightarrow\) starting at
    height \(h\) has a label \( i \in [0,h-1]\),
  \item each up step \(\nearrow\) starting at height \(h\) has a label
    \( i\in [0,h] \),
  \item each down step \(\searrow\) starting at height \(h\) has a label
    \( i\in [0,h-1] \).
  \end{itemize}
\end{defn}

Note that a Laguerre history cannot have a dashed horizontal step or a
down step starting at height \( h=0 \) due to the condition on its
label \( i\in [0,h-1] \). Hence, the first step of a Laguerre history
is always an up step or a solid horizontal step.

Now we define marked Laguerre histories introduced by Corteel and
Nunge \cite{Corteel2020a}. 

\begin{defn}
  A {\em marked Laguerre history} is a Laguerre history in which every step may be \emph{marked}, and such that the first step is always marked. The \emph{weight} of a marked Laguerre history \( H \) is defined to
  be \( q^{l(H)+m(H)} \), where \( l(H) \) is the sum of the labels and $m(H)$ is the sum of the starting heights of the marked steps. We denote by \( \MLH(n,r) \) the set of marked Laguerre histories of length \( n \) with \( r \) marks.
\end{defn}

See \Cref{exChemin} for an example of a marked Laguerre history.

\begin{figure}
  \centering
\begin{tikzpicture}[scale=1.5]
  \Grid{12}{3}
\HS{0}{0} \HLabel{0}{0}{0} \US{1}{0} \ULabel{1}{0}{0} \HS{2}{1} \HLabel{2}{1}{0}  \US{3}{1} \ULabel{3}{1}{1} \HS{4}{2} \HLabel{4}{2}{2} \DS{5}{2} \DLabel{5}{2}{0}  \DS{6}{1} \DLabel{6}{1}{0} \US{7}{0} \ULabel{7}{0}{0} \US{8}{1} \ULabel{8}{1}{1} \DS{9}{2} \DLabel{9}{2}{1} \dHS{10}{1} \HLabel{10}{1}{0} \DS{11}{1} \DLabel{11}{1}{0}
\HMark{0}{0}{x} \DMark{6}{1}{x}\UMark{3}{1}{x}\DMark{5}{2}{x}\UMark{7}{0}{x}\DMark{9}{2}{x}\HMark{10}{1}{x}
\end{tikzpicture}\hspace{0.3in}
\begin{tikzpicture}[scale=1.5]
\Grid{12}{3}
\HS{0}{0} \HLabel{0}{0}{0} \US{1}{0}  \HS{2}{1} \HLabel{2}{1}{0}  \US{3}{1} \HS{4}{2} \HLabel{4}{2}{2} \DS{5}{2} \DLabel{5}{2}{01}  \DS{6}{1} \DLabel{6}{1}{00} \US{7}{0}  \US{8}{1}  \DS{9}{2} \DLabel{9}{2}{11} \dHS{10}{1} \HLabel{10}{1}{0} \DS{11}{1} \DLabel{11}{1}{00}
\HMark{0}{0}{a} \UMark{1}{0}{d}\UMark{3}{1}{ad}\UMark{7}{0}{a}\UMark{8}{1}{d}\HMark{10}{1}{a}
\end{tikzpicture}
\caption{We show an example of a marked Laguerre history
  \( H\in \MLH(12,7) \) (left) and its corresponding modified marked Laguerre history $H'\in\MLH^*(12,7)$ (right). The label of a step is written above the
  step. The mark of type \( x \) on a step is indicated as \( x \)
  below it. In $H'$, for simplicity, the label \( (a,b) \) of a down step is
  written as \( ab \). The sum of the labels is \( 5 \), and the sum of
  the starting heights of the steps with each mark is
  \( 0+1+2+1+0+2+1=7 \) in $H$ and $0+0+(1+1)+0+1+1=4$ in $H'$, and there are three marks of type
  \( d \). Hence, the weight of \( H \) and $H'$ is \( q^{5+7} = q^{12} \).}
  \label{exChemin}
\end{figure}
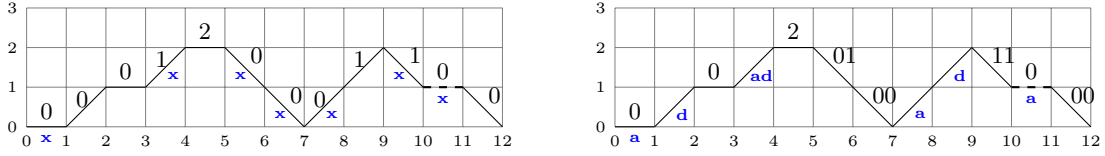

\subsection{A bijection from signed permutations to marked Laguerre histories}

Let \( \SP(n,r) \) denote the set of signed permutations
\( \tau\in S^\pm_n \) such that \( |\neg(\tau)|=r \) and
\( 1\in\neg(\tau) \). In this subsection, we define a bijection
\( \phi: \SP(n,r)\to \MLH(n,r) \).

We represent elements of $\SP(n,r)$ with ordinary arc diagrams. It will be convenient to work with a generalized class of arc diagrams corresponding to partially completed configurations. We first introduce some terminology.

An \emph{open upper arc} at \( a \) is a segment connecting
\( (a,0) \) and \( (a+\epsilon,\epsilon) \) for some \( \epsilon>0 \).
Similarly, an \emph{open lower arc} at \( a \) is a segment
connecting \( (a,0) \) and \( (a+\epsilon,-\epsilon) \).

Let \( \tau\in \SP(n,r) \). The \emph{partial arc diagram
  \( D_i(\tau) \)} of \( \tau \) up to \( i \) is the diagram obtained
from the full arc diagram \( D(\tau) \) of \( \tau \) by removing all upper
and lower arcs \( (a,b) \) with \( a>i \), and replacing each upper
(resp.~lower) arc \( (a,b) \) with \( a\le i<b \) by an open
upper (resp.~lower) arc at \( a \).

For $b\in[n]$, let \( \cro^+_b(\tau) \) denote the number of pairs $((a,c),(b,d))$ of upper arcs satisfying $a<c\leq b<d$, and let \( \cro^-_b(\tau) \) denote the
number of pairs \( ((a,b),(c,d)) \) of lower arcs satisfying \( a<c<b<d \), corresponding to configurations (U) and (L) below, respectively:
 \begin{center}
      \begin{tikzpicture}[scale=1.5]
      \node at (-1,1) {(U)};
        \draw[dotted] (-1,0) -- (7,0);
        \vertex{0}{a} \vertex{2}{c} \Rvertex{4}{b} \vertex{6}{d} 
        \upperEdge{2}{6} 
        \upperEdge{0}{4}
        \node at (7.5,0) {or};
        \begin{scope}[xshift=3.3cm]
         \draw[dotted] (-1,0) -- (5,0);
        \vertex{0}{a} \Rvertex{2}{b=c} \vertex{4}{d} 
        \upperEdge{2}{4} 
        \upperEdge{0}{2}
        \end{scope}
        \begin{scope}[xshift=6cm]
        \node at (-1,1) {(L)};
        \draw[dotted] (-1,0) -- (7,0);
        \vertex{0}{a} \vertex{2}{c} \Rvertex{4}{b} \vertex{6}{d} 
        \lowerEdge{6}{2} 
        \lowerEdge{4}{0}
        \end{scope}
      \end{tikzpicture}
    \end{center} 
Writing \( \cro_b(\tau) = \cro^+_b(\tau) + \cro^-_b(\tau) \), we have
\( \cro(\tau) = \sum_{b=1}^{n} \cro_b(\tau) \).

\begin{example}\label{ex:crossings} For the signed permutation 
\( \tau \) in \Cref{fig:arc_diagram} with $\cro(\tau)=6$, we have:
\begin{align*}
(\cro_i^+(\tau):1\leq i\leq 6)&=(0,2,0,1,0,0),\\ 
(\cro_i^-(\tau):1\leq i\leq 6)&=(0,0,0,2,1,0),\\
 (\cro_i(\tau):1\leq i\leq 6) &= (0,2,0,3,1,0).
\end{align*}
\end{example}

To facilitate our bijection, we define a \emph{modified marked Laguerre history}, which is an auxiliary object in simple bijection with ordinary marked Laguerre histories $\MLH(n,r)$.

\begin{defn}
  A \emph{modified marked Laguerre history} of length \(n\) is a labeled
  2-Motzkin path of length \(n\) such that
  \begin{itemize}
  \item each solid horizontal step \(\longrightarrow\) starting at height \(h\) has
    a label \( i\in [0,h]\),
  \item each dashed horizontal step \(\dashrightarrow\) starting at
    height \(h\) has a label \( i \in [0,h-1]\),
  \item each down step \(\searrow\) starting at height \(h\) has a pair of labels
    \( (i,j)\in [0,h-1]^2 \).
  \end{itemize}
  The steps may have marks of types $a$ and $d$, subject to the following
  conditions:
  \begin{itemize}
  \item Every up step may have a mark of type~$a$, a mark of type~$d$, or both.
  \item Every (solid or dashed) horizontal step may have a mark of type~$a$.
  \item The first step must have a mark of type \( a \).
    (If it is an up step, it may also have a mark of type \( d \).)
  \end{itemize}
 The \emph{weight} of a modified marked Laguerre history \( H' \) is defined to
  be \( q^{\ell(H')+m(H')+d(H')} \), where \( \ell(H') \) is the sum of the labels, $m(H')$ is the sum of the starting heights of the steps corresponding to each mark (if an up step has marks of both type $a$ and $d$, it gets counted twice), and $d(H')$ is the number of marks of type $d$. Denote the set of modified marked Laguerre histories with $n$ steps and a total of $r$ marks by $\MLH^*(n,r)$.
\end{defn}

There is a straightforward weight-preserving bijection between $\MLH^*(n,r)$ and $\MLH(n,r)$, which is a variant of the standard bijection on Laguerre histories. In a 2-Motzkin path, each up step starting from height $h-1$ is paired with the leftmost down step to its right starting at height $h$. Under this bijection, the label $(i,j)\in[0,h-1]^2$ on a down step of height $h$ in $H'\in\MLH^*(n,r)$ is split into  a label $i$ on the same down step and a label $j$ on its corresponding up step in $H\in\MLH(n,r)$, and marks of type $d$ on up steps in $H'$ are transferred to their corresponding down steps in $H$. Consequently, the bijection is weight-preserving:
\[
q^{\ell(H)+m(H)} = q^{\ell(H')+m(H')+d(H')}.
\]

Now we describe the bijection \( \phi: \SP(n,r)\to \MLH^*(n,r) \).

\begin{algorithm}\label{alg:SP-MLH}
  Let \( \tau\in \SP(n,r) \). We will construct the corresponding
  modified marked Laguerre history \( \phi(\tau)\in\MLH^*(n,r) \). For \( i=1,2,\dots,n \), the
  \( i \)th step of \( H=\phi(\tau) \) and its labels and marks are determined as
  follows. Here, \( h \) is the number of open upper arcs in
  \( D_{i-1}(\tau) \). In each case, we display the arcs, open upper arcs, and
  open lower arcs connected to \( i \) in
  \( D_i(\tau) \), and the corresponding \( i \)th step in \( H \).
  Some open arcs are drawn longer to indicate future crossings (in the partial arc diagram $D_j(\tau)$ for some $j>i$).
 \begin{description}
 \item[Case 1] If \( |\tau^{-1}(i)|,|\tau(i)|<i \), the \( i \)th step
   of \( H \) is an up step.
   \begin{description}
   \item[Case 1-1] If \( \tau(i), \tau^{-1}(i)>0 \), the \( i \)th step has no marks.
     
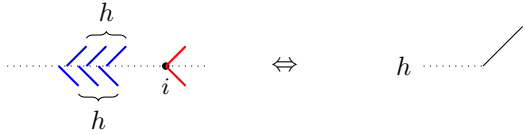
\begin{center}
      \begin{tikzpicture}[scale=1.5]
        \draw[dotted] (2,0) -- (7,0);
        \vertex{6}{i} 
        \lowerHalfEdge{6} \upperHalfEdge{6}
        \shortUpperHalfEdge{3.5} \shortUpperHalfEdge{4} \shortUpperHalfEdge{4.5}
        \shortLowerHalfEdge{3.3} \shortLowerHalfEdge{3.8} \shortLowerHalfEdge{4.3}
        \upperBrace{4}{.7}{\( h \)}
        \lowerBrace{3.8}{-.7}{\( h \)}
        \node at (9,0) {\( \Leftrightarrow \)};
        \node at (12,0) {\( h \)};
        \draw[dotted] (12.5,0) -- (14,0);
        \US{14}{0} 
      \end{tikzpicture}
    \end{center} 
       \item[Case 1-2] If \( \tau^{-1}(i)<0<\tau(i) \), the \( i \)th step has a mark of type $a$.
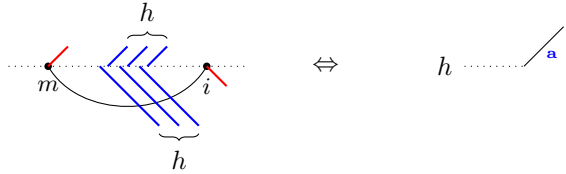
     \begin{center}
      \begin{tikzpicture}[scale=1.5]
        \draw[dotted] (1,0) -- (7,0);
        \vertex{2}{m} \vertex{6}{i} 
        \lowerEdge{6}{2} \lowerHalfEdge{6} \upperHalfEdge{2}
        \shortUpperHalfEdge{3.5} \shortUpperHalfEdge{4} \shortUpperHalfEdge{4.5}
        \longLowerHalfEdge{3.3} \longLowerHalfEdge{3.8} \longLowerHalfEdge{4.3}
        \upperBrace{4}{.7}{\( h \)}
        \lowerBrace{4.8}{-1.7}{\( h \)}
        \node at (9,0) {\( \Leftrightarrow \)};
        \node at (12,0) {\( h \)};
        \draw[dotted] (12.5,0) -- (14,0);
        \US{14}{0} \UMark{14}{0}{a}
      \end{tikzpicture}
    \end{center} 
   \item[Case 1-3] If \( \tau(i)<0<\tau^{-1}(i) \), the \( i \)th step has a mark of type $d$.
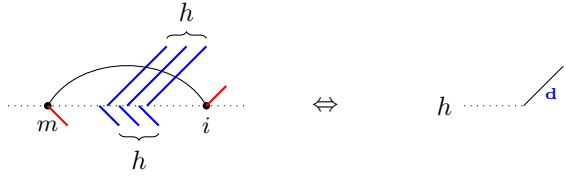
     \begin{center}
      \begin{tikzpicture}[scale=1.5]
        \draw[dotted] (1,0) -- (7,0);
        \vertex{2}{m} \vertex{6}{i} 
        \upperEdge{2}{6} \upperHalfEdge{6} \lowerHalfEdge{2}
        \longUpperHalfEdge{3.5} \longUpperHalfEdge{4} \longUpperHalfEdge{4.5}
        \shortLowerHalfEdge{3.3} \shortLowerHalfEdge{3.8} \shortLowerHalfEdge{4.3}
        \upperBrace{5}{1.7}{\( h \)}
        \lowerBrace{3.8}{-.7}{\( h \)}
        \node at (9,0) {\( \Leftrightarrow \)};
        \node at (12,0) {\( h \)};
        \draw[dotted] (12.5,0) -- (14,0);
        \US{14}{0} \UMark{14}{0}{d}
      \end{tikzpicture}
    \end{center} 
   \item[Case 1-4] If \( \tau(i), \tau^{-1}(i)<0 \), the \( i \)th
     step has a mark of type $a$ and a mark of type $d$.
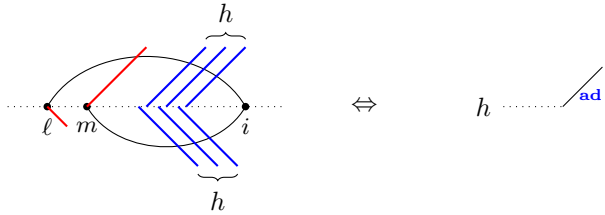
     \begin{center}
      \begin{tikzpicture}[scale=1.5]
        \draw[dotted] (0,0) -- (7,0);
        \vertex{1}{\ell} \vertex{2}{m} \vertex{6}{i} 
        \upperEdge{1}{6} \lowerEdge{6}{2}
        \longRedUpperHalfEdge{2} \lowerHalfEdge{1}
        \longUpperHalfEdge{3.5} \longUpperHalfEdge{4} \longUpperHalfEdge{4.5}
        \longLowerHalfEdge{3.3} \longLowerHalfEdge{3.8} \longLowerHalfEdge{4.3}
        \upperBrace{5}{1.7}{\( h \)}
        \lowerBrace{4.8}{-1.7}{\( h \)}
        \node at (9,0) {\( \Leftrightarrow \)};
        \node at (12,0) {\( h \)};
        \draw[dotted] (12.5,0) -- (14,0);
        \US{14}{0} \UMark{14}{0}{ad}
      \end{tikzpicture}
    \end{center} 
   \end{description}
 \item[Case 2] If \( |\tau(i)| = i \), the \( i \)th step of
     \( H \) is a solid horizontal step.
   \begin{description}
   \item[Case 2-1] If \( \tau(i)>0 \), the \( i \)th step has a
     label \( 0 \) and no marks.
     
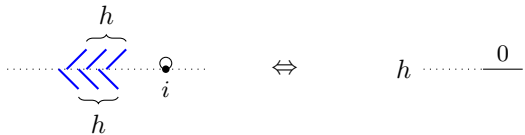
\begin{center}
      \begin{tikzpicture}[scale=1.5]
        \draw[dotted] (2,0) -- (7,0);
        \vertex{6}{i} \upperLoop{6}
        \shortUpperHalfEdge{3.5} \shortUpperHalfEdge{4} \shortUpperHalfEdge{4.5}
        \shortLowerHalfEdge{3.3} \shortLowerHalfEdge{3.8} \shortLowerHalfEdge{4.3}
        \upperBrace{4}{.7}{\( h \)}
        \lowerBrace{3.8}{-.7}{\( h \)}
        \node at (9,0) {\( \Leftrightarrow \)};
        \node at (12,0) {\( h \)};
        \draw[dotted] (12.5,0) -- (14,0);
        \HS{14}{0} \HLabel{14}{0}{0}
      \end{tikzpicture}
    \end{center} 
   \item[Case 2-2] If \( \tau(i)<0 \), the \( i \)th step has a
     label \( h \) and a mark of type $a$.
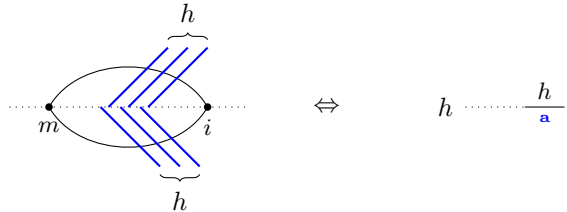
     \begin{center}
      \begin{tikzpicture}[scale=1.5]
        \draw[dotted] (1,0) -- (7,0);
       \vertex{2}{m} \vertex{6}{i} 
        \upperEdge{2}{6} \lowerEdge{6}{2}
        \longUpperHalfEdge{3.5} \longUpperHalfEdge{4} \longUpperHalfEdge{4.5}
        \longLowerHalfEdge{3.3} \longLowerHalfEdge{3.8} \longLowerHalfEdge{4.3}
        \upperBrace{5}{1.7}{\( h \)}
        \lowerBrace{4.8}{-1.7}{\( h \)}
        \node at (9,0) {\( \Leftrightarrow \)};
        \node at (12,0) {\( h \)};
        \draw[dotted] (12.5,0) -- (14,0);
        \HS{14}{0} \HLabel{14}{0}{\( h \)} \HMark{14}{0}{a}
      \end{tikzpicture}
    \end{center} 
   \end{description}

 \item[Case 3] If \( |\tau^{-1}(i)|<i<|\tau(i)| \), the
   \( i \)th step of \( H \) is a solid horizontal step.
   \begin{description}
   \item[Case 3-1] If \( \tau(i)>0 \), the \( i \)th step of \( H \)
     has a label \( k=\cro^+_i(\tau) \) and no marks.
     
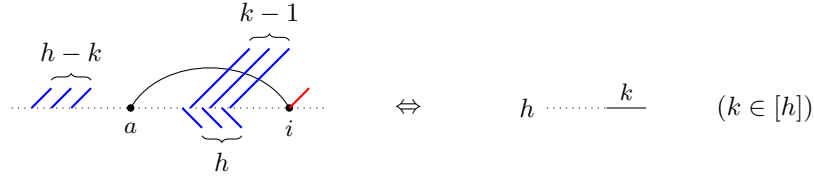
\begin{center}
      \begin{tikzpicture}[scale=1.5]
        \draw[dotted] (-1,0) -- (7,0);
        \vertex{2}{a} \vertex{6}{i} 
        \upperEdge{2}{6} 
        \upperHalfEdge{6} 
        \shortUpperHalfEdge{-.5} \shortUpperHalfEdge{0} \shortUpperHalfEdge{.5}
        \upperBrace{0}{.7}{\( h-k \)}
        \upperBrace{5}{1.7}{\( k-1 \)}
        \longUpperHalfEdge{3.5} \longUpperHalfEdge{4} \longUpperHalfEdge{4.5}
        \shortLowerHalfEdge{3.3} \shortLowerHalfEdge{3.8} \shortLowerHalfEdge{4.3}
        \lowerBrace{3.8}{-.7}{\( h \)}
        \node at (9,0) {\( \Leftrightarrow \)};
        \node at (12,0) {\( h \)};
        \draw[dotted] (12.5,0) -- (14,0);
        \HS{14}{0} \HLabel{14}{0}{\( k \)}
        \node at (18,0) {\( (k\in [h]) \)};
      \end{tikzpicture}
    \end{center} 
   \item[Case 3-2] If \( \tau(i)<0 \), the \( i \)th step of \( H \)
     has a label \( k=\cro^+_i(\tau) \) and a mark of type $a$.
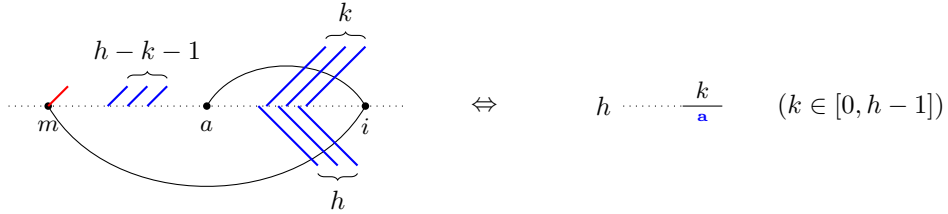
     \begin{center}
      \begin{tikzpicture}[scale=1.5]
        \draw[dotted] (-3,0) -- (7,0);
        \vertex{-2}{m}
        \vertex{2}{a} \vertex{6}{i} 
        \upperEdge{2}{6} \lowerEdge{6}{-2} 
        \upperHalfEdge{-2} 
        \shortUpperHalfEdge{-.5} \shortUpperHalfEdge{0} \shortUpperHalfEdge{.5}
        \upperBrace{0}{.7}{\( h-k-1 \)}
        \upperBrace{5}{1.7}{\( k \)}
        \longUpperHalfEdge{3.5} \longUpperHalfEdge{4} \longUpperHalfEdge{4.5}
        \longLowerHalfEdge{3.3} \longLowerHalfEdge{3.8} \longLowerHalfEdge{4.3}
        \lowerBrace{4.8}{-1.7}{\( h \)}
        \node at (9,0) {\( \Leftrightarrow \)};
        \node at (12,0) {\( h \)};
        \draw[dotted] (12.5,0) -- (14,0);
        \HS{14}{0} \HLabel{14}{0}{\( k \)} \HMark{14}{0}{a}
        \node at (18.5,0) {\( (k\in [0,h-1]) \)};
      \end{tikzpicture}
    \end{center} 
   \end{description}

 \item[Case 4] If \( |\tau(i)|<i<|\tau^{-1}(i)| \), the
   \( i \)th step of \( H \) is a dashed horizontal step.
   \begin{description}
   \item[Case 4-1] If \( \tau^{-1}(i)>0 \), the \( i \)th step of
     \( H \) has a label \( k=\cro^-_i(\tau) \) and no marks.
     
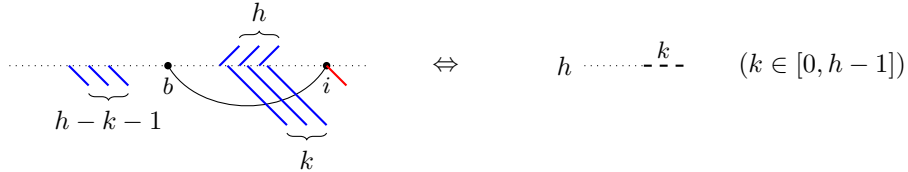
\begin{center}
      \begin{tikzpicture}[scale=1.5]
        \draw[dotted] (-2,0) -- (7,0);
        \vertex{2}{b} \vertex{6}{i} 
        \lowerEdge{6}{2} 
        \lowerHalfEdge{6} 
        \shortLowerHalfEdge{-.5} \shortLowerHalfEdge{0} \shortLowerHalfEdge{.5}
        \lowerBrace{0}{-.7}{\( h-k-1 \)}
        \lowerBrace{5}{-1.7}{\( k \)}
        \longLowerHalfEdge{3.5} \longLowerHalfEdge{4} \longLowerHalfEdge{4.5}
        \shortUpperHalfEdge{3.3} \shortUpperHalfEdge{3.8} \shortUpperHalfEdge{4.3}
        \upperBrace{3.8}{.7}{\( h \)}
        \node at (9,0) {\( \Leftrightarrow \)};
        \node at (12,0) {\( h \)};
        \draw[dotted] (12.5,0) -- (14,0);
        \dHS{14}{0} \HLabel{14}{0}{\( k \)}
        \node at (18.5,0) {\( (k\in [0,h-1]) \)};
      \end{tikzpicture}
    \end{center} 
   \item[Case 4-2] If \( \tau^{-1}(i)<0 \), the \( i \)th step of
     \( H \) has a label \( k=\cro^-_i(\tau) \) and a mark of type $a$.
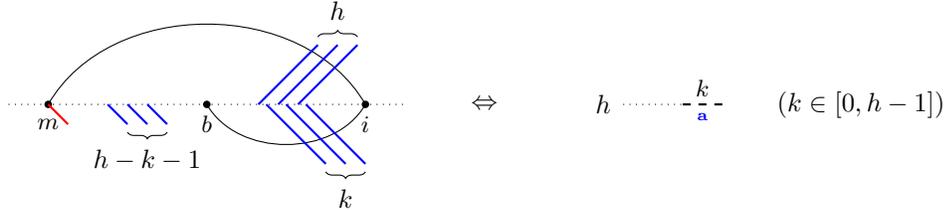
          \begin{center}
      \begin{tikzpicture}[scale=1.5]
        \draw[dotted] (-3,0) -- (7,0);
        \vertex{-2}{m}
        \vertex{2}{b} \vertex{6}{i} 
        \lowerEdge{6}{2} \upperEdge{-2}{6} 
        \lowerHalfEdge{-2} 
        \shortLowerHalfEdge{-.5} \shortLowerHalfEdge{0} \shortLowerHalfEdge{.5}
        \lowerBrace{0}{-.7}{\( h-k-1 \)}
        \lowerBrace{5}{-1.7}{\( k \)}
        \longLowerHalfEdge{3.5} \longLowerHalfEdge{4} \longLowerHalfEdge{4.5}
        \longUpperHalfEdge{3.3} \longUpperHalfEdge{3.8} \longUpperHalfEdge{4.3}
        \upperBrace{4.8}{1.7}{\( h \)}
        \node at (9,0) {\( \Leftrightarrow \)};
        \node at (12,0) {\( h \)};
        \draw[dotted] (12.5,0) -- (14,0);
        \dHS{14}{0} \HLabel{14}{0}{\( k \)} \HMark{14}{0}{a}
        \node at (18.5,0) {\( (k\in [0,h-1]) \)};
      \end{tikzpicture}
    \end{center} 
   \end{description}

 \item[Case 5] If \( |\tau(i)|,|\tau^{-1}(i)|>i \), the \( i \)th step
   of \( H \) is a down step with a pair \( (k,k') \) of labels, where
   \( k=\cro^+_i(\tau) \) and \( k'=\cro^-_i(\tau) \). In the diagram
   below, the relative positions of the vertices \( a \) and \( b \)
   are irrelevant.
   
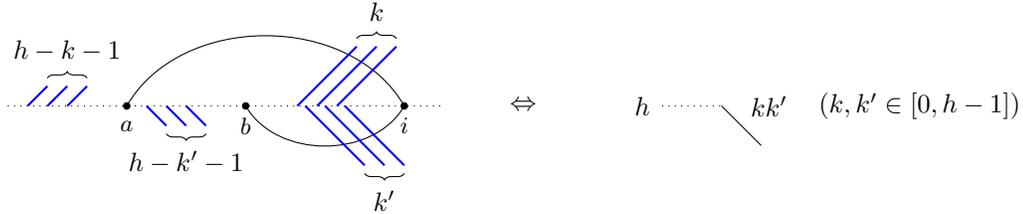
\begin{center}
      \begin{tikzpicture}[scale=1.5]
        \draw[dotted] (-4,0) -- (7,0);
        \vertex{-1}{a}
        \vertex{2}{b} \vertex{6}{i} 
        \lowerEdge{6}{2} \upperEdge{-1}{6} 
        \shortUpperHalfEdge{-3.5} \shortUpperHalfEdge{-3} \shortUpperHalfEdge{-2.5}
        \upperBrace{-3}{.7}{\( h-k-1 \)}
        \shortLowerHalfEdge{-.5} \shortLowerHalfEdge{0} \shortLowerHalfEdge{.5}
        \lowerBrace{0}{-.7}{\( h-k'-1 \)}
        \lowerBrace{5}{-1.7}{\( k' \)}
        \longLowerHalfEdge{3.5} \longLowerHalfEdge{4} \longLowerHalfEdge{4.5}
        \longUpperHalfEdge{3.3} \longUpperHalfEdge{3.8} \longUpperHalfEdge{4.3}
        \upperBrace{4.8}{1.7}{\( k \)}
        \node at (9,0) {\( \Leftrightarrow \)};
        \node at (12,0) {\( h \)};
        \draw[dotted] (12.5,0) -- (14,0);
        \DS{14}{0} \node [right] at (14.5,0) {\( kk' \)};
        \node at (19,0) {\( (k,k'\in [0,h-1]) \)};
      \end{tikzpicture}
    \end{center}
 \end{description}
\end{algorithm}

See \Cref{fig:map_SP_MLH} for an example of the map \( \phi \).

\begin{figure}
  \begin{tikzpicture}[scale=1.2]
\vertex{1}{1} \vertex{2}{2} \vertex{3}{3} \vertex{4}{4} \vertex{5}{5} \vertex{6}{6} \vertex{7}{7} \vertex{8}{8} \vertex{9}{9} \vertex{0}{-1} \vertex{-1}{-2} \vertex{-2}{-3} \vertex{-3}{-4} \vertex{-4}{-5} \lowerHalfEdge{1} \lowerEdge{1}{0} \upperHalfEdge{0} 
\begin{scope}[shift={(11,-3/2)}]
\Grid{9}{3}\US{0}{0} \UMark{0}{0}{a} 
\end{scope}
\end{tikzpicture}

\begin{tikzpicture}[scale=1.2]
\vertex{1}{1} \vertex{2}{2} \vertex{3}{3} \vertex{4}{4} \vertex{5}{5} \vertex{6}{6} \vertex{7}{7} \vertex{8}{8} \vertex{9}{9} \vertex{0}{-1} \vertex{-1}{-2} \vertex{-2}{-3} \vertex{-3}{-4} \vertex{-4}{-5} \lowerHalfEdge{1} \lowerEdge{1}{0} \upperHalfEdge{0} \lowerEdge{2}{-1} \upperHalfEdge{-1} \upperEdge{-2}{2} \lowerHalfEdge{-2} 
\begin{scope}[shift={(11,-3/2)}]
\Grid{9}{3}\US{0}{0} \UMark{0}{0}{a} \US{1}{1} \UMark{1}{1}{ad} 
\end{scope}
\end{tikzpicture}

\begin{tikzpicture}[scale=1.2]
\vertex{1}{1} \vertex{2}{2} \vertex{3}{3} \vertex{4}{4} \vertex{5}{5} \vertex{6}{6} \vertex{7}{7} \vertex{8}{8} \vertex{9}{9} \vertex{0}{-1} \vertex{-1}{-2} \vertex{-2}{-3} \vertex{-3}{-4} \vertex{-4}{-5} \lowerHalfEdge{1} \lowerHalfEdge{3} \lowerEdge{1}{0} \upperHalfEdge{0} \lowerEdge{2}{-1} \upperHalfEdge{-1} \lowerEdge{3}{-2} \upperEdge{-2}{2} 
\begin{scope}[shift={(11,-3/2)}]
\Grid{9}{3}\US{0}{0} \UMark{0}{0}{a} \US{1}{1} \UMark{1}{1}{ad} \dHS{2}{2} \HLabel{2}{2}{1} \HMark{2}{2}{a} 
\end{scope}
\end{tikzpicture}

\begin{tikzpicture}[scale=1.2]
\vertex{1}{1} \vertex{2}{2} \vertex{3}{3} \vertex{4}{4} \vertex{5}{5} \vertex{6}{6} \vertex{7}{7} \vertex{8}{8} \vertex{9}{9} \vertex{0}{-1} \vertex{-1}{-2} \vertex{-2}{-3} \vertex{-3}{-4} \vertex{-4}{-5} \upperLoop{4} \lowerHalfEdge{1} \lowerHalfEdge{3} \lowerEdge{1}{0} \upperHalfEdge{0} \lowerEdge{2}{-1} \upperHalfEdge{-1} \lowerEdge{3}{-2} \upperEdge{-2}{2} 
\begin{scope}[shift={(11,-3/2)}]
\Grid{9}{3}\US{0}{0} \UMark{0}{0}{a} \US{1}{1} \UMark{1}{1}{ad} \dHS{2}{2} \HLabel{2}{2}{1} \HMark{2}{2}{a} \HS{3}{2} \HLabel{3}{2}{0} 
\end{scope}
\end{tikzpicture}

\begin{tikzpicture}[scale=1.2]
\vertex{1}{1} \vertex{2}{2} \vertex{3}{3} \vertex{4}{4} \vertex{5}{5} \vertex{6}{6} \vertex{7}{7} \vertex{8}{8} \vertex{9}{9} \vertex{0}{-1} \vertex{-1}{-2} \vertex{-2}{-3} \vertex{-3}{-4} \vertex{-4}{-5} \upperLoop{4}{4} \upperHalfEdge{5} \lowerHalfEdge{1} \lowerHalfEdge{3} \lowerEdge{1}{0} \upperHalfEdge{0} \lowerEdge{2}{-1} \upperHalfEdge{-1} \lowerEdge{3}{-2} \upperEdge{-2}{2} \upperEdge{-3}{5} \lowerHalfEdge{-3} 
\begin{scope}[shift={(11,-3/2)}]
\Grid{9}{3}\US{0}{0} \UMark{0}{0}{a} \US{1}{1} \UMark{1}{1}{ad} \dHS{2}{2} \HLabel{2}{2}{1} \HMark{2}{2}{a} \HS{3}{2} \HLabel{3}{2}{0} \US{4}{2} \UMark{4}{2}{d} 
\end{scope}
\end{tikzpicture}

\begin{tikzpicture}[scale=1.2]
\vertex{1}{1} \vertex{2}{2} \vertex{3}{3} \vertex{4}{4} \vertex{5}{5} \vertex{6}{6} \vertex{7}{7} \vertex{8}{8} \vertex{9}{9} \vertex{0}{-1} \vertex{-1}{-2} \vertex{-2}{-3} \vertex{-3}{-4} \vertex{-4}{-5} \upperLoop{4}{4} \upperHalfEdge{5} \lowerHalfEdge{1} \lowerHalfEdge{3} \lowerEdge{1}{0} \upperHalfEdge{0} \lowerEdge{2}{-1} \upperHalfEdge{-1} \lowerEdge{3}{-2} \upperEdge{-2}{2} \lowerEdge{6}{-4} \upperEdge{-4}{6} \upperEdge{-3}{5} \lowerHalfEdge{-3} 
\begin{scope}[shift={(11,-3/2)}]
\Grid{9}{3}\US{0}{0} \UMark{0}{0}{a} \US{1}{1} \UMark{1}{1}{ad} \dHS{2}{2} \HLabel{2}{2}{1} \HMark{2}{2}{a} \HS{3}{2} \HLabel{3}{2}{0} \US{4}{2} \UMark{4}{2}{d} \HS{5}{3} \HLabel{5}{3}{3} \HMark{5}{3}{a} 
\end{scope}
\end{tikzpicture}

\begin{tikzpicture}[scale=1.2]
\vertex{1}{1} \vertex{2}{2} \vertex{3}{3} \vertex{4}{4} \vertex{5}{5} \vertex{6}{6} \vertex{7}{7} \vertex{8}{8} \vertex{9}{9} \vertex{0}{-1} \vertex{-1}{-2} \vertex{-2}{-3} \vertex{-3}{-4} \vertex{-4}{-5} \upperLoop{4}{4} \upperEdge{5}{7} \lowerHalfEdge{1} \lowerHalfEdge{3} \lowerEdge{1}{0} \upperHalfEdge{0} \lowerEdge{2}{-1} \upperHalfEdge{-1} \lowerEdge{3}{-2} \upperEdge{-2}{2} \lowerEdge{6}{-4} \upperEdge{-4}{6} \lowerEdge{7}{-3} \upperEdge{-3}{5} 
\begin{scope}[shift={(11,-3/2)}]
\Grid{9}{3}\US{0}{0} \UMark{0}{0}{a} \US{1}{1} \UMark{1}{1}{ad} \dHS{2}{2} \HLabel{2}{2}{1} \HMark{2}{2}{a} \HS{3}{2} \HLabel{3}{2}{0} \US{4}{2} \UMark{4}{2}{d} \HS{5}{3} \HLabel{5}{3}{3} \HMark{5}{3}{a} \DS{6}{3} \DLabel{6}{3}{02} 
\end{scope}
\end{tikzpicture}

\begin{tikzpicture}[scale=1.2]
\vertex{1}{1} \vertex{2}{2} \vertex{3}{3} \vertex{4}{4} \vertex{5}{5} \vertex{6}{6} \vertex{7}{7} \vertex{8}{8} \vertex{9}{9} \vertex{0}{-1} \vertex{-1}{-2} \vertex{-2}{-3} \vertex{-3}{-4} \vertex{-4}{-5} \upperLoop{4}{4} \upperEdge{5}{7} \lowerEdge{8}{1} \lowerHalfEdge{3} \lowerEdge{1}{0} \upperHalfEdge{0} \lowerEdge{2}{-1} \upperEdge{-1}{8} \lowerEdge{3}{-2} \upperEdge{-2}{2} \lowerEdge{6}{-4} \upperEdge{-4}{6} \lowerEdge{7}{-3} \upperEdge{-3}{5} 
\begin{scope}[shift={(11,-3/2)}]
\Grid{9}{3}\US{0}{0} \UMark{0}{0}{a} \US{1}{1} \UMark{1}{1}{ad} \dHS{2}{2} \HLabel{2}{2}{1} \HMark{2}{2}{a} \HS{3}{2} \HLabel{3}{2}{0} \US{4}{2} \UMark{4}{2}{d} \HS{5}{3} \HLabel{5}{3}{3} \HMark{5}{3}{a} \DS{6}{3} \DLabel{6}{3}{02} \DS{7}{2} \DLabel{7}{2}{11} 
\end{scope}
\end{tikzpicture}

\begin{tikzpicture}[scale=1.2]
\vertex{1}{1} \vertex{2}{2} \vertex{3}{3} \vertex{4}{4} \vertex{5}{5} \vertex{6}{6} \vertex{7}{7} \vertex{8}{8} \vertex{9}{9} \vertex{0}{-1} \vertex{-1}{-2} \vertex{-2}{-3} \vertex{-3}{-4} \vertex{-4}{-5} \upperLoop{4}{4} \upperEdge{5}{7} \lowerEdge{8}{1} \lowerEdge{9}{3} \lowerEdge{1}{0} \upperEdge{0}{9} \lowerEdge{2}{-1} \upperEdge{-1}{8} \lowerEdge{3}{-2} \upperEdge{-2}{2} \lowerEdge{6}{-4} \upperEdge{-4}{6} \lowerEdge{7}{-3} \upperEdge{-3}{5} 
\begin{scope}[shift={(11,-3/2)}]
\Grid{9}{3}\US{0}{0} \UMark{0}{0}{a} \US{1}{1} \UMark{1}{1}{ad} \dHS{2}{2} \HLabel{2}{2}{1} \HMark{2}{2}{a} \HS{3}{2} \HLabel{3}{2}{0} \US{4}{2} \UMark{4}{2}{d} \HS{5}{3} \HLabel{5}{3}{3} \HMark{5}{3}{a} \DS{6}{3} \DLabel{6}{3}{02} \DS{7}{2} \DLabel{7}{2}{11} \DS{8}{1} \DLabel{8}{1}{00} 
\end{scope}
\end{tikzpicture}
\caption{A step-by-step description of the map \( \phi(\tau) \) for
  \( \tau = \bar9\,\bar8\,\bar2\,4\,7\,\bar6\,\bar5\,1\,3 \), where
  \( \tau^{-1} = 8\, \bar3\, 9\, 4\, \bar7\, \bar6\, 5\, \bar2\, \bar1 \).}
\label{fig:map_SP_MLH}
\end{figure}
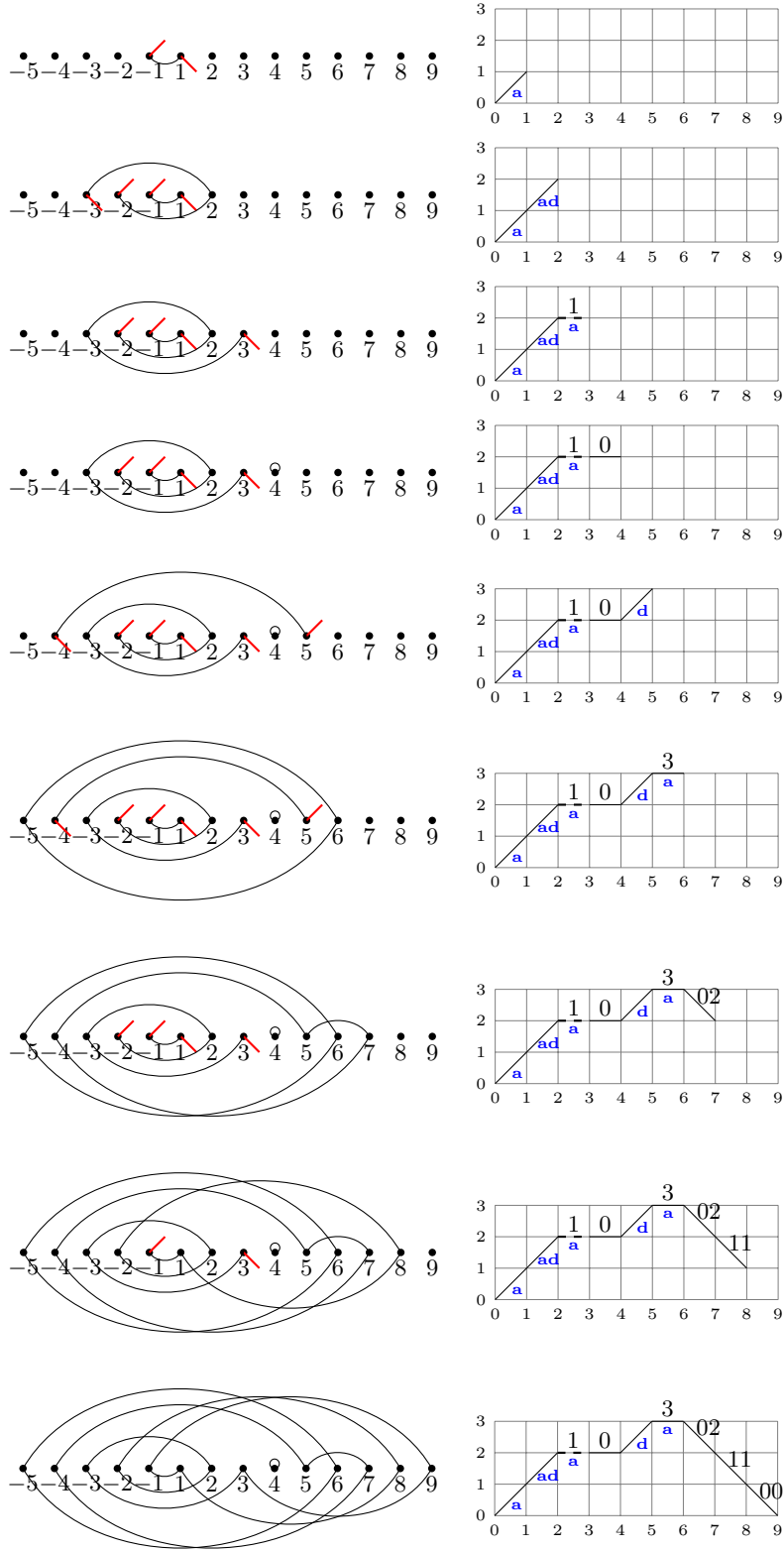

\begin{thm}\label{thm:SP-MLH}
  The map \( \phi: \SP(n,r)\to \MLH^*(n,r) \) is a bijection. Moreover,
  if \( \phi(\tau) = H \), then \( \wt^*(H) = q^{\cro(\tau)} \).
\end{thm}

\begin{proof}
  We will find the inverse \( \psi \) of \( \phi \). Let
  \( H\in \MLH(n,r) \). To find \( \psi(H)=\tau \), we construct the
  partial arc diagram \( D_i(\tau) \) for \( i=0,1,2,\dots,n \). In the
  process of the construction, an \emph{available vertex} is a vertex
  \( a<0 \) such that there are no arcs or open arcs attached to it.

  First, \( D_0(\tau) \) is the diagram consisting only of the vertices
  \( -r,\dots,-1 ,1,\dots,n \). For \( i\in [n] \), suppose that
  \( D_{i-1}(\tau) \) is given, and let \( h \) be the starting height
  of the \( i \)th step in \( H \). Then \( D_i(\tau) \) is constructed
  as follows.

\begin{description}
\item[Case 1] The \( i \)th step of \( H \) is an up step.
   \begin{description}
   \item[Case 1-1] If the \( i \)th step has no marks, draw an open
     upper arc at \( i \) and an open lower arc at \( i \).
   \item[Case 1-2] If the \( i \)th step has only a mark of type
     \( a \), draw a lower arc \( (m,i) \) and an open upper arc at
     \( m \), where \( m \) is the rightmost available vertex.
   \item[Case 1-3] If the \( i \)th step has only a mark of type
     \( d \), draw an upper arc \( (m,i) \) and an open lower arc at
     \( m \), where \( m \) is the rightmost available vertex.
   \item[Case 1-4] If the \( i \)th step has both a mark of type
     \( a \) and a mark of type $d$, draw a lower arc \( (m,i) \),
     an open upper arc at \( m \), an upper arc \( (\ell,i) \), and an
     open lower arc at \( \ell \), where \( m \) is the rightmost
     available vertex and \( \ell \) is the second rightmost available
     vertex.
   \end{description}
 \item[Case 2] The \( i \)th step of \( H \) is a solid horizontal
   step with one of the following properties.
   \begin{description}
   \item[Case 2-1] The \( i \)th step has a label \( 0 \) and no marks.
     Then draw a loop at \( i \).
   \item[Case 2-2] The \( i \)th step has a label \( h \) and a mark of type $a$.
     Then draw a lower arc \( (m,i) \) and an upper arc \( (m,i) \),
    where \( m \) is the rightmost available vertex.
   \end{description}

 \item[Case 3] The \( i \)th step of \( H \) is a solid horizontal
   step with one of the following properties.
   \begin{description}
   \item[Case 3-1] The \( i \)th step has a label \( k\in [h] \) and
     no marks. Then draw an upper arc \( (a,i) \) and an open upper
     arc at \( i \), where the vertex \( a \) is chosen (among those
     with an open upper arc in \( D_{i-1}(\tau) \)) so that
     \( \cro^+_i(\tau)=k \).
   \item[Case 3-2] The \( i \)th step of \( H \) has a label
     \( k \in [0,h-1] \) and a mark of type $a$. Then draw a
     lower arc \( (m,i) \), an open upper arc at \( m \), and an upper
     arc \( (\ell,i) \), where \( m \) is the rightmost available vertex
     and the vertex \( \ell \) is chosen so that \( \cro^+_i(\tau)=k \).
   \end{description}

 \item[Case 4] The \( i \)th step of \( H \) is a dashed horizontal
   step.
   \begin{description}
   \item[Case 4-1] The \( i \)th step has a label \( k\in[0,h-1] \)
     and no marks. Then draw a lower arc \( (b,i) \) and an open lower
     arc at \( i \), where the vertex \( b \) is chosen (among those
     with an open lower arc in \( D_{i-1}(\tau) \)) so that
     \( \cro^-_i(\tau)=k \).
   \item[Case 4-2] The \( i \)th step of \( H \) has a label
     \( k\in [0,h-1] \) and a mark of type $a$. Then draw an upper
     arc \( (m,i) \), an open lower arc at \( m \), and a lower arc
     \( (b,i) \), where \( m \) is the rightmost available vertex and
     the vertex \( b \) is chosen so that \( \cro^-_i(\tau)=k \).
   \end{description}

 \item[Case 5] The \( i \)th step of \( H \) is a down step with a
   pair \( (k,k') \) of labels. Draw an upper arc \( (a,i) \) and a
   lower arc \( (b,i) \), where the vertices \( a \) and \( b \)
   chosen so that \( \cro^+_i(\tau)=k \) and \( \cro^-_i(\tau)=k' \).
 \end{description}

 By the construction, it is straightforward to check that \( \psi \)
 is the inverse of the map \( \phi \) in \Cref{alg:SP-MLH}. For the
 second statement of the theorem, observe that for each
 \( i\in [n] \), if \( h \) is the starting height of the \( i \)th
 step of \( H \), then \( \cro_i(\tau) \) is equal to \( k+u+v \),
 where \( k \) is the sum of the labels of the \( i \)th step of
 \( H \) (if any), and
 \begin{align*}
   u &= 
   \begin{cases}
    h & \mbox{if the \( i \)th step has a mark of type $a$},\\
    0 & \mbox{otherwise,}
   \end{cases} \\
    v &= 
   \begin{cases}
    h+1 & \mbox{if the \( i \)th step has a mark of type $d$},\\
    0 & \mbox{otherwise.}
   \end{cases}
 \end{align*}
 Hence, 
 \[
   q^{\cro(\tau)} = q^{\sum_{i=1}^{n} \cro_i(\tau)} = \wt(H).
 \]
\end{proof}

Identifying $H'=\phi(\tau)\in\MLH^*(n,r)$ with its corresponding ordinary marked Laguerre history $H\in\MLH(n,r)$, we obtain the desired bijection from $\SP(n,r)$ to $\MLH(n,r)$.

\begin{cor}
There exists a bijection from $\SP(n,r)$ to $\MLH(n,r)$ such that $\cro(\tau)=\wt(H)$, where $H$ is the marked Laguerre history corresponding to $\tau\in\SP(n,r)$.
\end{cor}

\subsection{A connection with assembl\'{e}es}

The set $\AS(n,r)$ of signed permutations arising from assembl\'{e}es
is contained in $\SP(n,r)$. Moreover, for $\nu\in\SP(n,r)$, the
entries in $\neg(\nu)$ can appear in any order, so that
$|\SP(n,r)|=r!|\AS(n,r)|$. Thus it is natural to decompose $\SP(n,r)$
into an element of $\AS(n,r)$ together with a permutation of its
negative entries. In this subsection, we define a bijection
\[\rho:\SP(n,r)\to \AS(n,r)\times S_r
\] 
with the property that if \( \rho(\nu)=(\sigma,\tau)\) for $\nu\in\SP(n,r) $, then
\( \cro(\nu) = \cro(\tau) + \inv(\sigma) \).

For \( \nu\in \SP(n,r) \), let \(a_1 < \cdots < a_r \) be the elements of $\neg(\nu)$, and let \( b_1 < \cdots < b_r \) be the elements of $[n]$ such that \( \nu^{-1}(b_i)\in\neg(\nu) \). There is a unique permutation \( \sigma\in S_r \) such that
\( \nu(a_j) = - b_{\sigma(j)} \) for all \( j\in [r] \). Define 
\( \tau\in \AS(n,r) \) by
\[
  \tau(i) =
  \begin{cases}
   -b_j & \mbox{if \( i=a_j \) for some \( j\in [r] \)},\\
   \nu(i) & \mbox{otherwise}.
  \end{cases}
\]
Note that $\tau$ is well-defined since by construction, $\tau(a_1)<\cdots<\tau(a_r)$, where $1=a_1<\cdots<a_r$ are the elements in $\neg(\tau)$. We then define \( \rho(\nu) = (\tau,\sigma) \in \AS(n,r)\times S_r \). 
\begin{thm}\label{thm:SP-A*S}
  The map \( \rho:\SP(n,r) \to \AS(n,r)\times S_r \) is a bijection
  such that if \( \rho(\nu) = (\tau,\sigma) \), then
  \( \cro(\nu) = \cro(\tau) + \inv(\sigma) \).
\end{thm}

\begin{proof}
  Let \( (\tau,\sigma)\in \AS(n,r)\times S_r \) and suppose that
  \( a_1 < \cdots < a_r \) and \( b_1 < \cdots < b_r \) are the
  integers such that \( \tau(a_i)\in\neg(\tau) \) and \( \tau^{-1}(b_i)\in\neg(\tau) \). Let $\sigma\in S_r$ be the unique permutation satisfying $\tau(a_{\sigma(j)})=-b_j$.
  Then define \( \nu\in \SP(n,r) \) by
\[
  \nu(i) =
  \begin{cases}
   -b_{\sigma(j)} & \mbox{if \( i=a_j \) for some \( j\in [r] \)},\\
   \tau(i) & \mbox{otherwise}.
  \end{cases}
\]
By the construction, the map \( (\tau,\sigma)\mapsto \nu \) is the
inverse of \( \rho \). Hence, \( \rho \) is a bijection.

It remains to prove that \( \cro(\nu) = \cro(\tau) + \inv(\sigma) \).
We say that a crossing \( ((a,b),(c,d)) \) of two upper arcs is of
\emph{type \( U_{++} \)} (resp.~\( U_{-+} \) and \( U_{--} \))
if \( a,c>0 \) (resp.~\( a<0<c \) and \( a,c<0 \)). Similarly,
a crossing \( ((a,b),(c,d)) \) of two lower arcs is of \emph{type
  \( L_{++} \)} (resp.~\( L_{-+} \) and \( L_{--} \)) if
\( a,c>0 \) (resp.~\( a<0<c \) and \( a,c<0 \)). We denote by
\( U_{++}(\nu) \) the number of crossings of type \( U_{++} \) in
\( \nu \). We define \( U_{-+}(\nu) \), \( U_{--}(\nu) \),
\( L_{++}(\nu) \), \( L_{-+}(\nu) \), and \( L_{--}(\nu) \) similarly.
Then
\begin{align*}
  \cro(\nu) &= U_{++}(\nu) + U_{-+}(\nu) + U_{--}(\nu) +L_{++}(\nu) + L_{-+}(\nu) + L_{--}(\nu),\\
  \cro(\tau) &= U_{++}(\tau) + U_{-+}(\tau) + U_{--}(\tau) +L_{++}(\tau) + L_{-+}(\tau) + L_{--}(\tau).
\end{align*}
Since \( U_{--}(\tau) = L_{--}(\tau) = 0 \), it suffices to show the following identities:
\begin{align}
  \label{eq:U++}
  U_{++}(\nu) + L_{++}(\nu) &= U_{++}(\tau) + L_{++}(\tau),\\
  \label{eq:U-+}
  U_{-+}(\nu) + L_{-+}(\nu) &= U_{-+}(\tau) +  L_{-+}(\tau),\\
  \label{eq:U--}
  U_{--}(\nu) + L_{--}(\nu) &= \inv(\sigma).
\end{align}

The first identity \eqref{eq:U++} follows from the fact that \( \nu \)
and \( \tau \) have the same upper and lower arcs \( (a,b) \) with
\( a>0 \).

For the second identity \eqref{eq:U-+}, observe that if
\( ((a,b),(c,d)) \) is a crossing of type \( U_{-+} \) in \( \nu \),
then \( b=b_j \) for some \( j \), and \( a \) is uniquely determined
by \( b_j \). Thus, \( U_{-+}(\nu) \) is the number of triples
\( (b_j,c,d) \), where \( 0<c\le b_j<d \) and \( (c,d) \) is an upper
arc in \( \nu \). Since the same is true for \( \tau \), we have
\( U_{-+}(\nu) = U_{-+}(\tau) \), and similarly,
\( L_{-+}(\nu) = L_{-+}(\tau) \). This shows \eqref{eq:U-+}.

Now we prove \eqref{eq:U--}. Let \( u_1,\dots,u_r \) be the vertices
such that \( (u_i,a_i) \) is a lower arc in \( \nu \) for each
\( i\in [r] \). Observe that \( (u_i,b_{\sigma(i)}) \) is an upper arc
in \( \nu \) since \( \nu(a_i) = -b_{\sigma(i)} \).
Thus,
\[
    L_{--}(\nu) = \sum_{1\le i<j\le r} \chi(u_i<u_j<a_i<a_j),
\]
where \( \chi(P) \) is \( 1 \) if \( P \) is true and \( 0 \) otherwise.
Since there are two cases \( b_{\sigma(i)}<b_{\sigma(j)} \)
and \( b_{\sigma(j)}<b_{\sigma(i)} \), we have
\[
  U_{--}(\nu) = \sum_{1\le i<j\le r} \chi(u_i<u_j<b_{\sigma(i)}<b_{\sigma(j)})
  + \sum_{1\le i<j\le r} \chi(u_j<u_i<b_{\sigma(j)}<b_{\sigma(i)}).
\]
Since \( \inv(\sigma) = \sum_{1\le i<j\le r} \chi(\sigma(j)<\sigma(i))  \),
it suffices to show that for all \( 1\le i<j\le r \),
\begin{multline}\label{eq:4}
  \chi(\sigma(j)<\sigma(i)) \\
  = \chi(u_i<u_j<a_i<a_j) + \chi(u_i<u_j<b_{\sigma(i)}<b_{\sigma(j)}) + \chi(u_j<u_i<b_{\sigma(j)}<b_{\sigma(i)}).
\end{multline}
If \( \sigma(i)<\sigma(j) \), then \( b_{\sigma(i)}<b_{\sigma(j)} \), and by
the definition of the order \( <_\nu \), we have \( a_i<_\nu a_j \).
Therefore, \( u_j<u_i \), and both sides of \eqref{eq:4} are zero. If
\( \sigma(j)<\sigma(i) \), then \( b_{\sigma(j)}<b_{\sigma(i)} \). In this
case, \eqref{eq:4} is equivalent to
\( 1 = \chi(u_i<u_j) + \chi(u_j<u_i) \), which is implied by $u_i\neq u_j$. This
completes the proof.
\end{proof}

The following answers a question in \cite[p.~15]{Corteel2020a}.

\begin{cor}\label{cor:2}
  We have
  \[
    \sum_{H\in \MLH(n,r)}  \wt(H) = \sum_{\nu\in \SP(n,r)}  q^{\cro(\nu)} = [r]_q! \sum_{\tau\in \AS(n,r)} q^{\cro(\tau)}.
  \]
\end{cor}

\begin{proof}
  This follows from \Cref{thm:SP-MLH,thm:SP-A*S} and the well-known
  fact that \( \sum_{\sigma\in S_r} q^{\inv(\sigma)} = [r]_q! \).
\end{proof}

We note that Corteel and Nunge \cite[Theorem~22]{Corteel2020a} show
that
\[
    \sum_{H\in \MLH(n,r)}  \wt(H) = \sum_{\nu\in \SP(n,r)}  q^{\mathrm{tw}(\nu)},
\]
where \( \mathrm{tw}(\nu) \) is
the number of 31-2 and 31-\( \overline{2} \) patterns in the signed permutation $\nu$. Furthermore, Iraci, Nadeau, and Vanden Wyngaerd
\cite{Iraci2024} studied 2-31 patterns of signed permutations, which they called segmented permutations.
It would be interesting to make an explicit connection between these signed permutation patterns and statistics on RAT.

\bibliographystyle{abbrv}

\end{document}